\documentclass[preprint]{imsart}

\usepackage{amsfonts,amsmath,amssymb,amsthm,bbm}
\usepackage[font=footnotesize,labelfont=bf]{caption}
\usepackage{subcaption}
\usepackage{float,mathtools}
\usepackage[shortlabels]{enumitem}
\usepackage{graphicx}
\usepackage{algorithmicx}
\usepackage[ruled]{algorithm}
\usepackage{algpseudocode}
\usepackage[numbers]{natbib}
\usepackage[colorlinks]{hyperref}
\usepackage{chngcntr}
\usepackage{relsize} 
\usepackage{apptools}
\AtAppendix{\counterwithin{theorem}{subsubsection}}        % bibliography style

\theoremstyle{plain}
\newenvironment{proof*}{\paragraph{Proof}}{\hfill$\blacksquare$}

\newtheorem{remark}{Remark}
\newtheorem{theorem}{Theorem}
\newtheorem{lemma}{Lemma}
\newtheorem{proposition}{Proposition}

\newtheorem{definition}{Definition}
\newtheorem{example}{Example}
\AtAppendix{\counterwithin{lemma}{subsection}}
\AtAppendix{\counterwithin{theorem}{subsection}}
\AtAppendix{\counterwithin{proposition}{subsection}}
\AtAppendix{\counterwithin{corollary}{subsection}}

\newcommand{\norm}[1]{\left\lVert#1\right\rVert}

\def\cI{\mathbbm{1}}

\arxiv{arXiv:0000.0000}

\begin{document}

\begin{frontmatter}
\title{Sample Splitting and Weak Assumption Inference For Time Series}

\runauthor{R.Lunde}
\runtitle{Sample Splitting For Time Series}
%\thankstext{T1}{Footnote to the title with the ``thankstext'' command.}

\begin{aug}
\author{\fnms{Robert} \snm{Lunde}\ead[label=e1]{rlunde@utexas.edu}}
% \and
% \author{\fnms{Third} \snm{Author}\thanksref{t1,m2}
% \ead[label=e3]{third@somewhere.com}
% \ead[label=u1,url]{http://www.foo.com}}

% \thankstext{t1}{Some comment}
% \thankstext{t2}{First supporter of the project}
% \thankstext{t3}{Second supporter of the project}
% \runauthor{F. Author et al.}

\affiliation{The University of Texas at Austin}

\address{Department of Statistics and Data Sciences\\
The University of Texas at Austin\\
Austin, TX 78712 \\
\printead{e1}\\
\phantom{E-mail:\ }}

% \address{Address of the Third author\\
% Usually a few lines long\\
% Usually a few lines long\\
% \printead{e3}\\
% \printead{u1}}
\end{aug}

\begin{abstract}

  We consider the problem of inference after model selection under weak assumptions in the time series setting. Even when the data are not independent, we show that sample splitting remains asymptotically valid as long as the process satisfies appropriate weak dependence conditions and the functional of interest is suitably well-behaved.  In addition, if the inference targets are appropriately defined, we demonstrate that valid statistical inference is possible without assuming stationarity.  As a working example, we consider post-selection inference for regression coefficients under a random design assumption, in which the pair $(Y_i, X_i) \in \mathbb{R}^{p_n}$ is assumed to be an observation from a weakly dependent triangular array. We establish (asymptotic) sample splitting validity for regression coefficients under both $\beta$-mixing and $\tau$-dependence assumptions.   

  To facilitate statistical inference in the non-stationary, weakly dependent regime, we extend a central limit theorem of \citet{doukan-winteberger-invariance-principle}.  To extend their result, we derive some properties of the variance of a normalized sum of a weakly dependent process. In particular, we show that, under very general conditions, the variance is often well-approximated by independent blocks.  Using this result, we derive the validity of the block multiplier bootstrap under $\theta$-dependence and demonstrate the validity of an inference procedure that combines sample splitting with the bootstrap under weak assumptions.      
\end{abstract}

% \begin{keyword}[class=MSC]
% \kwd[Primary ]{60K35}
% \kwd{60K35}
% \kwd[; secondary ]{60K35}
% \end{keyword}


\end{frontmatter}

\section{Introduction}
\label{Introduction}
We consider the problem of performing inference after model selection in a weak-assumption setting for time series data. Post-selection inference is currently a very active area of research with a large body of literature.  We will discuss some of the literature most closely related to our work in Section \ref{related-work}.  For a more comprehensive literature review, we refer the reader to \citet{Dezuere-high-dim-inference-review} or \citet{bootstrapping-sample-splitting-assumption-free}.   

Sample splitting is an old method for inference after model selection. The procedure is very simple, and \citet{bootstrapping-sample-splitting-assumption-free} show that sample splitting achieves honest coverage in the independent, low assumption setting.  However, one common criticism of sample splitting is that it requires the selection and inference sets to be independent, which is not the case for time series data. 

In this paper, we study the theoretical properties of a sample splitting procedure for time series in which the data are divided into two contiguous blocks, where the first half is used for selection and the second half for inference.  One of the main ideas of our paper is that, although the datasets are not independent, sample splitting remains \textit{asymptotically valid} for a wide range of functionals under appropriate weak-dependence conditions.  To demonstrate the applicability of our results, as a running example, we consider post-selection inference for regression coefficients under a random design assumption, where the pair $(Y_i, X_i) \in \mathbb{R}^{p_n}$ is assumed to be an observation from a weakly dependent triangular array.  We also study the properties of a block multiplier procedure on the inference set, and establish its validity under weak assumptions on the data generating process.    

\subsection{Problem Setup}
Let $\{Y_{n,1}, \ldots, Y_{n, 2n}\}$ be  a triangular array of random vectors taking values in $\mathbb{R}^{p_n}$. We will allow $p_n$ to be an arbitrary nondecreasing sequence of natural numbers with some caveats to be explained later, but we will focus on low-dimensional asymptotics for a selected model.  For ease of exposition, we will suppress the triangular array notation when appropriate and write $Y_1, Y_2, \ldots, Y_{2n}$. The main example we have in mind is regression, where $Y_{n,i,1}$ is the response and $(Y_{n,i,2}, \ldots Y_{n,i,p})$ are the covariates. 

For concreteness, suppose that $\mathcal{D}_1 = (Y_1, \ldots Y_n)$ is used for model selection and $\mathcal{D}_2 = (Y_{n+1}, \ldots , Y_{2n} )$ is used for inference.  Let $\mathcal{M}$ represent a collection of possible models.  For each $m \in \mathcal{M}$, let $\hat{\theta}_m (Y_{n+1}^{2n}) : \mathcal{Y}^{n \times p^{(m)}} \mapsto \Theta_{m}$ denote an estimator formed from the second half of the data for the model $m$, where $\Theta_{m}$ is a Polish space that may depend on the value of $m$.  Let $p^{(m)}$ denote the dimension of the selected model, which is typically much smaller than $n$. For example, in a standard regression problem, $\hat{\theta}_m (Y_{n+1}^{2n})$ represents a subset of regression coefficients estimated from some sub-model $m$. 

The goal of many approaches to post-selection inference is to provide valid inference for a target parameter given that a particular model was chosen, while remaining agnostic about whether the model is ``correct''.  We will also adopt this perspective in this paper.  When $Y_1, \ldots, Y_{2n}$ are independent, it is very straightforward to see that sample splitting achieves this goal.  First note that, with sample splitting, we hold out $\mathcal{D}_1$ from the estimation step and treat it as fixed; see \citet{fithian-optimal-inference-after-selection} for a similar treatment. Suppose that the model selection procedure $\hat{m}(Y_1^n)$ is  $\sigma(\mathcal{D}_1)$-measurable  and let $\hat{\theta}_{\hat{m}} (Y_{n+1}^{2n})$ denote an estimator formed on the second half of the data for the random model $\hat{m}(Y_1^n)$.  The resulting distribution of the estimator is given by:
\begin{align}
\mathcal{L}\bigl(\hat{\theta}_{\hat{m}} (Y_{n+1}^{2n}) \ \bigr| \ \mathcal{D}_1 \bigr) = \mathcal{L}\bigl(\hat{\theta}_{m} (Y_{n+1}^{2n})\bigr) 
\end{align}   
In other words, sample splitting works in the independent case because the conditional distribution is equal to an unconditional distribution on $\mathcal{D}_2$, which is often very easy to work with in practice.  However, when $Y_1, \ldots, Y_{2n}$ are dependent, it is generally the case that:
\begin{align}
\mathcal{L}\bigl(\hat{\theta}_{\hat{m}} (Y_{n+1}^{2n}) \ | \ \mathcal{D}_1 \bigr) \neq \mathcal{L}\bigl(\hat{\theta}_{m} (Y_{n+1}^{2n})\bigr) 
\end{align}   
Therefore, just from basic properties of conditional probability, we cannot say whether treating $\mathcal{D}_1$ as fixed leads to valid inference; there may be spillover effects from $\mathcal{D}_1$ that invalidate inference.  If the data are weakly dependent, however, it will often be the case that the two distributions will be close asymptotically under an appropriate notion of distance.

We will now define a notion of sample-splitting validity; one of the major aims of the rest of this paper is to establish this form of validity under general conditions on dependence and the estimator.  Since sample splitting is valid in the independent case, our notion of validity will be tied to the conditional distribution approaching a distribution in which the inference and selection sets are independent.  To this end, let $\widetilde{Y}_{n+1}^{2n}$ represent an independent copy of $Y_{n+1}^{2n}$ that is also independent of $Y_{1}^{n}$.  The notion we adopt is as follows:
\begin{definition}[Sample Splitting Validity]
Sample splitting is said to be asymptotically valid in probability under the metric $d$ if:
\footnote{Although this level of generality is not needed in the body of the paper, we would like to note that, when the dimension of the functional $p^{(m)}$ varies across models, it is understood that $d$ is a ``meta-metric'' in the following sense.  Let $\mathcal{P}^{{(m)}}$ denote the space of Borel measures defined on the metric space $(\mathcal{X}^{(m)},\rho^{(m)})$.  Then, $d: \cup_{m \in \mathcal{M}} \mathcal{P}^{(m)} \times \mathcal{P}^{(m)} \mapsto [0,\infty]$, where $d(\cdot, \cdot)$ agrees with a metric on $\mathcal{P}^{{(m)}}$ for each $m$.} 
\begin{align}
\label{sample splitting validity in probability}
d\bigl(\mathcal{L}\bigl(\hat{\theta}_{\hat{m}}(Y_{n+1}^{2n}) \ | \ Y_1^n\bigr), \ \mathcal{L}\bigl(\hat{\theta}_{\hat{m}}(\widetilde{Y}_{n+1}^{2n}) \ | \  Y_1^n\bigr)\bigr) \xrightarrow{P} 0   
\end{align}
\end{definition}

A requirement for $d$ to be useful in our context is that, when $\{X_n \}_{n \geq 1}$ is a sequence of random variables taking values in $\mathbb{R}^p$, we have that:
\begin{align}
d\bigl(\mathcal{L}(X_n),\ \mathcal{L}(X)  \bigr) \rightarrow 0 \implies  X_n \rightsquigarrow X 
\end{align}  
That is, convergence in the metric implies convergence in distribution.  For technical reasons that will be expounded upon later,  we will use the bounded Lipschitz metric, denoted $d_{BL}$, which is known to metricize weak convergence for Borel measures on separable spaces.  See Definition \ref{bounded-lipschitz-definition} for a definition of $d_{BL}$.     

Although we will be working with a metric that metricizes weak convergence, we will not require that the estimator converges in distribution.  It may be the case that a conditional distribution approaches an unconditional distribution even if the latter fails to converge in distribution.  For details, see Section \ref{general-asymptotic}.     

Another assumption that will be avoided is stationarity.  Strict stationarity guarantees that the future will resemble the past. On the surface, this condition seems vital to statistical inference.  Without strict stationarity, there may not be a fixed parameter corresponding to the long-run variance, for example. However, strict stationarity is a strong assumption that is rarely met in practice.  While assuming weak stationarity would be enough to ensure that certain parameters, such as the long-run variance, exist, this is still an assumption that is not to be taken lightly. 

Since stationarity assumptions rarely hold, we believe that it is preferable to avoid them altogether at the cost of a loss in interpretability.  Here, we will consider parameters of the form $\theta_{\hat{m},n}$, which are random and depend on $n$, with no guarantee that they will converge even when the model is fixed.  For a concrete example in the regression setting, see Section  \ref{regression-setting}.  We will look to construct confidence intervals that have the property: 
\begin{align}
\label{confidence-statement}
\liminf_{n \rightarrow \infty} P \left(\theta_{\hat{m},n} \in C_{\hat{m},n} \right) \geq 1 - \alpha
\end{align}  
where $C_{\hat{m},n}$ is a confidence set that depends on the chosen model $\hat{m}(Y_1^n)$ and $Y_{n+1}^{2n}$.  The probability statement is over the joint distribution of $Y_1^{2n}$.  Since our notion of sample splitting validity only holds in probability, we will not be able to guarantee pathwise validity in which the probability statement is with respect to $\mathcal{L}(Y_{n+1}^{2n} \ | \ Y_1^{n})$.  Instead, what we have above is an on-average statement that nevertheless should be strong enough for statistical applications.  We will not consider honest confidence intervals (e.g. \citet{li-honest-confidence}) in the low dimensional setting and will leave a treatment of this topic to future work.  

We would also like to take a moment to discuss the interpretation of the probability statement in (\ref{confidence-statement}).  This statement differs from typical confidence intervals in the sense that the parameter changes with $n$ and depends on the data through the selected model. However, it still permits a frequentist interpretation.  Consider a large number of (independent) runs of the process.  For large $n$, the proportion of runs for which $\theta_{\hat{m},n}$ is contained in the confidence interval will be approximately $1-\alpha$. This confidence statement is retrospective and is not concerned with future values of the process. 

One may doubt the wisdom of sample splitting when the data are non-stationary.  While it may be the case that a model selected on the first half of the data will be a poor choice for the second half of the data under non-stationarity, even in these cases, sample splitting will often remain asymptotically valid in the sense that associated confidence intervals will have asymptotic $1-\alpha$ coverage as defined in (\ref{confidence-statement}).  We would like to emphasize that validity of inference is not tied to the fitness of the chosen model. In addition, we are not advocating for sample splitting as a universal inference tool for non-stationary problems.  Rather, inferences based on sample splitting are robust to departures from stationarity under appropriate conditions on dependence.     

We would also like to note that, while demonstrating (\ref{sample splitting validity in probability}) establishes sample splitting as an asymptotically valid procedure, it says nothing about how to conduct inference on $\mathcal{D}_2$.  Here there are many possibilities, but we will mainly consider the case where confidence intervals are constructed using the block multiplier bootstrap, similar to \citet{bootstrapping-sample-splitting-assumption-free}.  We will treat sample splitting together with the bootstrap as a single procedure and show that this procedure is also robust to departures from stationarity and remains valid even under dependence conditions that are weaker than mixing. A key step towards establishing validity of the block multiplier bootstrap is proving a non-stationary central limit theorem in this regime.

We are not the first to propose a resampling procedure robust to non-stationarity.  For example, \citet{nonstationary-subsampling} establish the validity of subsampling under $\alpha$-mixing and some additional regularity conditions.  Our results are under a weaker notion of dependence and does not require the existence of some fixed parameter.  While subsampling is typically valid for a wider range of functionals than the bootstrap, the bootstrap has the advantage of generalizing to growing dimensions, which we will examine in future work. 

\subsection{Notions of Weak Dependence} 
We will mainly consider two forms of dependence: $\beta$-mixing, due to \citet{Volkonskii-rozanov-introduce-beta-mixing} and $\mathlarger{\tau}$-dependence, introduced by \citet{tau-dependence}. For unconditional results, we will also consider $\theta$-weak dependence, which we will see is weaker than $\tau$-dependence. Since the dimension of the process is allowed to grow, a finite-dimensional notion of dependence will be appropriate. While such notions are less common in the literature, they have been considered in the case of $\beta$-mixing in various contexts \citep{agrawal-duchi-generalization} \citep{cck-testing-moment-inequalities} \citep{mcdonald-shalizi-schervish-estimating-beta-mixing}.
%  As we will see, both notions of dependence can be expressed as an on-average integral probability metric (IPM) between a conditional distribution and an unconditional distribution.  Recall that an IPM has the form:
% \begin{align}
% d(P,Q) = \sup_{f \in \mathcal{F}} \left| \int  f dP - \int f dQ \right| 
% \end{align} 
% The size of the function class $\mathcal{F}$ determines how well the IPM can distinguish between $P$ and $Q$ and is also related to various preservation properties associated with the metric.  However, a larger function class will result in a stronger notion of distance.    

\subsubsection{$\beta$-Mixing}
We will start by introducing the notion of $\beta$-dependence between two $\sigma$-fields.  The $\beta$-mixing coefficient will then be constructed by considering appropriate $\sigma$-fields related to the random vector $Y_1^{2n}$.   
\begin{definition}[$\beta$-dependence between $\sigma$-fields]
 Let $(\Omega, \mathcal{F}, P)$ be a probability space and for any $\mathcal{A}, \mathcal{B}, \subseteq \mathcal{F}$ define:
\begin{align}
\label{beta-sigma-field}
\begin{split}
\beta(\mathcal{A}, \mathcal{B}) = \frac{1}{2} \sup & \biggl\{ \ \sum_{i=1}^I \sum_{j=1}^J \left| P(A_i \cap B_j) - P(A_i) P(B_j)  \right| \biggl. : 
  \\  &  \{A_i\} \text{ is any finite partition of } \Omega \text{ in } \mathcal{A} 
  \\  & \biggl. \{B_j\} \text{ is any finite partition of } \Omega \text{ in } \mathcal{B}   \  \biggr\}
  \end{split}  
\end{align}
\end{definition}      

% \begin{align}
% \label{alternative-expressions-beta-sigma-field}
% \beta(\mathcal{A}, \mathcal{B}) &= d_{TV}(P^{\mathcal{A} \otimes \mathcal{B}},P^{\mathcal{A} \times \mathcal{B}} )
 % \\ &= \sup_{C \in \mathcal{A} \otimes \mathcal{B} } \left| P^{\mathcal{A} \otimes \mathcal{B}} (C) - P^{\mathcal{A} \times \mathcal{B}} (C) \right|
% \\ &= E \sup_{B \in \mathcal{B}} \left|  P(B \ | \ \mathcal{A}) - P(B) \right|
% \\ &= E \ d_{TV} ( P^{\mathcal{A} \otimes \mathcal{B}}( \ \cdot \ | \ \mathcal{B}), P^{\mathcal{A}}(\cdot) \ ) 
% \end{align}
% It turns out that we may alternatively express $\beta$-dependence as follows:
% \begin{align}
% \label{alternative-expressions-beta-sigma-field}
% \begin{split}
% \beta(\mathcal{A}, \mathcal{B}) &= \mathbb{E} \sup_{B \in \mathcal{B}} \left|  P(B \ | \ \mathcal{A}) - P(B) \right|
% \\ &= \mathbb{E} \ d_{TV} ( P^{\mathcal{A} \otimes \mathcal{B}}( \ \cdot \ | \ \mathcal{A}), P^{\mathcal{B}}(\cdot) \ ) 
% \end{split}
% \end{align}

% \begin{definition}[$\beta$-dependence between $\sigma$-fields]
%  Let $(\Omega, \mathcal{F}, P)$ be a probability space and for any $\mathcal{A}, \mathcal{B}, \subseteq \mathcal{F}$ define:
% \begin{align}
% \label{beta-sigma-field}
% \begin{split}
% \beta(\mathcal{A}, \mathcal{B}) = \frac{1}{2} \sup & \biggl\{ \ \sum_{i=1}^I \sum_{j=1}^J \left| P(A_i \cap B_j) - P(A_i) P(B_j)  \right| \biggl. : 
%   \\  &  \{A_i\} \text{ is any finite partition of } \Omega \text{ in } \mathcal{A} 
%   \\  & \biggl. \{B_j\} \text{ is any finite partition of } \Omega \text{ in } \mathcal{B}   \  \biggr\}
%   \end{split}  
% \end{align}
% \end{definition} 
Now, define the corresponding $\beta$-mixing coefficient for the vector $Y_1^{2n}$ as follows:
% Since we are dealing with a triangular array, it will be more natural to define the $\beta$-mixing coefficient for a vector opposed to a process. For this section, we adopt the following notion:            

\begin{definition}[$\beta$-mixing coefficient]
Define the $r$th mixing coefficient for $Y_1^{2n}$ by:
\begin{align}
\label{ beta-mixing-definition}
\beta_r (Y_1^{2n}) =\max_{1 \leq l \leq 2n-r} \beta( \sigma(Y_1, \ldots Y_l),  \sigma(Y_{l+r}, \ldots Y_{2n}))
\end{align}
\end{definition}

 The $\beta$-mixing coefficient also has the following representation as an $L_1$ distance between a conditional distribution and an unconditional distribution. This representation, coupled with Markov Inequality, facilitates the approximation of a conditional distribution with an unconditional one.      

\begin{align}
\label{ beta-mixing-definition}
\beta_r (Y_1^{2n})  =\max_{1 \leq l \leq 2n-r} \norm{ \ d_{TV}( \mathcal{L}(Y_{l+r}^{2n} \ | \ Y_1^{l}), \mathcal{L}(Y_{l+r}^{2n})) \ }_1
\end{align}

In addition, the following preservation property will be useful:

\begin{proposition}[Preservation of $\beta$-mixing coefficient] Let $f$ be a measurable function of $(Y_{n+r}, \ldots Y_{2n})$.  Then,
\begin{align}
\beta\bigl(\sigma\bigl(Y_1^n\bigr), \ \sigma\bigl(f(Y_{n+r+1}^{2n})\bigr)\bigr) \leq \beta_r(Y_1^{2n}) 
\end{align}
\end{proposition}

\paragraph{$\tau$-dependence}
\label{tau-dependence}
The $\tau$-coefficient is a measure of dependence introduced by \citet{tau-dependence}. The main (but not the only) difference from $\beta$-mixing is the choice of a different metric between the unconditional and conditional distributions. We will again introduce some preliminary notions before defining the $\mathlarger{\tau}$-coefficient.  

\begin{definition}[K-Lipschitz]
A map $f$ between metric spaces $(\mathcal{X}, \rho)$ and $(\mathcal{X}^\prime, \rho')$ is said to be $K$-Lipschitz if for all $x,y \in \mathcal{X}$,
\begin{align}
\label{lipschitz-definition}
\rho^\prime(f(x),f(y)) \leq K \rho(x,y) 
\end{align} 
The Lipschitz constant of $f$, denoted $||f||_L$, is defined as the smallest $K >0$ that satisfies (\ref{lipschitz-definition}).     
\end{definition}
In our problem setting, we will choose the sup-norm $||\cdot ||_\infty$ for both the domain and codomain. 

\begin{definition}[1-Wasserstein Distance]
The 1-Wasserstein distance between $P$ and $Q$, denoted $d_W(P,Q)$, is defined as:
\begin{align}
d_W(P,Q) = \sup_{f \in \Lambda_1} \int f \ d (P-Q) 
\end{align}  
where $\Lambda_1$ is the class of $1$-Lipschitz functions:  
\begin{align}
\Lambda_{1} = \biggl\{ f: \mathcal{X} \mapsto \mathbb{R} \ \bigr| \ ||f||_L  \leq 1  \biggr\} 
\end{align}
\end{definition}
Now we will introduce $\mathlarger{\tau}$-dependence, which is a measure of dependence between a random variable and a $\sigma$-field opposed to two $\sigma$-fields as in the case of $\beta$-dependence.  We will then build the $\mathlarger{\tau}$-coefficient using the notion of $\mathlarger{\tau}$-dependence.  

\begin{definition}[$\mathlarger{\tau}$-dependence]
Let $(\Omega, \mathcal{F}, P)$ be a probability space and let $\mathcal{E}$ be a sub-$\sigma$-field of $\mathcal{F}$.  Further $\mathcal{X}$ be a Polish space.  For any integrable random variable taking values in $\mathcal{X}$, define:
\begin{align}
\mathlarger{\tau}(\mathcal{E}, X) = \norm{  d_W \bigl( \mathcal{L}(X \ | \  \mathcal{E}), \ \mathcal{L}(X)\bigr) }_1
\end{align}
\end{definition}

% \begin{definition}[$\mathlarger{\tau}$-coefficient]
% Let $\Gamma(u,v,r)$ denote the collection of indices that satisfy $1 \leq i_1 \leq i_2 \leq \ldots i_u \leq i_j + r \leq j_1 \leq j_2 \leq j_v$.  Then, $\tau_Y(r)$ is defined as:     
% \begin{align}
% \mathlarger{\tau}_Y(r) = \sup_{v \geq 0} \frac{1}{v} \sup_{(i,j) \in \Gamma(u,v,r)} \ \mathlarger{\tau} ( \sigma(Y_1, Y_2, \ldots,  Y_{i_u}), (Y_{i_1}, Y_{i_2}, \ldots, Y_{i_v}))
% \end{align}
% \end{definition}

\begin{definition}[$\mathlarger{\tau}$-coefficient]
Let $\mathcal{V}_{a:b}$ denote the collection of all non-empty ordered subsets of integers $\{a, \ldots, b\}$ for $a < b$.  For an element $v \in \mathcal{V}_{a:b}$, let $v(i)$ denote the $i$th element and let $v(s)$ denote the last element.  Furthermore, let $s_v$ denote the cardinality of the set $v$.  Then,   
\begin{align}
\mathlarger{\tau}_r(Y_1^{2n}) = \max_{1 \leq l \leq 2n-r}  \max_{\ v \in \mathcal{V}_{l+r:2n} } \frac{1}{s_v} \ \mathlarger{\tau} ( \sigma(Y_1,\ldots, Y_l), (Y_{v(1)}, \ldots, Y_{v(s)}))
\end{align}
\end{definition}

% We would like to remark on a few additional differences between the $\tau$-coefficient and the $\beta$-mixing coefficient.  In the two-sided case, our $\tau$-coefficient is not strictly weaker than the case where we take $\mathcal{E} = \sigma(\ldots, Y_l)$.  However, it is weaker than the coefficient formed by taking the supremum over $\mathcal{E}_j = \sigma(Y_j, \ldots, Y_l)$, where $-\infty <  j < l$. 
We would like to remark that, unlike $\beta$-mixing, the $\tau$-coefficient is normalized by the number of random variables under consideration.  Defined in this way, the $\tau$-coefficient is closely related to the $\theta$-coefficient, which we will introduce shortly. 

Note that $\tau$-coefficient has the following preservation property.  The proof is a consequence of basic properties of the Wasserstein metric.  
\begin{proposition}[Lipschitz Preservation]
Suppose $f(Y_{n+r+1}^{2n})$ is a $K$-Lipschitz function.  Then, 
\begin{align}
\mathlarger{\tau}(\sigma(Y_1^n), f(Y_{n+r+1}^{2n})) \leq K (n-r) \ \mathlarger{\tau}_r(Y_1^{2n})
\end{align} 
\end{proposition} 

\paragraph{$\theta$-dependence}
Another notion of weak dependence that we will consider in this paper is $\theta$-dependence.  $\theta$-dependence is one of several dependence measures that were introduced in the seminal paper of \citet{doukhan-louhichi-new-dependence-measure}, falling under the larger class of $\Psi$-weak dependence measures.  Roughly speaking, these dependence measures are defined as:
\begin{align}
\label{covariance-past-future}
\sup_{f \in \mathcal{F}, \ g \in \mathcal{G}} \text{Cov}(f(`past'), g(`future'))
\end{align}
where different choices of function classes $\mathcal{F}$ and $\mathcal{G}$, among other things, lead to different dependence measures.  As in mixing theory, these coefficients aim to quantify how fast dependence decays as a function of the gap between the past and the future.

Most standard $\Psi$-weak dependence measures take $\mathcal{G}$ to be the class of bounded Lipschitz functions (in the case $\theta$-dependence, boundedness is not assumed for $\mathcal{G}$).   As a result of this choice of function class, the preservation properties of these coefficients are less robust than those of mixing coefficients, posing additional challenges.  Nonetheless, a rich theory has developed around these dependence measures; for an overview, see \citet{doukan-neumann-psi-weak-dependence} or \citet{weak-dependence}.

We will now proceed by defining the $\theta$-coefficient.  We will first state a definition more in the form of $\Psi$-weak dependence measures, and then state an alternative formulation that provides a connection to the $\tau$-coefficient.  In what follows, let $\mathcal{F}_u$ be the space of bounded functions mapping from $\mathcal{Y}^{u}$ to $\mathbb{R}$ and let $\mathcal{G}_v$ be the space of Lipschitz functions mappings from $\mathcal{Y}^v$ to $\mathbb{R}$. 

\begin{definition}[$\theta$-coefficient]Let $\Gamma_n(u,v,r)$ denote the collection of indices that satisfy $1 \leq i_1 i_2, \ldots i_u \leq i_j + r \leq j_1 \leq j_2 \leq j_v \leq 2n$.  Then $\theta_r(Y_1^{2n})$ is given by:
\begin{align}
\theta_r(Y_1^{2n}) = \sup_{u,v} \max_{(i,j) \in \Gamma_n(u,v,r)} \max_{f \in \mathcal{F}_u, g \in \mathcal{G}_v}\frac{\text{Cov}(f(Y_{i_1}, \ldots Y_{i_u} ), \ g(Y_{j_1} ,\ldots ,Y_{j_v}) )}{v \ ||f||_\infty \ ||g||_L}
\end{align} 
\end{definition} 

The $\theta$-coefficient can equivalently be expressed as follows. Analogous to the $\tau$-dependent case, we will first define $\theta$-dependence and then construct the $\theta$-coefficient.

\begin{definition}[$\theta$-dependence]
Let $(\Omega, \mathcal{F}, P)$ be a probability space and let $\mathcal{E}$ be a sub-$\sigma$-field of $\mathcal{F}$.  Further $\mathcal{X}$ be a Polish space.  For any integrable random variable taking values in $\mathcal{X}$, define:
\begin{align}
\theta(\mathcal{E}, X) = \sup_{g \in \Lambda_1 }\norm{ \mathbb{E}\bigl[ g(X) \ | \  \mathcal{E} \bigr] - \mathbb{E}\bigl[g(X)\bigr] }_1
\end{align}
\end{definition}

\begin{definition}[$\theta$-coefficient, alternative form]
Using the same notation for defining the $\tau_r(Y_1^{2n})$ in Section \ref{tau-dependence}, define the $\theta_r(Y_1^{2n})$ coefficient as: 
\begin{align}
\theta_r(Y_1^{2n}) = \max_{\ v \in \mathcal{V}_{l+r:2n} } \frac{1}{s_v} \ \theta( \sigma(Y_1,\ldots, Y_l), (Y_{v(1)}, \ldots, Y_{v(s)}))
\end{align}
\end{definition} 

See \citet{weak-dependence} for a proof of the equivalence of these two definitions.  From this latter definition, we can see that $\theta$-dependence and $\tau$-dependence differ only in the order of the supremum and the $L_1$ norm.  From this fact, we can immediately deduce the following:
\begin{proposition}[Relationship between $\theta$ and $\tau$]
\begin{align}
\theta(\mathcal{E},X) \leq \mathlarger{\tau}(\mathcal{E},X)
\end{align} 
\end{proposition}

We will often use $\theta$-dependence to derive results in the unconditional setting.  Aside from being weaker than $\tau$-dependence, $\theta$-dependence has some additional properties that make it preferable to $\tau$-dependence for some problems.  For instance, as mentioned above, $\theta$-dependence can be defined as a covariance between two classes of functions; this property will be crucial in our adaptation of the central limit theorem to the non-stationary setting.  In addition, the preservation properties of $\theta$-dependent sequences extend beyond Lipschitz functions; see Proposition \ref{theta-preservation} for details.

\subsection{Related Work in Post-Selection Inference}
\label{related-work}

Sample splitting has recently been investigated in the IID, low assumption setting by \citet{bootstrapping-sample-splitting-assumption-free}. In this work, the authors establish the validity of inference based on sample splitting uniformly over a large class of probability distributions.  Among other results, the authors establish a growing dimension Delta Method and study the properties of the Normal approximation and bootstrap on the inference set for a regression coefficient in a regime where the dimension is allowed to grow (slowly) with $n$.  In some sense, the present work can be viewed as an exploration of similar themes in the time series setting.  However, dependence introduces several new phenomena and our analysis uses different tools. Other work involving inference after sample splitting include \citet{wasserman-roeder-high-dim-vs} and \citet{meinhausen-stability-selection}.            

Another approach to post-selection inference that has closely related inference goals to ours is uniform inference, first studied in \citet{berk-POSI}.  In this approach, the goal is to establish confidence intervals that hold uniformly over all models in the model class without performing data splitting. While uniform validity is a very satisfying property, procedures that meet this criteria are generally computationally intensive and conservative.  This idea was extended to settings involving heteroskedastic, non-Normal errors with possible misspecification in \citet{Bachoc-valid-ci} and \citet{Bachoc-uniformly-valid}.  More recently, \citet{model-free-linear-regression} and \citet{valid-assumption-lean-regression} consider uniform inference with random design and possible time series dependence.  The notions of dependence considered including mixing and functional dependence measures introduced by \citet{wu-weak-dependence}. These coefficients often yield sharp results for problems related to the central limit theorem (e.g. \citet{jirak-berry-esseen-dependence}), but as noted in \citet{weak-dependence}, are not directly comparable to $\Psi$-dependence measures, as these measures assume that the process is a Bernoulli shift. 

While there are many other approaches in the literature proposed for post-selection inference, we will mention another general approach that has been investigated in a series of papers.  In selective inference, instead of conditioning on $\sigma(\mathcal{D}_1)$, one uses the whole dataset and conditions on a selection event, corresponding to $\hat{m}(Y_1^{2n}) = m$.  This approach was introduced in \citet{lockhart-sig-test-lasso} and various extensions have been studied in \citet{fithian-optimal-inference-after-selection} and \citet{lee-exact-post-selection-inf}, among others. In principle, this general framework is also viable for time series.  When the probability model is an exponential family, \citet{fithian-optimal-inference-after-selection} demonstrate that hypothesis tests formed by conditioning on $\hat{m}(Y_1^{2n}) = m$ are more powerful than those that condition on a finer $\sigma$-field.  While this type of optimality result is encouraging, these benefits come at the cost of stronger assumptions.  All of the proposed procedures in this framework require the covariates to be fixed and further require a geometric characterization of the selection regions, which are procedure-dependent.  It should be noted, however, that \citet{tian-asymptotic-selective} weaken the assumption of Normality to asymptotic Normality by deriving selective central limit theorems for affine selection procedures.

%  \section{Proposed Estimator}
%  \input{method}
% \label{sec:method}

 \section{Main Results}
 \label{main-results}
 \subsection{General Asymptotic Results}
 \label{general-asymptotic}
To establish results without assuming the existence of a limiting distribution, we will use the notion of weakly approaching random variables, introduced by \citet{Belyaev-Sjostedt-de-Luna-weakly-approach}. We will provide a definition below:
\begin{definition}[Weakly Approaching Random Variables] Two sequences of random variables $\{U_n\}_{n \in \mathbb{N}}$ and $\{V_n\}_{n \in \mathbb{N}}$ taking values in $\mathcal{U}$ are said to be weakly approaching if, for all bounded continuous functions we have that:
\begin{align}
\mathbb{E}f(U_n) - \mathbb{E}f(V_n) \rightarrow 0
\end{align}
This type of convergence is denoted as:
\begin{align}
\mathcal{L}(U_n)   {\overset{wa} \iff } \mathcal{L}(V_n)
\end{align}
\end{definition}
Notice that if $V_n = V$ for all $n$, the above definition corresponds to the more standard notion of weak convergence.  Weakly approaching is therefore a generalization of weak convergence. We have an analogous notion for weakly converging in probability:
\begin{definition}[Weakly Approaching Random Variables in Probability] Let $\{U_n, V_n, \mathbb{W}_n \}_{n \in \mathbb{N}}$ be a sequence of random variables defined on the same probability space $(\Omega, \mathcal{F}, P)$, where $U_n, V_n \in \mathcal{U}$ and $\mathbb{W}_n \in \mathcal{W}_n$. Then sequences of regular conditional probability laws corresponding to $\{\mathcal{L}(U_n)\}_{n \in \mathbb{N}}$ and $\{\mathcal{L}(V_n)\}_{n \in \mathbb{N}}$ are said to weakly approach in probability along $\{\mathbb{W}\}_{n \in \mathbb{N}}$ if for all bounded continuous functions we have that: 
\begin{align}
\mathbb{E}f(U_n \ | \ \mathbb{W}_n) - \mathbb{E}f(V_n \ | \ \mathbb{W}_n) \xrightarrow{P} 0
\end{align}
This type of convergence is denoted as:
\begin{align}
\mathcal{L}(U_n \ | \ \mathbb{W}_n)   {\overset{wa(P)} \iff } \mathcal{L}(V_n \ | \ \mathbb{W}_n)
\end{align}
\end{definition}
One may also define weakly approaching almost surely in an analogous fashion, but we do not state it here since we do not consider this mode of convergence in this paper. It turns out that many limit theorems from weak convergence theory have analogs even when there is no limiting distribution and we allow two sequences of random variables to drift together. We will see that the key requirement for many of these theorems is tightness, which is a necessary condition for weak convergence by Prokhorov's Theorem. Recall that tightness is defined as follows:

\begin{definition}[Tightness]
A sequence of random variables $\{X_n\}_{n \geq 1}$ is said to be tight if for any $\epsilon > 0$, there exists a compact set $C_{\epsilon}$ such that, for all $n \geq 1$:
\begin{align}
P(X_n \not\in C_{\epsilon}) \leq \epsilon
\end{align}
\end{definition}
Using these tools, we will focus on the example of post-selection inference for a collection of regression coefficients under a random design assumption.  We succeed in establishing validity under both $\beta$-mixing and $\tau$-dependence conditions using a weakly approaching version of the Delta Method.  The proof for the $\tau$-dependent case is more involved due to the fact that $\tau$-dependence has less robust preservation properties compared to $\beta$-mixing.

At the end of this section, we will also consider the problem of constructing confidence intervals for the regression coefficient under $\tau$-dependence.  

Finally, before stating our main results, we would like to briefly discuss our choice of metric.  In this section, we consider the bounded Lipschitz metric.  Let $||f||_{BL} = ||f||_\infty \vee ||f||_L$.  The definition we adopt is as follows.

\begin{definition}[Bounded Lipschitz Metric] 
\label{bounded-lipschitz-definition}
For two probability measures $P$ and $Q$, the bounded Lipschitz metric, denoted $d_{BL}$ is given by:
\begin{align}
d_{BL}(P,Q) =  \sup_{f \in BL_1} \int f \ d (P-Q)
\end{align}     
where $BL_1$ is the function class:
\begin{align}
BL_1 = \biggl\{ f: \mathcal{X} \mapsto \mathbb{R} \ \bigr| \ ||f||_{BL} \leq 1 \biggr\}
\end{align}
\end{definition}

We have at least two reasons for choosing this particular metric.  As mentioned in the introduction, the Bounded Lipschitz metric metricizes weak convergence. It turns out that metrics that metricize weak convergence have desirable properties even in the weakly approaching setting.  We will state a proposition that establishes this below.
\begin{proposition}
\label{metricizing-weakly-approaching}
 Let $d$ be a metric that metricizes weak convergence.  Assume that $\{V_n\}_{n \in \mathbb{N}}$ is tight.  Then, we have the following:
\begin{align}
\begin{split}
d(\mathcal{L}(U_n),\mathcal{L}(V_n)) \rightarrow 0 \ \  &\text{ if and only if } \ \ \mathcal{L}(U_n)   {\overset{wa} \iff } \mathcal{L}(V_n)  
% \\ \mathcal{L}(U_n \ | \ \mathbb{W}_n)   {\overset{wa(P)} \iff } \mathcal{L}(V_n \ | \ \mathbb{W}_n) \ \ &\text{ if and only if } \ \  d(\mathcal{L}(U_n \ | \ \mathbb{W}_n),\mathcal{L}(V_n \ | \ \mathbb{W}_n) ) \xrightarrow{P} 0
\end{split}
\end{align}
\end{proposition}
In \citet{Belyaev-Sjostedt-de-Luna-weakly-approach}, this equivalence is stated in terms of the multi-dimensional Levy metric, but an examination of the proof reveals that it remains valid for any metric that metricizes weak convergence.  The equivalence between a metric metricizing weak converging to $0$ and two random variables weakly approaching extends to the conditional case; since the proposition below does not appear in \citet{Belyaev-Sjostedt-de-Luna-weakly-approach}, a proof is provided in Appendix \ref{weakly-approaching-properties}.
\begin{proposition}
\label{metricizing-weakly-approaching-probability}
 Let $d$ be some metric that metricizes weak convergence.  Assume that $\{ V_n\}_{n \in \mathbb{N}}$ is tight.  Then, we have the following:
\begin{align}
\begin{split}
\mathcal{L}(U_n \ | \ \mathbb{W}_n)   {\overset{wa(P)} \iff } \mathcal{L}(V_n \ | \ \mathbb{W}_n) \ \ &\text{ if and only if } \ \  d(\mathcal{L}(U_n \ | \ \mathbb{W}_n),\mathcal{L}(V_n \ | \ \mathbb{W}_n) ) \xrightarrow{P} 0
\end{split}
\end{align}
\end{proposition}
 Furthermore, the Bounded Lipschitz metric interacts nicely with the metrics used to define our dependence coefficients. It is straightforward to see that for any probability measures $P$ and $Q$:
\begin{align}
 d_{BL}(P,Q) \leq d_W (P,Q) \ \ \  \text{ and } \ \ \  d_{BL}(P,Q) \leq d_{TV} (P,Q)
\end{align}

\subsection{Stability Theorem}

We are now ready to state our first result in this section, the stability theorem. The main idea behind this theorem is that estimators that satisfy an appropriate deletion stability condition will often be close in distribution to an estimator in which those points are actually deleted.  If the process is not too dependent, it will turn out sample splitting will be asymptotically valid for an estimator with observations deleted between the selection and inference sets, which in turn implies validity for the original estimator.  
\begin{theorem}[Stability Theorem]
\label{low-dim-theorem}
Suppose there exists $\mathcal{M}^*$ such that  $|\mathcal{M}^*| < \infty $ and  $P(\hat{m} \in \mathcal{M}^* ) \rightarrow 1$.   For each $m \in \mathcal{M}^*$, assume that there exists $i_n \rightarrow \infty$ such that the following conditions hold:
\begin{enumerate}[label=A\arabic*]
\item \label{tight} $\{\hat{\theta}_{m}(Y_{n+i_n+1}^{2n})\}_{n \geq 1}$ is tight. 
\item \label{stability-low-dim} $|\hat{\theta}_{m}(Y_{n+1}^{2n}) - \hat{\theta}_{m}(Y_{n+i_n+1}^{2n})| = o_P(1)$
\end{enumerate}
In addition, assume one of the following conditions. For each $m \in \mathcal{M}^*$,  
\begin{enumerate}[label=T]
\item \label{tau-condition} $(n-i_n) \cdot \mathlarger{\tau}_{i_n}(Y_1^{2n}) \cdot || \hat{\theta}_{m}(Y_{n+i_n+1}^{2n}) ||_L \rightarrow 0$.
\end{enumerate} 
\begin{enumerate}[label=B] 
\item \label{beta-condition}$\hat{\theta}_{m}(Y_{n+i_n+1}^{2n})$ is measurable and $\beta_{i_n}(Y_1^{2n}) \rightarrow 0$.
\end{enumerate} 
Then, 
% \begin{align*}
% \mathcal{L}(\hat{\theta}_{\hat{m}}(Y_{n+1}^{2n}) |Y_1^n)  {\overset{wa(P)} \iff} \mathcal{L}(\hat{\theta}_{\hat{m}}(\widetilde{Y}_{n+1}^{2n})  |Y_1^n)
% \end{align*}
%
\begin{align}
d_{BL}(\mathcal{L}(\hat{\theta}_{\hat{m}}(Y_{n+1}^{2n}) \ | \ Y_1^n), \mathcal{L}(\hat{\theta}_{\hat{m}}(\widetilde{Y}_{n+1}^{2n}) \ | \ Y_1^n)) \xrightarrow{P} 0   
\end{align}
\end{theorem}

\begin{proof}

By law of total probability partitioning on $\{\hat{m} \in \mathcal{M}^*\}$:
 \begin{align}
\begin{split}
& \ \ \sum_{m \in \mathcal{M}^*} P \left( d_{BL} \left[ \mathcal{L}(\hat{\theta}_{\hat{m}}(Y_{n+1}^{2n}) \ | \ Y_1^n), \mathcal{L}(\hat{\theta}_{\hat{m}}(\widetilde{Y}_{n+1}^{2n} \ | \ Y_1^n) \right] > \epsilon, \  \hat{m} = m\right)
\\ & \ \ \ \ \ + P \left( d_{BL} \left[ \mathcal{L}(\hat{\theta}_{\hat{m}}(Y_{n+1}^{2n}) \ | \ Y_1^n), \mathcal{L}(\hat{\theta}_{\hat{m}}(\widetilde{Y}_{n+1}^{2n}) \ | \ Y_1^n) \right] > \epsilon, \  \hat{m} \not\in \mathcal{M}^* \right) 
\\ & \leq \sum_{m \in \mathcal{M}^*} P \left\{ d_{BL} \left[ \mathcal{L}(\hat{\theta}_{m}(Y_{n+1}^{2n})\ | \ Y_1^n), \mathcal{L}(\hat{\theta}_{m}(\widetilde{Y}_{n+1}^{2n}) \right] > \epsilon \right\} + P(\hat{m} \not\in \mathcal{M}^* )
\end{split}
 \end{align}

By Proposition \ref{metricizing-weakly-approaching-probability} it suffices to show that, for each $m \in \mathcal{M}^*$: 
\begin{align}
\mathcal{L}(\hat{\theta}_{m}(Y_{n+1}^{2n}) \ | \ Y_1^n)  {\overset{wa(P)} \iff} \mathcal{L}(\hat{\theta}_{m}(Y_{n+1}^{2n}))
\end{align}

% \begin{align*}
% & \ \ \ d_{L} \left( \mathcal{L}(\hat{\theta}_{\hat{m}} (Y_{n+1}^{2n}) |Y_1^n), \  \mathcal{L}(\hat{\theta}_{m_\infty}(Y_{n+1}^{2n})) \right)
% \\ & \leq  d_{L} \left( \mathcal{L}(\hat{\theta}_{\hat{m}}(Y_{n+1}^{2n}) |Y_1^n), \mathcal{L}(\hat{\theta}_{\hat{m}}(Y_{n+i_n+1}^{2n}) |Y_1^n)  \right)
% \\ & \ \ \  + d_{L} \left( \mathcal{L}(\hat{\theta}_{\hat{m}}(Y_{n+i_n+1}^{2n}) |Y_1^n), \mathcal{L}(\hat{\theta}_{m_\infty}(Y_{n+i_n+1}^{2n})|Y_1^n) \right)
% \\ & \ \ \  + d_{L} \left( \mathcal{L}(\hat{\theta}_{m_\infty}(Y_{n+i_n+1}^{2n}) |Y_1^n), \mathcal{L}(\hat{\theta}_{m_\infty}(Y_{n+i_n+1}^{2n}) \right)
% \\ & \ \ \  + d_{L} \left( \mathcal{L}(\hat{\theta}_{m_\infty}(Y_{n+i_n+1}^{2n})), \mathcal{L}(\hat{\theta}_{m_\infty}(Y_{n+1}^{2n})) \right) = \mathbf{I} +  \mathbf{II} +  \mathbf{III}
% \end{align*}

Notice that:
\begin{align}
\hat{\theta}_{m}(Y_{n+1}^{2n}) = \hat{\theta}_{m}(Y_{n+i_n +1 }^{2n}) + [\hat{\theta}_{m}(Y_{n+1}^{2n}) -\hat{\theta}_{m}(Y_{n+i_n +1 }^{2n})]
\end{align}
Now, since $\{\hat{\theta}_{m}(Y_{n+i_n+1}^{2n})\}_{n \geq 1}$ is tight by \ref{tight}, and $|\hat{\theta}_{m}(Y_{n+1}^{2n}) - \hat{\theta}_{m}(Y_{n+i_n+1}^{2n})| = o_p(1)$ by \ref{stability-low-dim}, by Weakly Approaching Conditional Slutsky's Theorem (Proposition \ref{conditional-slutksy}), it is sufficient to show that:
\begin{align}
\mathcal{L}(\hat{\theta}_{m}(Y_{n+i_n+1}^{2n}) \ | \ Y_1^n)  {\overset{wa(P)} \iff } \mathcal{L}(\hat{\theta}_{m}(Y_{n+1}^{2n})) 
\end{align}
As an intermediate step, we will show that:
\begin{align}
\label{weakly-approaching-intermediate}
\mathcal{L}(\hat{\theta}_{m}(Y_{n+i_n+1}^{2n}) \ | \ Y_1^n)  {\overset{wa(P)} \iff } \mathcal{L}(\hat{\theta}_{m}(Y_{n+i_n+1}^{2n}))
\end{align}
To show (\ref{weakly-approaching-intermediate}), we will use Markov's inequality in conjunction with the dependence coefficients. Under assumption \ref{tau-condition}, we have that: 
% \begin{align}
% \begin{split}
% \mathcal{L}(\hat{\theta}_{m^*}(Y_{n+1}^{2n})) & \overset{wa}{\underset{(1)} \iff }  \mathcal{L}(\hat{\theta}_{m^*}(Y_{n+i_n+1}^{2n})) \overset{wa(P)} {\underset{(2)} \iff} \mathcal{L}(\hat{\theta}_{m^*}(Y_{n+i_n+1}^{2n}) |Y_1^n) 
% \\ & \overset{wa(P)} {\underset{(3)} \iff} \mathcal{L}(\hat{\theta}_{m^*}(Y_{n+1}^{2n}) |Y_1^n)
% % \overset{wa(P)} {\underset{(4)} \iff} \mathcal{L}(\hat{\theta}_{\hat{m}}(Y_{n+1}^{2n}) |Y_1^n)
% \end{split}
% \end{align}
%
% \begin{align} 
% \mathcal{L}(\hat{\theta}_{\hat{m}} (Y_{n+1}^{2n}) |Y_1^n) \overset{wa(P)} {\underset{(5)} \iff} \mathcal{L}(\hat{\theta}_{m_\infty}(Y_{n+1}^{2n}))
% \end{align}

 \begin{align}
 \begin{split}
& \ \ \ P \left\{ d_{BL} \left[ \mathcal{L}(\hat{\theta}_{m}(Y_{n+i_n+1}^{2n}) \ | \ Y_1^n), \mathcal{L}(\hat{\theta}_{m}(Y_{n+i_n+1}^{2n}) \right] > \epsilon \right\}
\\ & \leq P \left\{ d_{W} \left[ \mathcal{L}(\hat{\theta}_{m}(Y_{n+i_n+1}^{2n}) \ | \ Y_1^n), \mathcal{L}(\hat{\theta}_{m}(Y_{n+i_n+1}^{2n}) \right] > \epsilon \right\}
\\ & \leq \frac{(n-i_n) \ ||\ \hat{\theta}_{m}(Y_{n+i_n+1}^{2n}) \ ||_L \ \mathlarger{\tau}_{i_n}(Y_1^{2n}) }{\epsilon} \rightarrow 0
\end{split}
 \end{align}
 The $\beta$-mixing case is analogous since $d_{BL} \leq d_{TV}$. The result will follow if we can establish that: 
 \begin{align}
 \label{weakly-approaching-last-step}
  \mathcal{L}(\hat{\theta}_{m}(Y_{n+i_n+1}^{2n}))  {\overset{wa } \iff }  \mathcal{L}(\hat{\theta}_{m}(Y_{n+1}^{2n}))
 \end{align}   
 Again,  express $\hat{\theta}_{m}(Y_{n+1}^{2n})$ as $\hat{\theta}_{m}(Y_{n+i_n+1}^{2n}) + [\hat{\theta}_{m}(Y_{n+1}^{2n}) - \hat{\theta}_{m}(Y_{n + i_n+1}^{2n})] $.  Then by Weakly Approaching Slutsky's Theorem (Proposition \ref{slutksy}) we indeed have (\ref{weakly-approaching-last-step}). The result follows. 
\end{proof}

We will now discuss some of the conditions of theorem.  We will start by discussing the conditions on dependence.  While the stability theorem applies to a wide range of functionals under $\beta$-mixing, the $\tau$-dependent case is more delicate. When the mapping is not Lipschitz, the stability theorem cannot be applied off-the-shelf to $\tau$-dependent processes. However, in Theorem \ref{sub-exponential-sample-splitting} we will show that a truncation argument in the Bounded Lipschitz metric, combined with the Delta Method, can yield sample splitting validity for regression coefficients.

We would also like to to note that, in both the $\beta$-mixing and $\tau$-dependent cases, it is not clear how the dependence coefficients behave when new covariates are included. However, if the new covariates only include lags of the process, then the dependence coefficients can be inferred from that of the underlying process.  

We also assume that the effective number of models is finite, which we believe is reasonable in a fixed-dimension framework.  Condition \ref{tight} is a tightness condition that allows us to use a weakly approaching version of Slutsky's Theorem, which facilitates a short proof of the theorem.  
Condition \ref{stability-low-dim} is the asymptotic deletion stability condition, which is probabilistic rather than deterministic in nature.  We will give a couple of concrete examples below to show that this condition is actually quite weak in practice. 

\begin{example}[Asymptotic Deletion Stability of Sample Mean]
\label{sample-mean-example}
Let $Y_1, \ldots Y_n$ be centered random variables.  Suppose our estimator is given by:
\begin{align}
 \hat{\theta}_m(Y_{n+1}^{2n}) = \frac{1}{\sqrt{n}} \sum_{i=n+1}^{2n} Y_i
\end{align}
Then,
\begin{align}
\begin{split}
|\hat{\theta}_m (Y_{n+1}^{2n})-\hat{\theta}_m(Y_{n+i_n+1}^{2n}) | &= \left| \frac{1}{\sqrt{n}} \sum_{i=n+1}^{2n} Y_i -  \frac{1}{\sqrt{n-i_n}} \sum_{i=n+i_n + 1}^{2n} Y_i \right|
 \\ & \leq \left|\frac{1}{\sqrt{n}}\sum_{i=n+1}^{n+i_n} Y_i  \right| +  \left( \frac{\sqrt{n}+\sqrt{i_n}}{\sqrt{n}} -1 \right)   \left|\frac{1}{\sqrt{n-i_n}} \sum_{i=n+i_n + 1}^{2n} Y_i \right|
\end{split}
\end{align}
Assuming $\frac{1}{\sqrt{n-i_n}} \sum_{i=n+i_n + 1}^{2n} Y_i$ and $\frac{1}{\sqrt{i_n}} \sum_{i=n+1}^{n+i_n} Y_i$ are $O_P(1)$, it is clear that the A2 holds with $i_n = n^{1/2-\delta}$ for any $0 <\delta < 1/2$.  
\end{example}

For the next example, we will provide some sufficient conditions for the normalized maximum to satisfy A2. Let $U_n = \max_{1 \leq i \leq n} X_i$, where $X_i$ are possible dependent random variables taking values in $\mathbb{R}$ and let $a_n, b_n \geq 0$.  When $X_i, \ldots, X_n$ are IID, necessary and sufficient conditions for weak convergence results of the following form are known: 

\begin{align}
\frac{U_n - b_n}{a_n} \rightsquigarrow G
\end{align}
where $G$ is one of three limiting distributions.  In addition, rates for $a_n$ and $b_n$ are known for each of the three cases; see for example \citet{leadbetter-extremes}. In the dependent case, under a weak dependence condition weaker than mixing, $G$ is also one of three limiting distributions; see \citet{leadbetter-extremes-stationary}. Below we will give a general sufficient condition, and prove stability for the special case of stationary, strongly mixing Gaussian processes with $E(X_i) = 0$ and $E(X_i^2) = 1$.       

\begin{example}[Asymptotic Deletion Stability of Sample Maximum] Assume that $Y_{n+1} \ldots, Y_{2n}$ are observations from a strictly stationary stochastic process with mean $0$.  Suppose our estimator is given by:
\begin{align}
 \hat{\theta}_m(Y_{n+1}^{2n}) = \frac{U_{(n+1):2n} - b_n}{a_n} 
\end{align}
where $U_{n+1:2n} = \max_{(n+1) \leq j \leq 2n} \ Y_j$.  Then, it is sufficient to bound: 
\begin{align}
\label{max-stability-bound}
\begin{split}
 & \ \ \  \left| \frac{U_{1:n} - b_n}{a_n} -  \frac{U_{(i+1):n} - b_{n-i}}{a_{n-i}}   \right|
\\ &  \leq \left| \frac{U_{1:n} - b_n}{a_n} - \frac{U_{(i+1):n} - b_n}{a_n} \right| + \left|\frac{a_{n-i}}{a_{n}} - 1 \right| \cdot \left| \frac{U_{(i+1):n} - b_{n-i}}{a_{n-i}} \right| + \left| \frac{b_{n-i} - b_n}{a_n}  \right| 
\\ & \leq \left| \frac{U_{1:i}}{a_n} \right| + \left|\frac{a_{n-i}}{a_{n}} - 1 \right| \cdot \left| \frac{U_{(i+1):n} - b_{n-i}}{a_{n-i}} \right| +  \left| \frac{b_{n-i} - b_n}{a_n}  \right| 
 % \left| \frac{U_{i_n} - b_n}{a_n} - \frac{U_{i:n} - b_{n-i}}{a_{n-i}} \right| 
 \end{split}
\end{align}
From the bound above, we see that A2 is satisfied if $i_n \rightarrow \infty $ is chosen such that $U_{1:i} =o_P(a_n)$, $\frac{a_{n-i}}{a_{n}} \rightarrow 1$, and $|b_{n-i} - b_n| = o(a_n)$.    

In the case of a stationary Gaussian process that is strong mixing, zero mean, and unit variance, $\hat{\theta}_m(Y_{n+1}^{2n})$ is known to have a limiting distribution and is therefore tight (see \citet{deo-strong-mixing-gaussian-maximum}) when: 
\begin{align}
a_n = (2 \log n)^{-\frac{1}{2}}, \ \ \ \ \ b_n = (2 \log n)^{-\frac{1}{2}} - \frac{1}{2}(2 \log n)^{-\frac{1}{2}}(\log \log n + \log 4 \pi)   
\end{align}

Note that the derivative of $b_n$ is given by:
\begin{align}
\frac{db}{dn} = \frac{\log(4 \pi \log n))-4}{2n \log^{\frac{3}{2}}n}
\end{align}
Since $\sup_{n \geq 2} \left| \frac{db}{dn} \right| \leq 1$, by Mean Value Theorem, it follows that the third term in (\ref{max-stability-bound}) is bounded by $i_n/a_n$.  To bound the first two terms, it also suffices to take $i_n = o(a_n)$.    
\end{example}

One may wonder if it is possible to eliminate the stability condition.  However, the following counterexample illustrates a somewhat artificial case in which it is not satisfied, leading to sample splitting not being valid.  Here, sample splitting validity fails because a finite number  of points (here one) has nontrivial influence on the estimator asymptotically.

\begin{example}[A case where stability condition is not satisfied]
Suppose $Y_1, \ldots Y_{2n}$ is a 1-dependent Gaussian process with mean $0$ and that our estimator is given by:
\begin{align}
\hat{\theta}_m(Y_{n+1}^{2n}) = Y_{n+1} + \frac{1}{\sqrt{n}} \sum\limits_{i=n+2}^{2n} Y_i
\end{align}
We can immediately see that the distribution of $\hat{\theta}_m(Y_{n+1}^{2n})$ conditioning on $Y_1^n$ will not converge to the unconditional distribution due to the strong influence of $Y_{n+1}$.
\end{example}
One may observe that these kind of trivial estimators can be eliminated by leaving a gap between the selection and inference sets.  While deleting observations allows one to eliminate the stability condition, this introduces a tuning parameter that is typically not needed.  

\subsection{Some Additional Asymptotic Tools for Establishing Sample Splitting Validity}
Even though our notion of stability is weak, proving it for a complicated functional can be nontrivial.  In the fixed-dimension case, we can work around this issue by expressing the functional as a composition of mappings.  If it is possible to establish sample-splitting validity using the stability theorem for a simpler intermediate mapping, one can use weakly approaching analogs of standard limit theorems for the composition.

\begin{proposition}[Weakly Approaching Continuous Mapping Theorem]
\label{continuous-mapping}
 Let $h(\cdot)$ be a continuous function in $\mathcal{X}$. Then,
 \begin{enumerate}
 \item[(i)] If $X_n {\overset{wa} \iff } Y_n$, then $h(X_n) {\overset{wa} \iff } h(Y_n)$.  
 \item[(ii)] If $X_n \ | \ \mathbb{Z}_n {\overset{wa(P)} \iff } Y_n$, then $h(X_n) \ | \  \mathbb{Z}_n {\overset{wa(P)} \iff } h(Y_n)$
 \end{enumerate}
\end{proposition}
 The variant of the continuous mapping theorem stated here is an immediate consequence of the definition of weakly approaching random variables; as in weak convergence theory, weaker versions are possible, but we do not pursue this direction here.

 We will also state a version of the Delta Method below where the parameter $\theta_n$ need only belong to a closed ball asymptotically; although the proof is not particularly difficult, to our knowledge, it is the first result of its kind.   We only require that the second derivative is well-behaved over the ball.  This condition allows uniform control of the remainder term over a family of Taylor expansions. Uniform control is needed to deal with the fact that the centering of the Taylor expansion is allowed to change with $n$.

While we believe the differentiability condition is reasonable, it should be noted that weaker conditions are possible if one assumes that $\theta_n \rightarrow \theta$. See \citet{van-der-vaart-asymptotic-statistics}[Theorem 3.8] for a Delta Method along these lines.  

\begin{proposition}[Weakly Approaching Conditional Delta Method]
\label{delta-method}
Suppose that, for some sequence $\{\theta_n\}_{n \in \mathbb{N}} \in \mathcal{U}$ and some $\gamma > 0$,  
\begin{align}
\mathcal{L} ( n^\gamma [ U_n - \theta_n] \  | \  \mathbb{W}_n) {\overset{wa(P)} \iff} \mathcal{L}(V_n)
\end{align}
where $V_n$ is tight. Suppose there exists a closed ball $B_r$ with radius $0 <r <\infty$ such that $\theta_n \in B_r$ for all $n > N$ and that, for some $\epsilon >0$, $h(\cdot)$ is twice continuously differentiable on $B_{r + \epsilon}$.  Then,   
\begin{align}
\mathcal{L} (n^\gamma [h(U_n) - h(\theta_n)] \  | \  \mathbb{W}_n) {\overset{wa(P)} \iff} \mathcal{L}( \nabla h(\theta_n)^T V_n )
\end{align}
\end{proposition}
\begin{proof} 
A Taylor expansion at $\theta_n$ yields:
\begin{align}
n^\gamma[h(U_n)-h(\theta_n)] =  \nabla h(\theta_n)^T n^\gamma [U_n-\theta_n] + n^\gamma (U_n-\theta_n)^T R_{\theta_n}(U_n) (U_n-\theta_n)
\end{align}
where $R_{\theta_n}: \mathbb{R}^k \mapsto \mathbb{R}^{k \times k} $ is the remainder function for the expansion around $\theta_n$. The assumed differentiability condition implies that:
\begin{align}
\max_{u \in B_{r+\epsilon}} \norm{R_{\theta_n}(u)}_\infty \leq \max_{u \in B_{r+\epsilon}} \norm{ \nabla^2 h(u)}_\infty
\end{align}
where the maximum is attained by a finite constant due to the Boundedness Theorem. It follows that the family of Remainder functions $\{R_{\theta_n}(u)\}_{\theta_n \in B_r}$ is uniformly bounded for any Taylor expansion around $\theta_n \in B_{r}$ with radius $\epsilon$.  Thus, with high probability, the second term on the RHS is bounded by:
\begin{align}
C [(U_n-\theta_n)]^T \left[n^\gamma(U_n-\theta_n)\right]
\end{align}
for some $C < \infty$. By Slutsky's Theorem (Proposition \ref{slutksy}), it follows that this term converges to $0$ in probability.  Two more applications of Slutsky's Theorem give the desired result.    
% To show the convergence of the first term on the RHS, we will use characteristic functions.  We will state a proposition below from \citet{Belyaev-Sjostedt-de-Luna-weakly-approach}.  A related proposition will be needed to finish the proof.     
% %
% Then, we can show, for each fixed $s \in \mathbb{R}$, uniform convergence of the characteristic function over the set $ \{ t \in \mathbb{R}^k \ | \  t = \nabla h(\theta_n) s, \theta_n \in B \} $ to finish the proof.
\end{proof}
% \begin{remark}
% While it is the case that functions with bounded derivatives are also Lipschitz, it may be easier to use the Delta Method to establish sample splitting validity for $\mathlarger{\tau}$-dependent sequences for at least two reasons.  The first is that, for the Lipschitz preservation property to hold, it is necessary that the function is globally Lipschitz (it should be noted that this requirement may be bypassed by using a truncation argument in the Bounded Lipschitz metric).  In contrast,  the weakly approaching Delta Method requires differentiability only in the limiting ball.  In addition, the Delta Method does not require a bound on the Lipschitz constant for the outermost mapping.     
% \end{remark}

 \subsection{Sample Splitting Validity Under Random Design for Linear Regression Coefficients}
 \label{regression-setting}  
  
For a randomly chosen model $\hat{m}$, we will consider the problem of constructing asymptotically valid confidence intervals for the projection parameter $\beta_{\hat{m},n}$, which is a vector of regression coefficients fitted on the second half of the data. We will allow the regression coefficient to incorporate lags of the components of $Y_i$.  To make this more precise, we will introduce some additional notation.

For each $m \in \mathcal{M}$, let $X_{i}^{(m)} \in \mathbb{R}^{p^{(m)}}$ denote the covariates associated with $m$. $X_{i}^{(m)}$ may include both present values of $Y_{n, i,2}, \ldots Y_{n, i,p_n}$.  Let $Y_i^{m} = Y_{n,i,1}$. Define the projection parameter for a fixed $m \in \mathcal{M}$ as:  
 \begin{align}
 \label{projection-parameter}
\beta_{m,n} &= \mathbb{E}\left[ \frac{1}{n}\sum_{i=n+1}^{2n} X_{i}^{(m)} \ X_{i}^{(m) \  T} \right]^{-1} \mathbb{E}\left[ \frac{1}{n}\sum_{i=n+1}^{2n} Y_i^{(m)} \ X_{i}^{(m)} \right] 
  \end{align}
% \begin{align}
% \beta_{m,n} &= \argmin_{\beta \in \mathbb{R}^{d_m}} \mathbb{E} \left[ \frac{1}{n}\sum_{i=n+1}^{2n} (Y_i^{(m)} - \beta^T X_{i}^{(m)})^2 \right]
% \end{align} 

We may view the estimated parameter corresponding to a model $m$ as the empirical version of the projection parameter.  It can be readily seen that:
 \begin{align}
 \label{estimated-projection-parameter}
\hat{\beta}_{m,n} &= \left(\frac{1}{n}\sum_{i=n+1}^{2n} X_{i}^{(m)} \ X_{i}^{(m) \ T} \right)^{-1} \frac{1}{n}\sum_{i=n+1}^{2n} Y_i^{(m)} \ X_{i}^{(m)}
\end{align}
To further consolidate notation, define the following terms:
  % $W_{m,i} = \mathrm{vech}(X_{m,i} \ X_{m,i}^T)$ and $\alpha_i = Y_i \ X_{m,i}$.  Here, $\mathrm{vech}(\cdot)$ denotes the function that stacks elements on the upper diagonal of a matrix into a vector.  Further let $\psi_n$ denote:
\begin{align}
\begin{split}
V_{i}^{(m)} &= \begin{pmatrix}
          \mathrm{vech}(X_{i}^{(m)} \  X_{i}^{(m) \ T}) \\ 
          Y_i^{(m)} \  X_{i}^{(m)}
         \end{pmatrix} 
\\ \psi_{n}^{(m)} &= \frac{1}{n} \sum_{i=n+1}^{2n} V_{i}^{(m)} 
\end{split}
\end{align}

Then, it is clear from (\ref{projection-parameter}) and (\ref{estimated-projection-parameter}) that we may express $\beta_{m,n} = g^{(m)}(\mathbb{E}[\psi_{n}^{(m)}])$ and $\hat{\beta}_{n} = g^{(m)}(\psi_{n}^{(m)})$ for some mapping $g^{(m)}(\cdot)$ for each model $m$. For random $\hat{m}$, we may express the projection parameter as  $\beta_{\hat{m},n} = \sum_{m \in \mathcal{M}} \beta_{m,n} \cdot \cI(\hat{m}(Y_1^n) = m)$ and its empirical version as  $\hat{\beta}_{\hat{m},n} = \sum_{m \in \mathcal{M}} \hat{\beta}_{m,n}  \cdot \cI(\hat{m}(Y_1^n) = m)$.  If a particular model $m$ is infeasible to fit on the inference set, then $P(\hat{m}(Y_1^n) = m) = 0$.  In addition, in what follows let $\widetilde{\mathcal{L}}( \cdot)$ denote the law with respect to the joint measure of $Y_1^n$ and $\widetilde{Y}_{n+1}^{2n}$.  

We will start by assuming that $|Y_{n,i} - \mathbb{E}[Y_{n,i}]| < M_n \ a.s$. for $1 \leq i \leq 2n$.  In this case, we may use the Stability Theorem together with the Conditional Delta Method to establish sample splitting validity for the regression coefficient, as demonstrated in the following theorem.

\begin{theorem}[Sample Splitting Validity, Bounded Case] Suppose that \\ $ \max_{1 \leq i \leq n} \max_{1 \leq j \leq p}$ $|Y_{ij} - \mathbb{E}[Y_{ij}]|\leq M_n$ $a.s.$  Furthermore, suppose there exists $\mathcal{M}^*$ such that $|\mathcal{M}^*| < \infty$ and $ P(\hat{m} \in \mathcal{M}^*) \rightarrow 1$.  For each $m \in \mathcal{M}^*$, suppose that $\sqrt{n}(\psi_{n}^{(m)} - \mathbb{E}[\psi_{n}^{(m)}])$ is tight, and $g^{(m)}(\cdot)$ is twice continuously differentiable on  $B_{r+\epsilon}$ for some $\epsilon > 0$ as defined in Proposition \ref{delta-method}.   

Further, for each $m \in \mathcal{M}^*$, suppose that one of the following conditions hold:
\begin{enumerate}[label=T]
\item  $  n^{3/2} M_n \mathlarger{\tau}_{n^{1/2-\delta}}(Y_1^{2n}) \rightarrow 0$.
\end{enumerate} 
\begin{enumerate}[label=B] 
\item $\beta_{n}(Y_1^{2n}) \rightarrow 0$.
\end{enumerate} 
Then,
\begin{align}
d_{BL}(\mathcal{L}( \sqrt{n}(\hat{\beta}_{\hat{m},n} - \beta_{\hat{m},n}) \ | \ Y_1^n), \widetilde{\mathcal{L}}(  \sqrt{n}(\hat{\beta}_{\hat{m},n}- \beta_{\hat{m},n})  \ | \ Y_1^n) \xrightarrow{P} 0 
\end{align}
\end{theorem}

\begin{proof}
We will define $\hat{\theta}_{m}(Y_{n+1}^{2n}) \equiv g^{(m)} \circ f^{(m)}(Y_{n+1}^{2n})$ as the following composition of functions:
\begin{align}
\label{Delta-Method-composition}
\underbrace{Y_{n+1}^{2n} \mapsto V_{n+1}^{2n} \mapsto \sqrt{n} \psi_n}_{f^{(m)}(Y_{n+1}^{2n})} \mapsto \underbrace{\theta_{m}}_{g^{(m)}(\psi_n)}
\end{align}
We will need to check condition \ref{stability-low-dim} and \ref{tau-condition} for the mapping $f^{(m)}(Y_{n+1}^{2n})$. Since $g^{(m)}(\cdot)$ is differentiable at $\mathbb{E}[\psi_n^{(m)}]$, the result would then follow from the Delta Method.         
 Pick $i_n = n^{1/2 -\delta}$; the stability condition was established in Example \ref{sample-mean-example} under the tightness assumption, which is satisfied due to the Central Limit Theorem given in Theorem \ref{non-stationary-clt}.  For \ref{tau-condition}, notice that the Lipschitz constant of the mapping is given by $2 M_n \sqrt{n}$.
\end{proof}

Now, we will prove the unbounded case.  The main technical challenge is dealing with the fact that the mapping $f^{(m)}(Y_{n+1}^{2n})$ is only locally Lipschitz and not globally Lipschitz.  At several stages in the proof, we will reduce the problem to the bounded case by using a maximal inequality for sub-exponential random variables.  We will adopt the following definition:
\begin{definition}[Sub-Exponential($K_1$)]
A real-valued random variable $X$ is sub-exponential with parameter $K_1$, or $X \sim \mathrm{SubE}(K_1)$ if:
\begin{align}
P(|X| > t ) \leq \exp\left(\frac{1-t}{K_1} \right)
\end{align}
\end{definition}
As discussed in, for example, \citet{vershynin-nonasympotic-theory-random-matrices}, there are several equivalent characterizations of sub-exponential random variables. For convenience, in a later section we will also state the sub-exponential condition in terms of the sub-exponential norm, which is defined as follows:
\begin{definition}[Sub-Exponential Norm] The sub-exponential norm of a real-valued random variable $X$, denoted $\norm{X}_{\Psi_1}$ is given by:
\begin{align}
\norm{X}_{\Psi_1} = \sup_{p \geq 1} p^{-1} \left(\mathbb{E}[X^p]\right)^{1/p}
\end{align} 
\end{definition}
A random variable is sub-exponential if and only if $\norm{X}_{\Psi_1}<\infty$. In fact, it can be shown that $K_1$ and $\norm{X}_{\Psi_1}$ differ by at most a universal constant factor.  Typically, the sub-exponential condition for a random vector $X$ in $\mathbb{R}^d$ is defined as $\sup_{\alpha \in \mathcal{S}^{d-1}} \norm{\alpha^T X}_{\Psi_1}$, where $\mathcal{S}^{d-1}$ is the unit sphere in $\mathbb{R}^d$. Here, we will only require that the coordinates are sub-exponential, which is clearly a weaker condition. 
%  More precisely, we will require that:
% \begin{align}
% \sup_{i \in \mathbb{N}} \max_{1 \leq j \leq d} \norm{Y_{ij}}_{\psi_1} < \mu
% \end{align}
% The above is Orlicz-norm, which can be defined for a real-valued random variable $X$ and any increasing, convex function from $[0,\infty)$ to $[0,\infty)$ as follows:
% \begin{align}
% \norm{X}_{\Psi} = \inf_{c > 0} \left\{\mathbb{E}\left(\frac{|X|}{c} \right) \leq 1  \right\}
% \end{align}

At a high level, we will also need a way to return to the unbounded case after performing manipulations in the bounded case.  To this end, we will use a concept from analysis known as the Lipschitz extension.  In essence, Lipschitz functions are well-behaved enough that it is often possible to extend a function defined on a subset of the space to the whole space while preserving the Lipschitz property. In particular, when the codomain is $\mathbb{R}$, explicit constructions of the extension are known and facilitate our analysis.  The boundedness of the function class is also crucial; this property allows us to control the expectation on the complement of a bounded set.         

\begin{theorem}[Sample Splitting Validity of Regression Coefficient, Sub-Exponential Case]
\label{sub-exponential-sample-splitting}
Consider a model selection procedure $\hat{m}$ that has the property $P(\hat{m} \in \mathcal{M}^*) \rightarrow 1$ for some $\mathcal{M}^*$ satisfying $|\mathcal{M}^*| < \infty$. For each $m \in \mathcal{M}^*$, suppose that $\sqrt{n}(\psi_{n}^{(m)} - \mathbb{E}[\psi_{n}^{(m)}])$ is tight, and $g^{(m)}(\cdot)$ is twice continuously differentiable on  $B_{r+\epsilon}$ for some $\epsilon > 0$ as defined in Proposition \ref{delta-method}. 

Further, for each $m \in \mathcal{M}^*$, suppose that one of the following conditions hold:
\begin{enumerate}[label=T]
\item  For some $\delta>0$, $n^{3/2 + \delta} \mathlarger{\tau}_{n^{1/2-\delta}}(Y_1^{2n}) \rightarrow 0$ and the coordinates of $(Y_i^{(m)} - \mathbb{E}[Y_i^{(m)}], \  X_{i}^{(m)}-\mathbb{E}[X_i^{(m)}])$ are $\mathrm{subE}(K_1)$ random variables.   
\end{enumerate} 
\begin{enumerate}[label=B] 
\item $\beta_{n}(Y_1^{2n}) \rightarrow 0$.
\end{enumerate} 

Then, 
\begin{align}
d_{BL}(\mathcal{L}( \sqrt{n}(\hat{\beta}_{\hat{m},n} - \beta_{\hat{m},n}) \ | \ Y_1^n), \widetilde{\mathcal{L}}(  \sqrt{n}(\hat{\beta}_{\hat{m},n}- \beta_{\hat{m},n})  \ | \ Y_1^n) \xrightarrow{P} 0 
\end{align}
\end{theorem}

\begin{proof}
Again, we will will use the Delta Method.  First we will find $i_n$ such that Condition \ref{stability-low-dim} is satisfied. Again $i_n = n^{1/2 -\delta}$ for $\delta >0$ suffices by Example \ref{sample-mean-example} under the tightness assumption.    

 Our next step is to show $d_{BL}(\mathcal{L}(f(Y_{n+i_n +1}^{2n}) \ | \ Y_1^n), \widetilde{\mathcal{L}}(f(Y_{n+i_n+1}^{2n}) \ | \ Y_1^n) \xrightarrow{P} 0$.  Here, we will have to deviate from the stability theorem since the mapping is not Lipschitz.  However, may still partition on $\{\hat{m} = m\}$ and use the law of total probability, effectively freezing this element of randomness.  Let $\mathcal{S}^{(m)}$ denote the indices corresponding to the coordinates included in model $m$. For notational convenience, assume that the coordinates are centered; the average expected value will be the ``parameter'' in the Delta Method.  Consider the following set: 
\begin{align}
A_n = \left\{ \max_{n+ i_n + 1 \leq j \leq 2n} \ \max_{k \in \mathcal{S}^{(m)}} | Y_{jk}|  \leq  M_n  \right\}
\end{align}  
% We will prove the sub-Gaussian case.  The other cases are analogous, but require a different bound for the maximal inequality in (\ref{maximal-inequality-subgaussian}).  
% Start by defining the set:
% \begin{align}
% A_n = \left\{ \max_{n+ i_n 1 \leq j \leq 2n} \max_{1 \leq k \leq p} | Y_{jk}|  \leq  M_n  \right\}
% \end{align}
Let $Q_n$ and $R_n$ denote the probability measures associated with  $\mathcal{L}(Y_{n+i_n+1}^{2n})$ and $\mathcal{L}(Y_{n+i_n+1}^{2n} \ | \  Y_1^n)$, respectively.  We have that:   

\begin{align}
\begin{split}
\sup_{h \in BL_1} \int h (\sqrt{n} \psi_n) \  d (Q_n- R_n)  &\leq \sup_{h \in BL_1} \int_{A_n} h (\sqrt{n} \psi_n) \  d (Q_n- R_n) 
\\ \ \ \ \  & + \sup_{h \in BL_1} \int_{A_n^c} h (\sqrt{n}  \psi_n) \  d (Q_n- R_n) 
\\ &= \mathbf{I} + \mathbf{II}
\end{split}
\end{align} 

We will start by bounding $\mathbf{II}$.  Notice that:
\begin{align}
\mathbf{II} \leq Q_n(A_n^c) + R_n(A_n^c)
\end{align}
Therefore, by Markov's inequality, it follows that:
\begin{align}
P(\mathbf{II} > \epsilon ) \leq \frac{2 Q_n(A_n^c)}{\epsilon}
\end{align}

Since the selected coordinates are $\mathrm{subExp}(K_1)$, we have the following maximal inequality:
\begin{align}
P(A_n^c) \leq  2 (n-i_n) \cdot p^{(m)} \cdot \exp{\left(\frac{1-M_n}{K_1}\right)} 
\end{align}

 Note that $M_n$ can be chosen to grow very slowly while still ensuring that $\mathbf{II} \xrightarrow{P} 0$.  However, since we will need to take several factors into consideration, we will wait until the end of the proof to choose tuning parameters. 

For $\mathbf{I}$, notice that the local Lipschitz constant on $A_n$ is $ 2 \sqrt{n} M_n$.  We will now replace the original function class, which had poor global Lipschitz properties, with a more well-behaved function class, and show that the difference between these classes is small asymptotically.    We can view the current function class $\mathcal{H}_n$ as:
\begin{align}
\mathcal{H}_n = \left\{ h: \mathbb{R}^{p^{(m)}} \mapsto \mathbb{R} \ \biggr\rvert \  \norm{h}_\infty \leq 1, \ ||h||_L \leq  2 \sqrt{n} M_n \ \forall y \in A_n, \ \ h(y) = 0 \  \forall y \in A_n^c  \right\}
\end{align}

Since $\sup_{h \in \mathcal{H}_n} \int h d (Q-R) \leq 1 < \infty$, it follows that for any $\epsilon_n > 0$, there exists $h_n^* \in \mathcal{H}_n $ such that $\sup_{h \in \mathcal{H}_n} \int h d (Q_n - R_n) \leq \int h_n^* d (Q_n - R_n) + \epsilon_n$.\footnote{It may be the case that the supremum is attained. In this case, $\epsilon_n$ can be taken to be equal to 0.  Otherwise, it should be noted that we are implicitly invoking the Axiom of Choice when choosing a corresponding $h_n^*$.}  We will choose an arbitrary positive sequence $\epsilon_n \downarrow 0$.  

In addition, let $\bar{\mathcal{H}}_n$ denote the class of Lipschitz extensions of $h_n^*$. Furthermore, let $\mathcal{G}_n$ denote the following function class:
\begin{align}
\mathcal{G}_n = \{ g:  ||g||_L \leq 2 \sqrt{n} M_n\}
\end{align}
Notice that all Lipschitz extensions of $h_n^*$ are contained in $\mathcal{G}_n$.  If there exists $h_n \in \bar{\mathcal{H}}_n$ that satisfies $\int_{A_n^c} h_n d (Q_n - R_n) \geq 0$, it follows that the supremum over $\mathcal{G}_n$ is greater than or equal to the supremum over $\mathcal{H}_n$; otherwise a correction is needed.  Therefore, it follows that:     
\begin{align}
\label{lipschitz-extension-correction-step}
\begin{split}
\mathbf{I} &= \sup_{h \in \mathcal{H}_n} \int h \  d (Q_n- R_n)
\\ & \leq \int h_n^* \ d (Q_n- R_n) + \epsilon_n
\\ &\leq \sup_{g \in \mathcal{G}_n} \int g \ d (Q_n- R_n) + \left[0  \vee \inf_{h \in \bar{\mathcal{H}}_n} \int_{A_n^c} h \ d (R_n - Q_n) \right] + \epsilon_n
\\ & \leq 2 \sqrt{n} M_n d_W(Q_n, R_n) + 0 \vee \mathbf{III} + \epsilon_n
\end{split}
\end{align} 
By Squeeze Theorem, it is sufficient to find a sequence of non-negative upper bounds such that $\mathbf{III} \leq L_n \xrightarrow{P} 0$ to show $0 \vee \mathbf{III} \xrightarrow{P} 0$.  Since $R_n$ is a conditional probability measure, the function $h_n^*$ depends on $\omega$ (that is, different functions may be chosen for different values of $Y_1^n$).    However, we will suppress dependence on $\omega$ for notational convenience; we will see that the bounds will hold over all $\omega$, so not much is lost in doing so.

The first term will be controlled by the $\mathlarger{\tau}$-coefficient.  Notice that the class of Lipschitz functions that take as input the selected coordinates are a subset of all Lipschitz functions mapping from $\mathbb{R}^p$ to $\mathbb{R}$. Since the argument in $\mathbf{III}$ is an infimum, we are free to choose any $\bar{h}_n$ in the class.  We will choose the maximal Lipschitz extension:
\begin{align}
\bar{h}_n(y) = \inf_{a \in A_n} \left[ h_n^*(a) + 2 \sqrt{n} M_n \norm{y - a}_\infty \right]
\end{align}.

Again, we take the sup-norm to be $\max_{k \in \mathcal{S}^{(m)}}|X_k|$ since we only need to consider Lipschitz functions that take as input the included covariates.  We will now split $\mathbf{III}$ into two parts and we may take $L_n = \mathbf{IV} + \mathbf{V}$:
\begin{align}
\begin{split}
\mathbf{III} &\leq  \int_{A_n^c} \bar{h}_n(y) \ d (R_n - Q_n)
\\ & \leq \left| \int_{A_n^c} \bar{h}_n \ d R_n \right|+ \left| \int_{A_n^c} \bar{h}_n \ d Q_n \right|
\\ & \leq \int_{A_n^c} |\bar{h}_n(y)| d Q_n + \int_{A_n^c} |\bar{h}_n(y)| d R_n
\\ &= \mathbf{IV} + \mathbf{V} 
\end{split}
\end{align}  
% \begin{align}
% \begin{split}
% \mathbf{III} &\leq \int_{A_n^c} |\bar{h}_n(y)| d Q_n - \int_{A_n^c} |\bar{h}_n(y)| d R_n
% \\&= \mathbf{IV} + \mathbf{V} 
% \end{split}
% \end{align}
We will start by showing that $\mathbf{IV} \rightarrow 0$ for all $\omega \in \Omega$.  Since $ \norm{h_n^*}_\infty \leq 1$, it follows that $\bar{h}_n >0$ on $B_n^c$, where $B_n$ is given by:
\begin{align}
B_n = \left \{\max_{n+ i_n + 1 \leq j \leq 2n} \ \max_{k \in \mathcal{S}^{(m)}} | Y_{jk}|  \leq  M_n+ \frac{1}{2 \sqrt{n}M_n}  \right\}
\end{align} 
On the portion of $A_n^c$ where $\bar{h}_n >0$, we will drop the absolute value; it will turn out that the other part is negligible. Let $C_n = A_n^c \cap B_n$ and $D_n = A_n^c \cap B_n^c$ .  Then, it follows that:   
\begin{align}
\begin{split}
\mathbf{IV} & \leq \int_{C_n} |\bar{h}_n| \ d Q_n + \int_{D_n} \left[ h_n^*(\mathbf{0}) + 2 \sqrt{n} M_n \norm{y - \mathbf{0}}_\infty \right] \ d Q_n
\\ &\leq \int_{C_n} |\bar{h}_n| \ d Q_n + Q_n(D_n) +  \int_{D_n} 2 \sqrt{n} M_n \norm{y}_\infty \ d Q_n 
\\ &\leq \int_{A_n^c} \left[ \  2 \ \mathbbm{1}(B_n) +  2 \sqrt{n} M_n \norm{y}_\infty \mathbbm{1}(B_n^c) \ \right] \ dQ_n + Q_n(A^c)
\\ & \equiv \int_{A_n^c} \phi \  d Q_n + Q_n(A_n^c) 
\\ &= \mathbf{VI} +  Q_n(A_n^c) 
\end{split}
\end{align}
We will proceed by bounding $\mathbf{VI}$.  Our strategy now is to use the tail formula for the expectation. A crucial part in this step is breaking the integral into two parts, where one is controlled by $P(A_n^c)$ and the other by $\sqrt{n}M_n || y||_\infty$.  Choosing the breakpoint $c_n$ properly will allow us to send both terms to 0 (it will also allow us to restrict our attention to $\phi$ defined on $B_n^c$).
\begin{align}
\begin{split}
\mathbf{VI} &= \int_0^\infty P( \{\phi > s \} \cap A_n^c) \ ds
\\ &\leq \int_0^{c_n} P(A_n^c) \ ds + \int_{c_n}^\infty P(\phi >s ) \ ds
\\ &\leq c_n P(A_n^c) + \int_{c_n}^\infty P( 2 \sqrt{n} M_n \norm{y}_\infty >s ) \ ds
\\ &\leq c_n P(A_n^c) + 2 p^{(m)} (n-i_n) \int_{c_n}^\infty  \exp{\left( \frac{1-s} {2 \sqrt{n} M_n K_1} \right)} \ ds 
\\ &  \leq c_n P(A_n^c) +  4 p^{(m)} \ n^{3/2} M_n K_1 \exp{\left( \frac{1-c_n} {2 \sqrt{n} M_n K_1} \right)}
\\ & = \mathbf{VII} + \mathbf{VIII}
\end{split}
\end{align}
Now we will choose tuning parameters.  For some $\delta > 0$ and $n$ large enough, let:
\begin{align}
M_n = K_1\log(2n^{3/2 + 2 \delta}p^{(m)}), \ \ \  c_n =  \ n^{1/2 + \delta}, 
\end{align}
It may be verified that these choices of parameters ensure that:
\begin{align}
\mathbf{II} = P(A_n^c) \leq \frac{1}{n^{1/2 + 2\delta}} \rightarrow 0
\end{align}
\begin{align}
\mathbf{VII} = c_n P(A_n^c) \leq \frac{1}{n^\delta} \rightarrow 0
\end{align}
\begin{align}
\mathbf{VIII} =  4 p^{(m)} \  n^{3/2} M_n K_1 \exp{\left( \frac{1-n^{\delta}} {2 M_n K_1} \right)} \rightarrow 0
\end{align}
We can therefore conclude that $\mathbf{IV} \rightarrow 0$ for all $\omega \in \Omega$.  What remains is bounding $\mathbf{V}$.  Using similar reasoning to $\mathbf{IV}$, we can argue that:
\begin{align}
\mathbf{V} \leq \int_{A_n^c} \left[ \  2 \ \mathbbm{1}(C_n) + 2 \sqrt{n} M_n \norm{y}_\infty \mathbbm{1}(B_n^c) \ \right] \ d R_n + R_n(A_n^c)
\end{align}
Since this is a non-negative, fixed function for $\omega \in \Omega$, by Markov's inequality and law of iterated expectations, we can conclude that:
\begin{align}
P( \mathbf{V} > \epsilon ) \leq \frac{\int_{A_n^c} \left[ \  2 \ \mathbbm{1}(B_n) + 2 \sqrt{n} M_n \norm{y}_\infty \mathbbm{1}(B_n^c) \ \right] \ d Q_n + Q_n(A_n^c)}{\epsilon} 
\end{align}
But in bounding $\mathbf{IV}$, we concluded that the term in the numerator goes to zero.  Therefore $\mathbf{V} \xrightarrow{P} 0$ and the claim follows after an application of the Delta Method.  
\end{proof}

% \begin{remark}
% In (\ref{lipschitz-extension-correction-step}), it seems that it would often be the case that there exists a Lipschitz extension such that $ \int_{A_n^c} h \ d (R_n - Q_n) \geq 0$.  However, our strategy is technically more straightforward than establishing conditions under which such a Lipschitz extension exists for all $n > N$ for some $N$ for a collection of $\omega$ with probability tending to 1.  We reduce the problem to one about the tail behavior of the unconditional distribution; the alternative strategy is more involved.  
% \end{remark}

  \subsection{Bootstrap Inference Under $\theta$-Dependence and Non-stationarity}
 In this section, we consider the block-multiplier bootstrap for conducting inference after sample splitting.  The block-multiplier bootstrap is a generalization of the wild bootstrap, first proposed in \citet{wu-jackknife-resampling} and further examined in \citet{mammen-wild-bootstrap-regression} under independence and later by \citet{shao-dependent-wild} for stationary strongly mixing processes.  The block multiplier bootstrap is related to the block bootstrap, introduced by \citet{Kunsch-bootstrap-for-stationary-obs}, but differs in a fundamental way that allows bootstrap validity even under possible non-stationarity.     

In the block bootstrap, the original data is also split into blocks, but on each bootstrap iteration, blocks are sampled according to some distribution and stitched together to form a bootstrapped time series. When there is non-stationarity, disrupting the ordering of the blocks may be undesirable.  While some block bootstrap results exist under non-stationarity (see, for example, \citet{local-block-bootstrap} and \citet{synowiecki-periodic-bootstrap}), a certain degree of local stationarity or periodicity seems to be required for variants of the block bootstrap.  

In the theorem below, we establish the validity of the block multiplier bootstrap for the sample mean under $\theta$-dependence. This result for the sample mean in turn implies validity for a wide range of functionals, including regression coefficients. Again, since the $\theta$-dependence is weaker than $\tau$-dependence, this result implies validity of inference after sample splitting for $\tau$-dependent processes. We also require mean-stationarity, but in Remark \ref{non-stationary-mean} we discuss how one may relax this assumption.  While certain bootstrap results have been extended to $\Psi$-weak dependence measures (see, for example, \citet{hwang-psi-dependence-bootstrap}), to our knowledge, our result is the first one that does not require stationarity.         

We will also impose a uniform sub-exponential condition.  Our proof strategy involves controlling a covariance term between products of random variables, which necessitates light tails\footnote{It is also possible to control this covariance term with assumptions on the L\`{e}vy Concentration function as in \citet{weak-dependence}.}.  The tail condition can alternatively be stated in terms of higher moments, but since we assume sub-exponential tails for sample-splitting validity, we will use this condition here for consistency.  However, we would like to note that the condition on the block length is as weak as those for block bootstraps under mixing conditions; see \citet{Lahiri-resampling-for-dependent} for an overview.  We will leave the possibility of sharpening the tail assumptions in the theorem to future work.      

 Before stating our theorem, we will introduce some notation. Let $b_n$ denote the block length and let the number of blocks be given by $m_n = \lfloor n/b_n \rfloor$. Furthermore,  let the indices corresponding to block $l$ be given by $\mathcal{K}_l$, where $\mathcal{K}_l = \{(l-1)b+1, \ldots, lb \}$ for $1 \leq l \leq m$ and $\mathcal{K}_{m+1} = \{mb+1, \ldots, n\}$. Let $B_l = \sum_{ i \in \mathcal{K}_l} (Y_i - \bar{Y}_n)$ denote the de-meaned $l$th block sum, where $\bar{Y}_n = \frac{1}{n} \sum_{i=1}^n Y_{n,i}$.

Consider the following normalized sums:
\begin{align}
S_n^X = \frac{1}{\sqrt{n}} \sum_{i=1}^n (Y_i - \mathbb{E}[Y_i]) 
\end{align}

 The block-multiplier version of $S_n^X$, denoted as $W_n^Y$, is given by:
\begin{align}
W_n^Y = \frac{1}{\sqrt{n}}\sum_{l=1}^{m} e_l B_l 
\end{align} 
where $e_l \sim N(0,1)$ are generated independently from $Y_1^{2n}$.  A Gaussian multiplier is chosen here for simplicity; however, other multipliers may yield better properties beyond the first-order correctness established here; see, for example, \citet{mammen-wild-bootstrap-regression} or \citet{deng-zhang-beyond-gaussian-approx}. In the theorem below, we will establish conditions for when the block multiplier sum weakly approaches the normalized sum.
\begin{theorem}[Validity of Block Multiplier Bootstrap for $\theta$-dependent Processes]
\label{bootstrap-validity-tau-dependence}
Let $\{Y_{n,i}\}$ be a sequence in $\mathbb{R}^d$ with constant mean such that \\  $\sup_{i \in \mathbb{N}}$ $\max_{1 \leq j \leq d} \norm{Y_{ij} - \mathbb{E}[Y_{ij}]}_{\psi_1} < \mu $ and $\theta_r(Y_1^n) = O(r^{-\theta})$ as $n \rightarrow \infty$, where $\theta > 4$.   Further suppose that the block length satisfies $b_n \rightarrow \infty$  and $b_n= o(n)$.  Then, 
\begin{align}
\mathcal{L}(W_n^Y \ | \ Y_1^n) {\overset{wa(P)} \iff }  \mathcal{L}(S_n^X)  
\end{align}
\end{theorem}
\begin{proof}
% , $m = \lfloor n/q_n \rfloor$.  Consider the collection of indices $\{ \mathcal{K}_l\}_{l=1}^{m+1}$, where   Let $B_l$ denote the corresponding block sum:
% \begin{align}
% B_l = \sum_{ i \in \mathcal{K}_l} Y_i
% \end{align}  
By triangle inequality, we have that, for $Z_n \sim N(0,\Sigma_n)$:
\begin{align} 
|\mathbb{E}[f(W_n^Y) \ | \ Y_1^n ] - \mathbb{E}[f(S_n^X)]| \leq |\mathbb{E}[f(W_n^Y) \ | \ Y_1^n ] - \mathbb{E}[f(Z_n)]| + | \mathbb{E}[f(S_n^X) - f(Z_n)]| 
\end{align}
Under our assumptions, the second term converges to $0$ by the non-stationary central limit theorem (Theorem \ref{non-stationary-clt}). Therefore, bootstrap consistency will follow if we show that:
\begin{align}
\mathcal{L}\left(W_n^Y \ | \ Y_1^n \right) {\overset{wa(P)} \iff }  \mathcal{L}\left(Z_n \right) 
\end{align}
Let $X_i = Y_i - \mathbb{E}[Y_i]$, $\bar{X}_n = \frac{1}{n} \sum_{i=1}^n X_i$ and $C_l = \sum_{i \in \mathcal{K}_l} X_i$.  For convenience assume $n = m_n b_n$. Notice that $\mathcal{L}(W_n^Y \ | \ Y_1^n) = N(0, \Sigma_n^W)$, where $\Sigma_n^W$ is given by:
\begin{align}
\begin{split}
\Sigma_n^W &= \frac{1}{n}\sum_{l=1}^{m} \left[\sum_{i \in \mathcal{K}_l} (Y_i - \bar{Y}_n ) \sum_{i \in \mathcal{K}_l} (Y_i - \bar{Y}_n )^T\right]
\\ &= \frac{1}{n} \sum_{l=1}^{m}  \left[\sum_{i \in \mathcal{K}_l} (X_i - \bar{X}_n ) \sum_{i \in \mathcal{K}_l} (X_i - \bar{X}_n )^T\right]
% \\ &= \frac{1}{n}\sum_{l=1}^{m+1} \sum_{i \in \mathcal{K}_l} X_i)(\sum_{i \in \mathcal{K}_l} X_i)^T - b_n \bar{X}_n \bar{X}_n^T
\\ &= \frac{1}{n} \sum_{l=1}^m C_l C_l^T  - b_n \bar{X}_n \bar{X}_n^T
% \Sigma_n^W &= \frac{1}{n}\sum_{l=1}^{m+1} B_l B_l^T
% \\ &= \frac{1}{n}\sum_{l=1}^{m+1}  - \bar{Y}_n\bar{Y}_n^T 
% \\ &= \left(\frac{1}{n}\sum_{l=1}^{m+1} B_l B_l^T -  \mathbb{E}[\bar{Y}_n] \mathbb{E}[\bar{Y}_n^T]\right) +  \left( \mathbb{E}[\bar{Y}_n] \mathbb{E}[\bar{Y}_n^T] - \bar{Y}_n\bar{Y}_n^T \right)
% \\ &= \mathbf{I} +  \mathbf{II}  
% \\ &=  \left(\frac{1}{n} \sum_{l=1}^{m+1} B_l B_l^T - \mathbb{E}[\bar{Y}\bar{Y}^T ] +  \mathbb{E}[\bar{Y}\bar{Y}^T ] -\bar{Y}\bar{Y}^T)  
\end{split}
\end{align}
Since both are mean 0 Gaussian random variables, by the conditional version of Proposition \ref{gaussian-characteristic-difference-multivariate}, it suffices to show that $||\Sigma_n^W - \Sigma_n||_\infty \xrightarrow{P} 0$.  Now, by Proposition \ref{theta-dependent-covariance-lemma}, our assumption on the $\theta$-coefficient implies that:
\begin{align}
c^\prime = \sup_{i \in \mathbb{N}} \max_{1 \leq j \leq d} \mathrm{Var}(X_{ij}) + 2 \sum_{l \in \mathbb{N}} \sup_{i \in \mathbb{N}} \max_{j,k} \left |\text{Cov}\left(X_{ij},X_{(i+l)k} \right) \right| < \infty
\end{align}
 Therefore, since $b_n \rightarrow \infty$ and $b_n = o(n)$, by Lemma \ref{multivariate-big-block-lemma}, we have that $||\frac{1}{n}\sum_{l=1}^m\mathbb{E}[C_l C_l^T] - \Sigma_n||_\infty \rightarrow 0$. It follows that we can further reduce the problem to showing, for any $j,k \in\ \{1, \ldots, d \}$ that $(\frac{1}{n}\sum_{l=1}^m (C_l C_l^T - \mathbb{E}[C_l C_l^T]))_{jk} \xrightarrow{P} 0 $ and $b_n (\bar{X}_n \bar{X}_n^T)_{jk} \xrightarrow{P} 0 $.  We will begin by showing the latter.  

Since $\sqrt{n} \bar{X}_n  = O_P(1)$ by central limit theorem (Theorem \ref{multivariate-clt}) and $b_n = o(n)$, it follows that $\sqrt{b_n} \bar{X}_n = o_P(1)$.  Therefore, by Continuous Mapping Theorem, it follows that this term converges to $0$ in probability.

%   In the previous display, notice that we may express $\Sigma_n^W$ as:
% \begin{align}
% \begin{split}
% \Sigma_n^W &= \left(\frac{1}{n} \sum_{l=1}^{m+1} (B_l B_l^T - \mathbb{E}[B_l B_l^T]) +
% \\ & + \left(\mathbb{E}[B_l B_l^T]) - \mathbb{E}[\bar{Y}\bar{Y}^T ] \right) +  \mathbb{E}[\bar{Y}\bar{Y}^T ] -\bar{Y}\bar{Y}^T)
% \\ &= \mathbf{I} +  \mathbf{II}   
% \end{split}
% \end{align}
% Written in this way, $\mathbf{I}$ is now centered and $\mathbf{II}$ may be controlled .  We will start by showing $\mathbf{II} \xrightarrow{P} 0$.  Let $ $ Notice that, by Proposition \ref{theta-preservation}, we have that  

% Notice that by Proposition \ref{theta-dependent-covariance-lemma}, our assumption on the $\theta$-coefficient implies that:
% \begin{align}
% c^\prime = \sup_{i \in \mathbb{N}} \max_{1 \leq j \leq d} \mathrm{Var}(Y_{ij}^2 + 2 \sum_{l \in \mathbb{N}} \sup_{i \in \mathbb{N}} \max_{j,k} \left |\text{Cov}\left(Y_{ij},Y_{(i+l)k} \right) \right| < \infty
% \end{align}
%  By Lemma \ref{multivariate-big-block-lemma}, we then have that $||E(\Sigma_n^W) - \Sigma_n||_\infty \rightarrow 0$.  Therefore, we can further reduce the problem to showing, for any $j,k \in\ \{1, \ldots, d \}$ and any $\epsilon > 0$: 
% \begin{align}
% P\left(\left|\Sigma_{n,jk}^W - E(\Sigma_{n,jk}^W) \right| > \epsilon \right) \rightarrow 0
% \end{align}
 % by central limit theorem and, and the function $f(x,y) = xy$ is continuous, the Continuous Mapping Theorem (Proposition \ref{weakly-approaching-in-prob}) yields that $\mathbf{II}_{jk} \xrightarrow{P} 0$.   

Now, let $U_{l,jk} =  [C_l C_l^T]_{jk}$.
 % By Proposition \ref{theta-preservation-blocked-sum},  $\theta_r(\bar{B}_1^{m+1}) \leq \theta_{rq}(Y_1^{2n})$ and by Proposition \ref{theta-preservation} $\theta_s(U_{1,jk}^{m+1}) \preceq \theta_s(\bar{B}_1^{m+1})^{\frac{1}{2}}$.  
Our strategy now is to use Chebychev's inequality and bound $\text{Var}(\frac{1}{n}\sum_{l=1}^{m+1} U_{ljk})$.  To this end, define the terms:
\begin{align}
\mathbf{A}_n = \max_{1 \leq l \leq m_n} \mathrm{Var}(U_{l, jk})  \ \ \text{and} \ \ \mathbf{B}_n =2 \ \sum_{r=1}^{m_n}  \max_{1 \leq l \leq m_n} \left |\text{Cov}\left(U_{l,jk},U_{(l+r),jk} \right) \right|  
\end{align}
Although it is not explicit in the notation above, note that $U_{l,jk}$, which is a functional of the blocked sum, changes with $n$. 
We will begin by bounding $\mathbf{A}_n$.  Notice that:
\begin{align}
\begin{split}
\label{fourth-order-moment-bound}
\mathbf{A}_n &\leq  \max_{1 \leq l \leq m_n} \mathbb{E} \left[ \sum_{ (s_1,t_1), \ldots, (s_4,t_4) \in \mathcal{S}_{n,l}} \prod_{u=1}^4 X_{s_u, t_u} \right] 
% \\ & \leq b_n^4 \  \mu^{1/4}
\end{split}
\end{align}
where $\mathcal{S}_{n,l}$ is the set of vectors of length 4 consisting of ordered pairs $(s_u, t_u)$ such that $\ s_1 \leq \ldots \leq s_4$, $\{s_1, \ldots, s_4\} \subset \mathcal{K}_l$, and $|\{u \ | \ t_u =j\}| =|\{u \ | \ t_u =k\}| =2$.
Following Theorem 4.1 of \citet{weak-dependence}, we may bound this term using an analog of the Marcinkiewicz-Zygmund inequality. Define the term:
\begin{align}
C_{n,r,4} = \max_{(s_1,t_1), \ldots, (s_4,t_4) \in \mathcal{S}_{n}} \left| \mathrm{Cov}(X_{s_1,t_1} \cdots X_{s_q,t_q}, X_{s_{q+1},t_{q+1}} \cdots X_{s_4,t_4} ) \right|
\end{align}
where $\mathcal{S}_n$ is set satisfying $\{s_1, \ldots, s_4 \}$ such that $ s_1 \leq \cdots \leq s_4 \leq n$ and $q$ and $r$ satisfy $s_{q+1} - s_{q} \geq r$ and $q \in \{1, 2, 3 \}$ and $t_1,\ldots, t_4 \in \{1,\ldots , d\}$.  By Proposition \ref{covariance-products}, we certainly have that $C_{n,r,4} = O(r^{-3})$, and analogous to \citet{weak-dependence}, we may conclude, for some $C_4$ not depending on $n$, that $\mathbf{A}_n \leq C_4 b_n^{2}$.

For $\mathbf{B}_n$, we will bound the cases $r=1$  and $2 \leq r \leq m_n$ separately.  For $r=1$, we will bound $\mathbb{E}[U_{l,jk}U_{l+1,jk}]$ and $\mathbb{E}[U_{l,jk}] \mathbb{E}[U_{(l+1),jk}]$ separately.  For the former, notice that:
\begin{align}
\max_{1 \leq l \leq m_n}  \mathbb{E}[U_{l,jk},U_{(l+1),jk}] =  \max_{1 \leq l \leq m_n}  \mathbb{E} \left[ \sum_{ (s_1,t_1), \ldots, (s_4,t_4) \in \mathcal{S}_{n,l,1}} \prod_{u=1}^4 X_{s_u, t_u} \right] 
% \\ & \leq b_n^4 \  \mu^{1/4}
\end{align}
where $\mathcal{S}_{n,l,1}$ is the set of vectors of length 4 consisting of ordered pairs $(s_u, t_u)$ such that $\ s_1 \leq \ldots \leq s_4$, $\{s_1, s_2\} \subset \mathcal{K}_l$, $\{s_3, s_4\} \subset \mathcal{K}_{l+1}$, and $\{t_1, t_2 \} = \{t_3, t_4 \} = \{j,k \}$.  We may bound this term with $C_4 b_n^{2}$ analogous to $\mathbf{A}_n$.  

For $\mathbb{E}[U_{l,jk}] \ \mathbb{E}[U_{(l+r),jk}]$, notice that:
\begin{align}
\left| \ \mathbb{E}[U_{l,jk}] \ \mathbb{E}[U_{(l+r),jk}] \ \right| & \leq \left( b_n c^{\prime} \right)^2
\end{align}

Now we will derive a bound for the sum of the terms greater than 1.  Let $\alpha_r = (r-1)b_n$.  Observe that, for $r \geq 2$:
\begin{align} 
\max_{1 \leq l \leq m_n} \left|\mathrm{Cov}(U_{l,jk},U_{(l+r),jk})\right| \leq 4 b_n^3 \sum_{v=1}^{b_n} C_{n,(\alpha_r+v),4} 
% \\ & \leq b_n^4 \  \mu^{1/4}
\end{align}
Now, since $C_{n,r,4} = O(r^{-3})$ and $\sum_{x=1}^\infty \frac{1}{x^3}$ is summable, we have that, for some finite $C$ and $C^\prime$:
\begin{align}
\begin{split}
m_n \mathbf{B_n} &\leq 2 m_n C_4 b_n^2  + 2m_n\left( b_n c^{\prime} \right)^2 + 4 m_n b_n^4 \sum_{l=1}^{m_n-1}C (lb_n)^{-3} 
\\ &\leq 2 C_4 n b_n + 2\left(c^{\prime}\right)^2 n b_n +  C^\prime n
\end{split}
\end{align}

Now, applying Chebychev's inequality, we have that, for some $C < \infty$:
\begin{align}
\begin{split}
P\left(\left|\frac{1}{n}\sum_{l=1}^{m} (U_{ljk}- \mathbb{E}[U_{l,jk}]) \right| > \epsilon \right) &\leq \frac{ \text{Var}\left(\sum\limits_{l=1}^{m} U_{ljk}\right)}{(n \epsilon)^2}
\\ & \leq \frac{m_n(\mathbf{A}_n + \mathbf{B}_n)}{(n \epsilon)^2}
\\ & \leq \frac{C b_n}{n \epsilon^2}
% \\  & \leq \frac{ m_n \sup_{i \in \mathbb{N}} b^4 \max_{1 \leq j \leq d} EY_{ij}^4}{n^2 \epsilon^2} + \frac{2 m_n \sum_{l \in \mathbb{N}} \sup_{i \in \mathbb{N}} \max_{j,k} \left |\text{Cov}\left(X_{ij},X_{(i+l)k} \right) \right|}{n^2 \epsilon^2}
% % \\ &\leq \frac{C_n m^{1+\delta/2}}{(n\epsilon)^{2+\delta}}
\end{split}
\end{align}
 Choosing $b_n = o(n)$, we have that both terms converge to $0$ for all $\epsilon > 0$ and the result follows.  
\end{proof}
% \begin{remark}
% When $\{Y_{n,i}\}$ is centered, we are able to show that bootstrap consistency holds for $b_n = o(n)$. The key tool here is an adaptation of the Marcinkiewicz-Zygmund inequality in Theorem 4.1 of \citet{weak-dependence}. We leave the possibility of finding a weaker condition for $b_n$ in the non-centered case to future work.   
% \end{remark}
\begin{remark}
\label{non-stationary-mean}
An estimate of the mean function (trend) is necessary to estimate the variance with the bootstrap.  Here, the assumption of constant mean allows the use of the sample mean to estimate the mean function.  It would be an interesting future endeavor to study the effects of trend estimation on bootstrap validity under different assumptions.
\end{remark}

\begin{remark}
A similar theorem is possible for $\alpha$-mixing sequences.  One may replace the covariance inequality and non-stationary central limit theorem with the appropriate counterparts; however, it may be possible to impose weaker tail assumptions.  We do not pursue this direction here.  
\end{remark}

   \subsection{Sample Splitting and the Bootstrap for Regression Coefficients}
 We now put together results from the previous sections to derive conditions under which confidence intervals produced by the block multiplier bootstrap after sample splitting provide valid inference for the regression coefficient under a $\tau$-dependence condition.  For simplicity, suppose that a particular coefficient $\beta_{1,\hat{m},n} \in \mathbb{R}$ is of interest; this can easily be generalized to a confidence band or Bonferroni interval.  
In what follows, let $C_{\hat{m},n}^*$ correspond to the confidence set:
\begin{align}
C_{\hat{m},n}^* &= \left[ \hat{\beta}_{1,\hat{m},n} - z_{(1-\alpha)/2}\frac{\sigma_{\hat{m}}^*}{\sqrt{n}}, \ \hat{\beta}_{1,\hat{m},n} + z_{(1-\alpha)/2}\frac{\sigma_{\hat{m}}^*}{\sqrt{n}}   \right]
\end{align} 
where $z_{(1-\alpha)/2}$ is the $(1-\alpha)/2$-quantile of the Standard Normal and  $\sigma_{\hat{m}}^*$ is the standard deviation corresponding the bootstrapped regression coefficients.  In the regression coefficient case, the bootstrap distribution is approximated as follows.  Consider a large number $B$ of Monte Carlo iterations.  On the $j$th iteration of the bootstrap, generate $e_1, \ldots, e_{m+1} \sim N(0,1)$ and form the bootstrapped functional $\psi_{n, \hat{m}, j}^{*} = \psi_{n}^{(\hat{m})} + n^{-1/2} W_{n,j}^{(\hat{m})}$.  We may then construct a corresponding bootstrap realization by letting $\hat{\beta}_{1,\hat{m},n,j}^{*} = g^{(m)}(\psi_{n, \hat{m},j}^*)$.  The standard deviation $\sigma_{\hat{m}}^*$ may then be approximated by:
\begin{align}
\sigma_{\hat{m}}^* \approx \sqrt{\frac{1}{B} \sum_{j=1}^B \left( \sqrt{n}(\hat{\beta}_{1,\hat{m},n,j}^* - \hat{\beta}_{1,\hat{m},n})\right)^2}
\end{align} 
We are now ready to state the following theorem, which establishes validity of sample splitting together with the bootstrap.  The proof relies heavily on several results that we have established in previous theorems.    

\begin{theorem}[Valid Bootstrap Inference for Regression Coefficient after Sample Splitting]  Suppose $\{Y_{n,i}\}$ is a triangular array of random variables such that for each $n,$ $Y_{n,i}$ is a random variable in $\mathbb{R}^{p_n}$ and that the vector has constant mean. Consider a model selection procedure $\hat{m}$ that has the property $P(\hat{m} \in \mathcal{M}^*) \rightarrow 1$ for some $\mathcal{M}^*$ satisfying $|\mathcal{M}^*| < \infty$. For each $m \in \mathcal{M}^*$, suppose that $g^{(m)}(\cdot)$ is differentiable in a neighborhood of $\mathbb{E}[\psi_{n}^{(m)}]$.  Further, for each $m \in \mathcal{M}^*$, suppose that $\theta_r(Y_1^{2n}) = O(r^{-\theta})$ for some $\theta > 4$  the coordinates of $(Y_i^{(m)}, X_{i}^{(m)})$ are $\mathrm{subE}(K_1)$ random variables.
% \end{enumerate} 
% \begin{enumerate}[label=B] 
% \item $\beta_r(Y_1^n)) = O(r^{-(3 + \delta)})$ for some $\delta > 0$ and the coordinates of $(Y_i^{(m)}, X_{i}^{(m)})$ are $\mathrm{subE}(K_1)$ random variables  
% \end{enumerate}
Then,
\begin{align}
P\left( \beta_{1,\hat{m},n} \in C_{\hat{m},n}^* \right) = 1 - \alpha - o(1)
\end{align} 
\end{theorem}
\begin{proof}
We will establish that $\mathcal{L}(\sqrt{n}(\hat{\beta}_{1,\hat{m},n}^* - \hat{\beta}_{1,\hat{m},n}) \ | \ Y_1^{2n})$ weakly approaches $\widetilde{\mathcal{L}}(\sqrt{n}(\hat{\beta}_{1,\hat{m},n} - \beta_{1,\hat{m},n}) \ | \ Y_1^n)$ in probability; this is sufficient by triangle inequality and Theorem \ref{sub-exponential-sample-splitting}.  Again, by law of total probability, we have that:

 \begin{align}
\begin{split}
& \ \ \ P \left( d_{BL} \left[ \mathcal{L}\bigl(\sqrt{n}(\hat{\beta}_{1,\hat{m},n}^* - \hat{\beta}_{1,\hat{m},n} \ | \ Y_1^{2n}), \widetilde{\mathcal{L}}\bigl(\sqrt{n}(\hat{\beta}_{1,\hat{m},n} - \beta_{1,\hat{m},n}) \ | \ Y_1^n \bigr) \right] > \epsilon \right)
\\ & \leq \sum_{m \in \mathcal{M}^*}  P \left( d_{BL} \left[ \mathcal{L} \bigl(\sqrt{n}(\hat{\beta}_{1,m,n}^* - \hat{\beta}_{1,m,n} \ | \ Y_1^{2n} \bigr), \widetilde{\mathcal{L}}\bigl(\sqrt{n}(\hat{\beta}_{1,m,n} - \beta_{1,m,n})\bigr) \right] > \epsilon \right) 
\\ & \ \ \ \ \ \ + P(m \not\in \mathcal{M}^*)
% \\ & \leq \sum_{m^* \in \mathcal{M}^*} P \left\{ d_{BL} \left[ \mathcal{L}(\hat{\theta}_{m^*}(Y_{n+1}^{2n})\ | \ Y_1^n), \mathcal{L}(\hat{\theta}_{m^*}(\widetilde{Y}_{n+1}^{2n}) \right] > \epsilon \right\} + P(\hat{m}^* \not\in \mathcal{M}^* )
\end{split}
 \end{align}
 Therefore it suffices to establish validity for each fixed $m \in \mathcal{M}^*$.  Now observe that:
\begin{align}
\begin{split}
  \mathcal{L}\bigl(\sqrt{n}(\hat{\beta}_{1,m,n}^* - \hat{\beta}_{1,m,n}) \ | \ Y_1^{2n} \bigr) & {\overset{wa(P)} \iff} N(0,\nabla g(\psi_n^{(m)})^T \Sigma_n) 
  \\ &{\overset{wa(P)} \iff} N(0,\nabla g(\mathbb{E}[\psi_n^{(m)}])^T \Sigma_n)
 \end{split}
\end{align}
where the first line follows due Theorem \ref{bootstrap-validity-tau-dependence}, Proposition \ref{theta-preservation}, the Conditional Delta Method (Theorem \ref{delta-method}), and the second line follows from the Continuous Mapping Theorem (Proposition \ref{continuous-mapping}), Lemma \ref{gaussian-characteristic-difference-multivariate}, and the fact that $\psi_n^{(m)} -\mathbb{E}[\psi_n^{(m)}] \xrightarrow{P} 0$.

 Now, by the Conditional Delta Method and the Central Limit Theorem (Theorem \ref{multivariate-clt}), we also have that:
\begin{align}
 \widetilde{\mathcal{L}}\bigl(\sqrt{n}(\hat{\beta}_{1,m,n} - \beta_{1,m,n}) \bigr) {\overset{wa} \iff}  N(0,\nabla g(\mathbb{E}[\psi_n^{(m)}])^T \Sigma_n) 
\end{align}  

\end{proof}

 % \subsection{Preliminaries}
 % % \input{main_theorem_statement}
 % % \label{Main Theorem Statement}
 % \subsection{Proof of Main Theorem}
 % \label{Proof of Main Theorem}
 % \input{main_theorem_proof}

 % \subsection{Application of main theorem to $AR(p)$ models under weak assumptions}
% \input{ar_model_application}
%\label{sec:ar_model_application}

%\section{Simulation Study}
% \label{Simulation Study}
% \label{sec:simulations}
%\input{simulations}

\section{Discussion}
We have shown that, under appropriate assumptions on dependence and estimator, sample splitting remains asymptotically valid for time series data.  The procedure we consider involves dividing the data into two blocks and using the first half for model selection and the second half for inference.  While it is inherently appealing that the procedure lacks a tuning parameter, a thorough simulation study will be needed to compare its properties to a procedure in which a buffer is introduced to further reduce dependence between the selection and inference sets.  
Our arguments suggest that, asymptotically, the two procedures have similar properties, but it is not clear how different they are in finite samples.  It is likely that there is an inference-estimation trade-off in which a buffer leads to less spillover effects from selection at the cost of less accurate estimation.  It would be reassuring to verify that the procedure of leaving no buffer between the two datasets is comparable to a more complicated procedure in a variety of simulation settings.

 In addition, while our notion of sample splitting validity is asymptotic, we believe that it will often be the case that the error incurred by sample splitting a weakly dependent sequence will be of smaller order than the error incurred by using asymptotic/bootstrap approximations for the sampling distribution.  At best, a Berry-Esseen type bound leads to a Normal approximation with error $O(1/\sqrt{n})$ in the Kolmogorov distance. For a central limit theorem or bootstrap result to apply, we impose a dependence condition that leads to the bounded Lipschitz distance between the unconditional and conditional distributions of a normalized sum being of smaller order for appropriate choices of $i_n$.  This, of course, is complicated by the deletion stability condition, but is something that we believe we can address in more detail in future work.   

We would also like to mention that, while we succeed in establishing validity of sample splitting under weak assumptions, in particular, without assuming stationarity, fixed covariates or correct model specification, we do so by giving up on estimating more ambitious targets.  In particular, our parameters are data-dependent and are only defined as a functional of an expectation on the inference set.  Depending on the application, it may be of interest to introduce additional assumptions as appropriate to ensure that the target parameter is one of interest.

We also succeed in establishing a non-stationary Central Limit Theorem for $\theta$-dependent sequences by extending \citep{doukan-winteberger-invariance-principle}.  Based on this result, we also establish validity of the Gaussian multiplier bootstrap in the non-stationary setting.  Our approach is based on the Dependent Lindeberg method of \citet{dependent-lindeberg-method}, but we are able to generalize their result to the non-stationary setting by exploiting a phenomenon regarding the variance of a normalized sum of weakly dependent variables.  In ongoing work, we are working to extend our analysis to the high-dimensional regime for the maximum of a random vector in the sense of \citet{cck-clt-bootstrap-high-dimension} under analogous dependence measures.     

\paragraph{Acknowledgments} The author would like to thank Sangwon Hyun and Michael Spece for helpful discussions.
% \appendix
% \input{appendices}
\appendix
\section*{Appendices}
\addcontentsline{toc}{section}{Appendices}
\renewcommand{\thesubsection}{\Alph{subsection}}
\subsection{Properties of Weakly Approaching Sequences}
\label{weakly-approaching-properties}
For the reader's convenience, we collect some properties of weakly approaching sequences used in Section \ref{main-results}.  The first two propositions establish a transitivity property for tightness.  
\begin{proposition}[Tightness Preservation for Weakly Approaching Sequences, Lemma 5 in \citet{Belyaev-Sjostedt-de-Luna-weakly-approach}]
\label{tightness-preservation}
If $\{V_n\}_{n \geq 1}$ is tight and $\mathcal{L}(U_n) {\overset{wa} \iff } \mathcal{L}(V_n) $ as $n \rightarrow \infty$, then $\{U\}_{n \geq 1}$ is tight.  
\end{proposition}
\begin{proposition}[Tightness Preservation for Conditionally weakly Approaching Sequences, Lemma 6 in \citet{Belyaev-Sjostedt-de-Luna-weakly-approach}]
If $\{Y_n\}_{n \geq 1}$ is tight and $\mathcal{L}(X_n \ | \ \mathbb{Z}_n) {\overset{wa(P)} \iff } \mathcal{L}(Y_n) $ as $n \rightarrow \infty$, then $\{X\}_{n \geq 1}$ is uniformly tight in probability and hence tight.  
\end{proposition}
Below we will state Slutsky's Theorem, both in the unconditional and conditional cases.  The statements below are a combination of results from \citet{Belyaev-Sjostedt-de-Luna-weakly-approach} and \citet{Belyaev-Sjostedt-de-Luna-properties-weakly-approach} that are packaged together for conciseness.    
\begin{proposition}[Weakly Approaching Slutsky's Theorem]\label{slutksy}
Let $\{U_n, V_n, X_n \}_{n \geq 1}$ be a sequence of random variables defined on the same probability space.  Let $\{c_n\}_{n \in \mathbb{N}}$ be a uniformly bounded sequence of constants.  Further assume that $\{V_n\}_{n \geq 1}$ is tight, $X_n \xrightarrow{P} c_n$ as $n \rightarrow \infty$, and that $\mathcal{L}(U_n) {\overset{wa} \iff } \mathcal{L}(V_n) $ as $n \rightarrow \infty$.  Then, the following hold:
\begin{enumerate}
\item[(i).] $\mathcal{L}(U_n + X_n) {\overset{wa} \iff } \mathcal{L}(V_n + c_n)$ as $n \rightarrow \infty$
\item[(ii).] $\mathcal{L}(X_n^T U_n) {\overset{wa} \iff } \mathcal{L}(c_n^T V_n)$ as $n \rightarrow \infty$
\end{enumerate}
\end{proposition}
\begin{proposition}[Weakly Approaching Conditional Slutsky's Theorem]\label{conditional-slutksy}
Let $\{ U_n, V_n, X_n, \mathbb{W}_n \}_{n \geq 1}$ be defined on the same probability space and let $\{c_n\}_{n \in \mathbb{N}}$ be a uniformly bounded sequence of constants.  Further assume that $\{V_n\}_{n \geq 1}$ is tight, $X_n \xrightarrow{P} c_n$ as $n \rightarrow \infty$, and that $\mathcal{L}(U_n \ | \ \mathbb{W}_n) {\overset{wa(P)} \iff } \mathcal{L}(V_n) $ as $n \rightarrow \infty$.  Then, the following hold:
\begin{itemize}
\item[(i).] $\mathcal{L}(U_n + X_n \ | \ \mathbb{W}_n) {\overset{wa(P)} \iff } \mathcal{L}(V_n + c_n)$ as $n \rightarrow \infty$
\item[(ii).] $\mathcal{L}( X_n^TU_n \ | \ \mathbb{W}_n) {\overset{wa(P)} \iff } \mathcal{L}(c_n^TV_n)$ as $n \rightarrow \infty$
\end{itemize}  
\end{proposition}
We will state the following continuity theorem(s) for characteristic functions. Below, we will denote the characteristic function of the random variable $X$ with $\varphi_X(t) = E e^{itx}$.  
\begin{proposition}[Continuity Theorem for Weakly Approaching Laws in $\mathbb{R}^d$]
\label{weakly-approaching-continuity-theorem}
Let $\{U_n \}_{n \in \mathbb{N}}$ and $\{V_n \}_{n \in \mathbb{N}}$ be sequences of random variables taking values in values in $\mathbb{R}^d$. Further assume that $\{V_n\}$ is tight.  Then,
\begin{align}
\mathcal{L}(U_n)   {\overset{wa} \iff } \mathcal{L}(V_n) \ \text{ if and only if } \ \varphi_{U_n}(t) - \varphi_{V_n}(t) \rightarrow 0 \ \  \forall \ t \in \mathbb{R} 
\end{align}
\end{proposition}
\begin{proposition}[Conditional Continuity Theorem for Weakly Approaching Laws in $\mathbb{R}^d$]
\label{weakly-approaching-conditional-continuity-theorem}
Let $\{ U_n, V_n, X_n, \mathbb{W}_n \}_{n \geq 1}$ be sequences of random variables defined on the same probability space, where $U_n$ and $V_n$ take values in $\mathbb{R}^d$ and $\mathbb{W}_n \in \mathcal{W}_n$  Further assume that $\{V_n\}_{n \in \mathbb{N}}$ is tight.  Then,
\begin{align}
\mathcal{L}(U_n \ | \ \mathbb{W}_n )   {\overset{wa(P)} \iff } \mathcal{L}(V_n) \ \text{ if and only if } \ \varphi_{U_n}(t \ | \ \mathbb{W}_n) - \varphi_{V_n}(t) \xrightarrow{P} 0 \ \ \forall \ t \in \mathbb{R} 
\end{align}
\end{proposition}

Analogous to weak convergence theory, the  Cram\`{e}r-Wold Device is an immediate consequence of the Continuity Theorem for characteristic functions:
\begin{proposition}[Cram\`{e}r-Wold Device]
\label{weakly-approaching-cramer-wold}
Let $\{U_n \}_{n \in \mathbb{N}}$ and $\{V_n \}_{n \in \mathbb{N}}$ be sequences of random variables taking values in values in $\mathbb{R}^d$. Further assume that $\{V_n\}$ is tight.  Then,
\begin{align}
\mathcal{L}(U_n)   {\overset{wa} \iff } \mathcal{L}(V_n) \ \text{ if and only if } \ \mathcal{L}(\alpha^T U_n)   {\overset{wa} \iff } \mathcal{L}(\alpha^T V_n) \ \ \forall \ \alpha \in \mathbb{R}^d 
\end{align}
\end{proposition}
We will now prove Proposition \ref{metricizing-weakly-approaching-probability}.
%     
%\begin{proposition*}
%Let $d(\cdot)$ be some metric that metricizes weak convergence.  Assume that $\{ V_n\}_{n \in \mathbb{N}}$ is tight.  Then, we have the following:
%\begin{align}
%\begin{split}
%\mathcal{L}(U_n \ | \ \mathbb{W}_n)   {\overset{wa(P)} \iff } \mathcal{L}(V_n \ | \ \mathbb{W}_n) \ \ &\text{ if and only if } \ \  d(\mathcal{L}(U_n \ | \ \mathbb{W}_n),\mathcal{L}(V_n \ | \ \mathbb{W}_n) ) \xrightarrow{P} 0
%\end{split}
%\end{align}
%\end{proposition*}

\begin{proof}
Recall that $X_n \xrightarrow{P} X $ if and only if for every subsequence $X_{n_k}$, there exists a further subsequence $X_{n_{k_l}}$ such that $X_{n_{k_l}} \xrightarrow{a.s.} X$.  Now observe that, for all $\omega \in \Omega$ such that $d(\mathcal{L}(U_{n_{k_l}} \ | \ \mathbb{W}_{n_{k_l}}(\omega)),\mathcal{L}(V_{n_{k_l}} \ | \ \mathbb{W}_{n_{k_l}}(\omega)) ) \rightarrow 0$, by Proposition \ref{metricizing-weakly-approaching}, we also have that that $\mathcal{L}(U_{n_{k_l}} \ | \ \mathbb{W}_{n_{k_l}}(\omega))   {\overset{wa} \iff } \mathcal{L}(V_{n_{k_l}} \ | \ \mathbb{W}_{n_{k_l}}(\omega))$.  Now since the two coincide for every subsequence, the result follows.       
\end{proof}
We will state a result below that allows one to derive unconditional convergence results from conditional convergence.  One may prove this result with a variant of the Bounded Convergence Theorem.  

\begin{proposition}[Weakly Approaching Hoeffding's Trick, Lemma 4 in \citet{Belyaev-Sjostedt-de-Luna-weakly-approach}]
\label{weakly-approaching-hoeffding}
Let $\{ U_n, V_n, \mathbb{W}_n \}_{n \geq 1}$ be sequences of random variables defined on the same probability space, where $U_n$ and $V_n$ take values in $\mathbb{R}^d$ and $\mathbb{W}_n \in \mathcal{W}_n$.  Further assume that $\{V_n\}$ is tight. Then we have the following:
\begin{align}
\mathrm{If \ \ }\mathcal{L}(U_n \ | \ \mathbb{W}_n)   {\overset{wa(P)} \iff }  \mathcal{L}(V_n) \mathrm{ \ \ \ then \ \ \ }  \mathcal{L}(U_n)   {\overset{wa} \iff }  \mathcal{L}(V_n)
\end{align}
\end{proposition}
%
%\begin{proof}
%We will work with the definition directly and show that, for all $f \in C_b(\mathcal{X})$, $\mathbb{E}[f(U_n \ | \ \mathbb{W}_n] - \mathbb{E}[f(V_n)] %\xrightarrow{P} 0$ implies $\mathbb{E}[f(U_n) ] - \mathbb{E}[f(V_n)] \rightarrow 0$.

%Define the event:
%\begin{align}
%\mathcal{A}_{n} = \biggl\{ \bigl|\mathbb{E}[f(U_n)\ | \ \mathbb{W}_n] - \mathbb{E}[f(V_n)]\bigr| > \epsilon_n \biggr\}
%\end{align}  

%Since we have convergence in probability, we may construct a  sequence $\{\epsilon_n\}_{n \in \mathbb{N}}$ such that  $\epsilon_n \downarrow 0$, and 
%$P(\mathcal{A}_n) \rightarrow 0$.  Now, observe that:      
%\begin{align}
%\begin{split}
%\left|\mathbb{E}[f(U_n) ] - \mathbb{E}[f(V_n)] \right| &\leq  \int_{\mathcal{A}_{n}}
%\bigl|\mathbb{E}[ f(U_n) \ | \ \mathbb{W}_n] - \mathbb{E}[f(V_n)] \bigr| \ dP + 
%\int_{\mathcal{A}_{n}^c} \bigl|\mathbb{E}f(U_n) \ | \ \mathbb{W}_n] - \mathbb{E}[f(V_n) %\bigr| \ dP 
%\\ & \leq 2C P(\mathcal{A}_{n}) + \int \bigl|[\mathbb{E}f(U_n) \ | \ \mathbb{W}_n] - 
%\mathbb{E}[f(V_n)] \bigr| \cdot \mathbbm{1}(\mathcal{A}_{n}^c) \ dP
%\end{split}
%\end{align} 
%For $n$ large enough, the first term can be bounded by $\epsilon/2$, and the second term may bounded by $\epsilon/2$ by the Bounded Convergence %Theorem.  The result follows.      
%\end{proof}
%
Finally, we have the following analog to P\'{o}lya's Theorem for weakly approaching random variables.
\begin{proposition}[Weakly Approaching P\'{o}lya's Theorem, Lemma 9 in \citet{Belyaev-Sjostedt-de-Luna-weakly-approach}]
\label{weakly-approaching-polya-theorem}
Let $\{ U_n, V_n, \mathbb{W}_n \}_{n \geq 1}$ be sequences of random variables defined on the same probability space, where $U_n$ and $V_n$ take values in $\mathbb{R}^d$ and $\mathbb{W}_n \in \mathcal{W}_n$.  Further assume that $\{V_n\}$ is tight and $F_{V_n}$ is continuous. Then,
\begin{align}
\mathcal{L}(U_n \ | \ \mathbb{W}_n)   {\overset{wa(P)} \iff }  \mathcal{L}(V_n) \ \text{ if and only if } \ \sup_{t \in \mathbb{R}^d} \left| F_{U_n  |   \mathbb{W}_n }(t) - F_{V_n}(t) \right| \xrightarrow{P} 0 
\end{align}
\end{proposition}

\subsection{Non-stationary Central Limit Theorem under $\theta$-dependence: Overview and Preliminaries}
In this section, we discuss how existing proofs in the literature can be modified to allow a central limit theorem for non-stationary sequences for certain types of $\Psi$-weak dependence measures, specifically $\theta$-weak dependence.  Remarkably, it turns out that most of the existing proofs for the stationary case do not need to be modified very much; a few tricks here and there will do.  Part of the reason why we are able to derive a non-stationary central limit theorem may be attributed to the power of the Lindeberg method, which generalizes nicely to the weakly dependent setting.

We should take a moment to clarify what we mean by ``power'' in our context.  Another common technique for proving central limit theorems under dependence is the martingale approximation method after \cite{gordin-central-limit-theorem}.  Using this technique, \citet{Dedecker-doukhan-covariance-applications} derive a central limit theorem under a weaker condition on the $\theta$-coefficient than a version established with the Lindeberg method in \citet{doukan-winteberger-invariance-principle}. However, it appears that generalizing the martingale approximation to the non-stationary setting is more involved than adapting the Lindeberg method.  For recent work on the martingale approximation in non-stationary settings, see for example, \citet{wu-gaussian-approximation-multivariate} and \citet{functional-clt-nonstationary-martingale}. Stein's Method may also be considered under weak dependence, but dependency graph approaches are better suited for processes that admit an $m$-dependent approximation; see for example \citet{zhang-cheng-gaussian-approximation} and \citet{zhang-wu-gaussian-approximation}.  While an approximation of the characteristic function is also possible under weak dependence using Stein's Method (see Lemma 7.2 of \citet{weak-dependence}, after \citet{bolthausen-clt-stationary}) we do not pursue this direction here.    

Although only a few tweaks are needed, we would like to remark that our result is the first non-stationary central limit theorem for standard $\Psi$-weak dependence measures.  The closest result to ours is that of \citep{neumann-non-stationary-clt}. \citet{neumann-non-stationary-clt} also considers a dependent Lindeberg method, but interestingly, considers dependent rather than independent Gaussian random variables in the interpolation.  In some sense, this technique is more general than what we consider here, but we will show that this generality is not needed for many problems under mild conditions on dependence. The resulting conditions in the theorem are also not terms of standard $\Psi$-weak dependence measures.  In addition,  \cite{neumann-non-stationary-clt} assumes that the variance of the normalized sum converges; we make it very explicit when this holds and also state a version of the theorem without this condition.  For non-stationary central limit theorems under mixing conditions, see for example, \citet{rio-nonstationary-clt} and \citet{nonstationary-subsampling}.        

Below, we will give a brief list of components that are needed to establish a central limit theorem using the proof strategy of \citet{doukan-winteberger-invariance-principle}.  We will also note where we make modifications.
\begin{enumerate}
\item Establish a covariance inequality for $\theta$-dependent sequences - The covariance inequality is used in all future steps. We impose a uniform condition on $\mathbb{E}|X_i|^{2+\delta}$ to derive a covariance inequality that holds for non-stationary sequences under a $\theta$-weak dependence condition.
\item Establish a $2+\delta$-order moment bound on the blocked sum.  For stationary sequences, a well known bound on the variance of the sum $S_n = X_1 + \ldots + X_n$  is given by:
\begin{align}
\text{Var}(S_n) \leq  n \sum_{k \in \mathbb{Z}} |\text{Cov}(X_0, X_k)|
\end{align}
In the non-stationary case, a similar bound is possible; the bound is more conservative, but this turns out to be inconsequential.  To introduce additional flexibility, let $S_{n,j} = X_j + \ldots X_{n+j-1}$.  We have that:
\begin{align}
\label{nonstationary-bound-for-variance}
\text{Var}(S_{n,j}) \leq  n \cdot \left(\sup_{i \in \mathbb{N}} \mathrm{Var}(X_i) + 2\sum_{k \in \mathbb{N} }\sup_{i \in \mathbb{N}} |\text{Cov}(X_{i}, X_{i+k})| \right)
\end{align}
Furthermore, we can control this term using the covariance inequality above, which only assumes a uniform $2+\delta$ moment condition and a rate on the $\theta$-coefficient.  While the proof for the $2+\delta$-order moment bound on the blocked sum is non-trivial, it generalizes without much difficulty using this trick. 
\item Use the dependent Lindeberg method of \citet{dependent-lindeberg-method} to derive a Gaussian approximation for the ``big blocks''.  Here, it will suffice to plug in the covariance inequality and $2+\delta$-order moment bound from above. The modifications needed here are quite minor.
\item Show that the auxiliary terms are negligible.  Here, another result is needed, which will be detailed in Lemma \ref{variance-big-block-lemma}. 
 
\end{enumerate}  
For completeness, we will provide proofs or proof sketches with the modifications below.

\subsubsection{Covariance inequality for $\theta$-dependent sequences}
The following proof closely follows that of Lemma 4.2 in \citet{weak-dependence}.  
\begin{proposition}[Covariance inequality for $\theta$-dependent sequences] 
\label{theta-dependent-covariance-lemma} 
Let $X_1^n$ be vector of real-valued random variables such that $\max_{1\leq i \leq n} E|X_i|^m \leq \mu$ for some $m > 2$.  Then, for all $1 \leq i \leq n-k$, the following inequality holds:
\begin{align}
|\text{Cov}(X_{i}, X_{i+k})| \leq 8 \ \mu^{\frac{1}{m-1}} \ \theta_k(X_1^n)^{\frac{m-2}{m-1}}
\end{align}
\end{proposition}

\begin{proof}
We will proceed by truncation. Let $\overline{X}_i =f^{(M)}(X_i) - E f^{(M)}(X_i)$, where $f^{(M)}(x) = (x \wedge M) \vee (-M)$. Notice that $||\overline{X}_i||_\infty \leq 2M$ and $||f^{(M)}||_L = 1$, and $||\overline{X}_i||_m \leq 2 \mu^{1/m}$. 

Analogous to \citet{weak-dependence}, we will decompose the covariance as:
\begin{align}
\begin{split}
 \text{Cov}(X_i, X_{i+k}) &= \text{Cov}(\overline{X}_i, \overline{X}_{i+k}) + \text{Cov}\left( X_i - \mathbb{E}[X_i] - \overline{X}_i, \ X_{i+k} - \mathbb{E}[X_{i+k}] \right) 
 \\ & \ \ \ + \text{Cov}\left(\overline{X}_i, \ X_{i+k} - \mathbb{E}[X_{i+k}] - \overline{X}_{i+k} \right)
 \end{split}
 \end{align}
 Let $1/a + 1/m =1$.  Now we will use the following bound:
\begin{align}
\label{covariance-inequality-bound}
\begin{split}
|\text{Cov}(X_i, X_{i+k})|& \leq 2M \cdot \theta_k(X_1^n) \  + 2||\overline{X}_i||_m \ ||X_{i+k} - \mathbb{E}[X_{i+k}] - \overline{X}_{i+k}||_a 
\\ & \ \ \ + 2||X_{i+k}||_m \ ||X_i - \mathbb{E}[X_i] - \overline{X}_i||_a 
\\ & \leq 2M \cdot \theta_k(X_1^n) + 4 \left( \max_{1 \leq i \leq n} ||X_i||_m \cdot \max_{1 \leq j \leq n }||X_{j} - \mathbb{E}[X_j] -  \overline{X}_{j}||_a \right)
\end{split}     
\end{align}
For the second term, we have that:
\begin{align}
\label{truncation-norm}
\begin{split}
|| X_j - \mathbb{E}[X_j] - \overline{X}_j  ||_a & \leq ||X_j -f^{(M)}(X_j) ||_a + || \ \mathbb{E}[ f^{(M)}(X_j) - X_j] \ ||_a   
\end{split}
% \begin{split}
% E|X_{j} - f^{(M)}(X_{j})|^a &\leq E |X_j|^a \ \mathbbm{1}(|X_j| >M)
% \\ & \leq E |X_j|^m M^{a-m}
% \\ & \leq \mu M^{a-m}
% \end{split} 
\end{align} 
For the first part of (\ref{truncation-norm}), we have:
\begin{align}
\begin{split}
\mathbb{E}|X_{j} - f^{(M)}(X_{j})|^a &\leq \mathbb{E} |X_j|^a \ \mathbbm{1}(|X_j| >M)
\\ & \leq \mathbb{E} |X_j|^m M^{a-m}
\\ & \leq \mu M^{a-m}
\end{split} 
\end{align}
For the second term, we apply Jensen's inequality on $||\cdot||_a$ and derive the same bound as the first term.  Therefore,
\begin{align}
||X_{j} - \mathbb{E}[X_{j}]-\overline{X}_j||_a \leq 2 \mu^{\frac{1}{a}} M^{1-m/a} 
\end{align}
Notice that $m/a = m-1$, and therefore $1-m/a = 2-m$. Now plugging this in for (\ref{covariance-inequality-bound}), we have that,
\begin{align}
\begin{split}
|\text{Cov}(X_i, X_{i+k})| &\leq 8 \ (M \cdot \theta_k(X_1^n) + \mu M^{2-m})   
\end{split}
\end{align} 
Pick $M^{m-1} = \mu/\theta_k(X_1^n)$ and the result follows.     
\end{proof}

Below we will state a multivariate generalization of the covariance inequality.
\begin{proposition}[Covariance inequality for $\theta$-dependent sequences, Multivariate Version ] 
\label{theta-dependent-covariance-lemma} 
Let $X_1^n$ be vector of random vectors taking values in $\mathbb{R}^d$ and suppose that $\max_{i,j}E|X_{ij}|^m \leq \mu$ for some $m > 2$.  Then, for all $1 \leq i \leq n-k$, the following inequality holds:
\begin{align}
|\text{Cov}(X_{ij}, X_{(i+l)k})| \leq 8 \ \mu^{\frac{1}{m-1}} \ \theta_k(X_1^n)^{\frac{m-2}{m-1}}
\end{align}
\end{proposition}

\subsubsection{A $2 + \delta$-order moment bound for $\theta$-dependent sequences}   
Using the covariance inequality from the previous section, we will generalize Lemma 4.3 in \citet{weak-dependence} to the non-stationary case. Since the original proof is not short and we only need to tweak a few bounds to attain our result, we will provide just a sketch below, noting where the changes need to be made.  Recall that $S_{n,j} = X_j + \ldots + X_{j+n-1}$.      

\begin{proposition}[A $2 + \delta$-order moment bound for $\theta$-dependent sequences] 
\label{2-plus-delta-moment-bound}
Let $X_{j}^{n+j-1}$ be observations from a centered stochastic process $X$ with $\theta_r(X_{j}^{n+j-1})= O(r^{-\lambda})$ for $\lambda > 4 + \frac{1}{\zeta}$ as $n \rightarrow \infty$.  Suppose that $\sup_{i \in \mathbb{N}} \mathbb{E}|X_i|^{2+\zeta}$. Then for all $\delta \in (0, \zeta \wedge 1)$, there exists $C > 0$ such that:
\begin{align}
||S_{n,j}||_{2+\delta} \leq C \sqrt{n} 
\end{align}      
\end{proposition}

\noindent\textit{Proof sketch.}  The term $ c = \sum_{k \in \mathbb{Z}} |\text{Cov}(X_0, X_k)|$, appears in both the constant $C$ and various bounds in the original proof.  This term is related to the variance of the partial sum $S_n$ under a stationarity assumption. It can be shown that:
\begin{align}
\text{Var}(S_n) \leq n\sum_{k \in \mathbb{Z}} |\text{Cov}(X_0, X_k)|
\end{align}
By Proposition \ref{theta-dependent-covariance-lemma}, the $2+\zeta$ moment condition and our assumption on $\theta_X$ imply that $c < \infty$. Under possible non-stationarity, we can instead use the following term:
\begin{align}
c^\prime  \equiv \sup_{i \in \mathbb{N}} \mathrm{Var}(X_i) + 2 \sum_{k \in \mathbb{N}} \sup_{i \in \mathbb{N}} | \text{Cov} (X_i, X_{i+k} ) | 
\end{align} 
In fact, the value of the bound remains unchanged since the covariance inequality holds uniformly under both stationarity and the uniform moment condition. The conditions in our modification here are stated in terms of the sum starting at some fixed index $j$ so that we may derive a uniform bound that holds over for any block sum while still preserving the recursive nature of the argument.  

For the first term, notice that $\sup_{i \in \mathbb{N}} \mathbb{E} |X_i|^{2+\delta} \leq \mu $ implies that $\sup_{i \in \mathbb{N}} \mathbb{E} X_i^{2} \leq \mu \vee 1$.    We may plug $c^\prime$ for $c$ in bounding the terms $\mathbb{E}(1 + |A| + |B|)^\delta$, $\mathbb{E}|A| (1 + |A| + |B|)^\delta$, and $\mathbb{E}|B| (1 + |A| + |B|)^\delta$.  

To bound the terms $\mathbb{E} A^2(1+ |B|)^\delta$ and  $\mathbb{E} B^2(1+ |A|)^\delta$, we can proceed analogously to \citet{weak-dependence}.  However, we will slightly modify the bound for term $\mathbb{E} A^2(|B| - |\overline{B}|)$.  Let $\mathcal{B}$ denote the indices corresponding to the elements of $B$, given by $\mathcal{B} = \{n+q+1, \ldots, N \}$.  In the original argument, stationarity is invoked to use the bound:
\begin{align}
\label{moment-of-block-sum}
\begin{split}
\mathbb{E}|B|^m &=   n^m \ \mathbb{E} \left |  \frac{1}{n} \sum_{i \in \mathcal{B}} X_i \right|^m
\\  & \leq n^m \ \sum_{i \in \mathcal{B}} \frac{\mathbb{E} |X_i|^m}{n}
\\ & \leq n^m \mathbb{E} |X_0|^m
\end{split}
\end{align}
However, under possible non-stationarity we can use $\mu$ instead of $\mathbb{E} |X_0|^m$ in the last line. $\mathbb{E}|A|^m$ can be bounded analogously.  Note that this is a loose bound, but tools for establishing moment inequalities for non-integer powers are limited.  We may proceed analogously to \citet{weak-dependence} up until the bound for the covariance term $\text{Cov}(\overline{A}^2, (1+ |\overline{B}|)^\delta)$. Here, it will turn out that the bound is identical to the $\lambda$-dependent case, given by $n^3 M^2 \theta_X(q)$.  The rest of the argument is identical to \citet{weak-dependence}.

\subsubsection{$\theta$-preservation}
Although not needed in the proof the central limit theorem under non-stationarity, the following result will be needed to establish sample splitting validity of the sample mean of a quadratic term. Below we slightly extend Proposition 2.2 of \citet{weak-dependence}. Since the proof is omitted in the reference, we provide one here.  The proof strategy is similar to that of Proposition 2.1 in \citet{weak-dependence}.   
 \begin{proposition}[$\theta$-preservation]
 \label{theta-preservation}
Let $Y_1^{n}$ be a collection of $\mathbb{R}^d$-valued random variables. Let $m >1$.  Suppose there exists $0 < \mu <\infty$ such that:
\begin{align}
\max_{1\leq i \leq n} \max_{1\leq j \leq d} \mathbb{E}|Y_{ij}|^m \leq \mu
\end{align}
. Let $h: \mathbb{R}^d \mapsto \mathbb{R}^b$ such that $h(0) = 0$ and for $x,y \in \mathbb{R}^d$, there exists a in $[1,m)$ and $c>0$ such that:
\begin{align}
||h(x) - h(y)||_\infty \leq c ||x -y ||_\infty ( ||x||_\infty^{a-1}  + ||y||_\infty^{a-1} )
\end{align}
Let $U_i = h(Y_i)$.  Then, there exists a constant $C$ depending on $\mu$, $b$, $c$, and $d$, such that the $\theta$-coefficient of $U_1^n$ satisfies:
\begin{align}
\theta_r(U_1^n) \leq C \cdot \theta_r(Y_1^n)^{\frac{m-a}{m-1}} 
\end{align} 
\end{proposition}

\begin{proof}
Let $\mathcal{M}_{\mathbf{i}} = \sigma(Y_1 ,\dots Y_i)$ and $\mathbf{j} = (j_1, \ldots j_v)$.  Let $\mathcal{X} \subseteq \mathbb{R}^d$ and $\mathcal{Y}\subseteq \mathbb{R}^b$.  Let $\Lambda_1$ denote the class of 1-Lipschitz functions mapping from $\mathcal{Y}^{v}$ to $\mathbb{R}$.  Suppose $i + r \leq j_1$.   By definition of $\theta$-dependence, our goal is to bound:
\begin{align}
\label{theta-to-bound}
 \sup_{g \in \Lambda_1 } \frac{1}{v} \norm{ \mathbb{E}\left[ \ g( h(Y_{j_1}), \ldots, h(Y_{j_v}))  \ | \ \mathcal{M}_{\mathbf{i}} \ \right] - \mathbb{E}\left[g(h(Y_{j_1}), \ldots, h(Y_{j_v})) \ \right] }_1
\end{align} 

Denote $x^{(M)} = (x \wedge M) \vee (-M)$ for $x \in \mathbb{R}$.  For $(x_1, \ldots, x_d) \in \mathbb{R}^d$, analogously denote $x^{(M)}= (x_1^{(M)}, \ldots, x_d^{(M)})$.  We will work with fixed $g$ and derive a bound that holds uniformly over $\Lambda_1$.  We will proceed by expressing the composition $G$ as a mapping from $\mathcal{X}^u$ to $\mathbb{R}$. We will also define a version of $G$ with truncated arguments, denoted $G^{(M)}$: 
\begin{align}
G(Y_{\mathbf{j}}) = g(h(Y_{j_1}), \ldots, h(Y_{j_v})), \ \ \  G^{(M)}(Y_{\mathbf{j}}) = g(h(Y_{j_1}^{(M)}), \ldots, h(Y_{j_v}^{(M)}))
\end{align}  

By triangle inequality, the $L_1$ norm part of (\ref{theta-to-bound}) is bounded by:
\begin{align}
\begin{split}
 &  \norm{ \mathbb{E}[ \ G(Y_{\mathbf{j}}) - G^{(M)}(Y_{\mathbf{j}}) \ | \ \mathcal{M}_{\mathbf{i}} \ ] }_1  + \norm{ \mathbb{E}[ \ G(Y_{\mathbf{j}}) - G^{(M)}(Y_{\mathbf{j}}) \ ] }_1  + \norm{ \ \mathbb{E}[ G^{(M)}(Y_{\mathbf{j}})] - \mathbb{E}[ \ G^{(M)}(Y_{\mathbf{j}}) \ | \ \mathcal{M}_{\mathbf{i}} \ ]}_1 
 \\ & = \mathbf{I} + \mathbf{II} + \mathbf{III}
 \end{split}
\end{align}
For the first term, notice that:
\begin{align}
\begin{split}
\mathbf{I} &= \mathbb{E}\left| \ \mathbb{E}[ \ G(Y_{\mathbf{j}}) - G^{(M)}(Y_{\mathbf{j}}) \ | \ \mathcal{M}_{\mathbf{i}} \ ] \ \right| 
\\ & \leq \mathbb{E}\left[ \ \mathbb{E} \left[ \ |G(Y_{\mathbf{j}}) - G^{(M)}(Y_{\mathbf{j}})| \ \biggr|  \ \mathcal{M}_{\mathbf{i}}\ \right] \ \right]
\\ & \leq ||g||_L \sum_{k=1}^b \sum_{l=1}^v \mathbb{E}\left[ \ \mathbb{E}\left[ \ | h_k(Y_{j_l}) - h_k(Y_{j_l}^{(M)}) | \ \biggr|  \ \mathcal{M}_{\mathbf{i}} \ \right] \  \right]
\\ & \leq 2bc ||g||_L \sum_{l=1}^v  \mathbb{E} \left[ \ \norm{Y_{j_l}}_\infty^a \mathbbm{1}\left(\norm{Y_{j_l}}_\infty > M  \right)\  \right]
% \\ & \leq 2bcv ||g||_L  E E \left[ \ \norm{Y_{j_l}}_\infty^m \ M^{a-p} \ \bigr\rvert  \ \mathcal{M}_{\mathbf{i}} \ \right]
\\ & \leq 2bcv ||g||_L   M^{a-m} \ \mathbb{E} \left[ \max_{1 \leq w \leq d} |Y_{j_l,w}|^m \right]
\\ &\leq 2bcdv ||g||_L   \mu M^{a-m}
% \\ & \leq bcdv  \ ||g||_L \ \mu M^{a-p} 
\end{split}    
\end{align}

An analogous bound holds for $\mathbf{II}$. For $\mathbf{III}$, notice that $||G^{(M)}||_L \leq 2c M^{a-1} ||g||_L$.  Now since for $g \in \Lambda_1$, $||g||_L \leq 1$, we have that:
\begin{align}
 \theta_r(U_1^n) \leq 4bcd \ (\mu \vee 1) \ [M^{a-m} + M^{a-1} \theta_r(Y_1^n)]   
\end{align}
Optimizing over $M$, we see that $M = \theta_r(Y_1^n)^{1/1-m}$.  Plugging this in, we get the desired result.  
\end{proof}

\subsubsection{Covariance Inequality for Products Under $\theta$-dependence}

In this section we derive a bound for the following quantity under a $\theta$-dependence assumption:
\begin{align}
C_{n,r,v} = \max_{(s_1, \ldots s_v) \in \mathcal{S}_{n,r,v}} \left|\mathrm{Cov}(X_{s_1,t_1} \cdots X_{s_q,t_q}, X_{s_{q+1},t_{q+1}} \cdots X_{s_v,t_v}) \right|
\end{align}
where $\mathcal{S}_{n,r,v}$ is a collection of pairs $(s_1,t_1), \ldots (s_v,t_v)$ such that $1 \leq s_1 \leq \cdots s_q \leq s_q+r \leq s_{q+1} \leq \cdots \leq s_v \leq n$ and $q \in \{1,\ldots, v-1\}$, and $t_1, \ldots ,t_v \in \{1, \ldots, d\}$. Bounding this quantity will be crucial for deriving moment inequalities analogous to Theorem 4.1 of \citet{weak-dependence}.

\begin{proposition}
\label{covariance-products}
 Let $X_1^n$ be a collection of centered random vectors in $\mathbb{R}^d$.  For $m > 1$, suppose there exists $0 < \mu <\infty$ such that:
\begin{align}
\label{product-moment-condition}
\max_{1\leq i \leq n} \max_{1\leq j \leq d} \mathbb{E}|X_{ij}|^{m(v-1)} \leq \mu
\end{align}
Then, there exists constants depending on $\mu$ and $v$ such that:
\begin{align}
C_{n,r,v} \leq C \theta_r(X_1^n)^{\frac{m-v}{m}} 
\end{align}
\end{proposition} 
\begin{proof}

We will derive a bound on the covariance for any two random variables $U^{(r,v)}$ and $V^{(r,v)}$ formed by taking $U^{(r,v)} = X_{s_1, t_1} \cdots X_{s_q,t_q}$ and $V^{(r,v)} =X_{s_{q+1},t_{q+1}} \cdots X_{s_v,t_v}$ for any $(s_1, \ldots s_v) \in \mathcal{T}_{n,r,v}$ and $t_1, \ldots t_v \in \{1,\ldots, d\}$.  For notational convenience, we will suppress dependence on $r$,$v$ and $t$ below.

For the $\theta$-coefficient, the worst-case bound is attained when $q =1$. As before, we will proceed by truncation. Let $\overline{U} =f^{(M)}(U) - E f^{(M)}(U)$, where $f^{(M)}(x) = (x \wedge M) \vee (-M)$ and $\overline{V} =g^{(M)}(V) - E g^{(M)}(V)$, where $g^{(M)}: \mathbb{R}^{v-1} \mapsto \mathbb{R}$ performs the truncation component-wise for each element in the product.  We will express its input vector as $Z$.   Now we have that: 
\begin{align} 
\mathrm{Cov}(U,V) = \mathrm{Cov}(\overline{U}, \overline{V}) + \mathrm{Cov}(U - \overline{U},V) + \mathrm{Cov}(\overline{U},  V - \overline{V})
\end{align} 
For any $U$ and $V$ in $\mathcal{T}_{n,v,r}$, notice that:
\begin{align}
\mathrm{Cov}(\overline{U}, \overline{V}) \leq 2 \ (v-1) \ M^v \theta_r(X_1^n)
\end{align}
Now we will bound the latter terms. Analogous to Proposition \ref{theta-dependent-covariance-lemma}, we will use the Holder inequality to bound these terms.  Let $1/a + 1/m = 1$.  We have that:
\begin{align}
\begin{split}
|\text{Cov}(U, V)|& \leq  M^v \theta_r(X_1^n) + 2||\overline{U}||_m \ ||V - \overline{V}||_a + 2||V||_m \ ||U - \overline{U}||_a 
% \\ &= \mathbf{I} + \mathbf{II} +  \mathbf{III}
% \\ & \leq 2M \cdot \theta_k(X_1^n) + 4 \left( \max_{1 \leq i \leq n} ||X_i||_m \cdot \max_{1 \leq j \leq n }||X_{j} - \mathbb{E}[X_j] -  \overline{X}_{j}||_a \right)
\end{split}     
\end{align}
Now, to bound the second term, notice that the product function $h^{(v)}(x) = \prod_{i=1}^{v} x_i$ satisfies $|h(x) - h(y)| \leq ||x||_\infty^{v-1} + ||y||_\infty^{v-1} ||x - y ||_\infty$.  Therefore, we may proceed by using a truncation argument analogous to Proposition \ref{theta-preservation}. We have that, for some $C < \infty$ depending on $v$ and $\mu$:   
\begin{align}
\begin{split}
\norm{ V - \overline{V} }_a &= \left(\mathbb{E} |V- \overline{V} |^{a}\right)^{1/a}
\\& \leq \left(\mathbb{E}\norm{Z}_\infty^{(v-1)a} \mathbbm{1}(\norm{Z}_\infty >  M) \right)^{1/a}
\\ &\leq C \left(\mathbb{E}\norm{Z}_\infty^{m} M^{a(v-1) - m}\right)^{1/a}
\\ &\leq C M^{v-m}
\end{split}
\end{align}

Notice that the third term may be bounded by an analogous argument, but results in a lower order term.  However, the stronger moment condition given in (\ref{product-moment-condition}) is required to bound this term.  The optimal choice of $M$ is of the order $\theta_r^{-1/m}(X_1^n)$.  Plugging this in, we have that:
\begin{align}
\mathrm{Cov}(U^{(r,v)}, V^{(r,v)}) \leq C \cdot \theta_r(X_1^n)^{\frac{m-v}{m}} 
\end{align}       
% \begin{align}
% M^{v+1-m} + M^v \theta_r
% \\ M^{v-m} = M^{v-1} \theta_r
% \\ M^{1-m} = \theta_r 
% \\ M = \theta_r^{1/1-m}
% \end{align}

\end{proof}

\subsection{Dependent Lindeberg Method}
\label{dependent-lindeberg-method}
 In cases where there is temporal dependence that is diminishing with time, it turns out that the Lindeberg method of interpolation is convenient, as it preserves the past and future of the sequence.
 Our focus now is to state a Lindeberg lemma for dependent sequences, following \citet{dependent-lindeberg-method}. This lemma will be used to establish a central limit theorem for partial sums of the large blocks. Once this lemma is applied, one will additionally have to show that the distribution of the large blocks is close to that of the original sequence.  As mentioned above, this latter step will rely on the fact that asymptotically, the variance of the partial sum of a weakly dependent sequence is explained by the variance of the big blocks.  One will also need to control the higher moments of the blocked sums; previous results for the stationary case can be adapted to the non-stationary setting.

 \subsubsection{Notation}
 \label{Notation}

 We will proceed by introducing notation. We will largely adopt notation from \citet{cck-testing-moment-inequalities} for blocking.  We will decompose $\{X_i\}_{1 \leq i \leq n}$ into an alternating sequence of ``big blocks'' and ``small blocks''. Let $q_n >r_n$ be a sequence of positive integers. Asymptotically, we will require $q_n = o(n)$, $r_n = o(q_n)$, and $q_n , r_n \rightarrow \infty$. Let $I_1 = \{ 1, \ldots, q \}$, $J_1 = \{q+1, \ldots, q+r+1 \}, \ldots , I_m = \{ (m-1)(q+r), \ldots, (m-1)(q+r) +q \},$ $ J_m = \{ (m-1)(q+r) + q +1 , \ldots, m(q+r) \}$, $J_{m+1} = \{m(q+r) +1, \ldots n \}$.  In other words, $I_l$ $J_l$ denotes the indices corresponding to the $l$th ``big block'' and ``small block'', respectively, with $J_{m+1}$ containing the remainder. The number of blocks of a given type is given by $m_n = \lfloor n/(q_n +r_n) \rfloor$.  

Define partial sums corresponding to $l$th ``big block'' and ``small block'', respectively, with:
\begin{align}
S_l = \sum_{i \in I_l} X_i , \ \ S_l^\prime = \sum_{i \in J_l} X_i 
\end{align}
Define the corresponding normalized block sums as:
\begin{align}
U_n =  \frac{1}{\sqrt{n}} \sum_{l=1}^m S_l, \ \ \ V_n =  \frac{1}{\sqrt{n}} \sum_{l=1}^{m+1} S_l^\prime 
\end{align}

Let $(\check{Z}_1 \ldots \check{Z}_m)$ and $(\check{Z}_1^\prime \ldots \check{Z}_m^\prime)$ be Gaussian random variables, mutually independent and independent of $(Y_1, \ldots Y_m)$, with $\text{Var}(\check{Z}_l) = \text{Var} (S_l)$ and $\text{Var}(\check{Z}_l^\prime) = \text{Var} (S_l^\prime)$, respectively. 

Define the Gaussian analog of the normalized sum of the ``big blocks'' as:
\begin{align}
Z_n = \frac{1}{\sqrt{n}} \check{Z}_l
\end{align}

Furthermore, let $\{\widetilde{S}_l\}_{l=1}^m$, $\{ \widetilde{S}_l^\prime \}_{l=1}^{m}$ be a collection of mutually independent random vectors in $\mathbb{R}^{p}$ such that:
\begin{align}
\widetilde{S}_l \stackrel{d}{=} S_l, \ \  \widetilde{S}_l^\prime \stackrel{d}{=} S_l^\prime
\end{align}

To establish a central limit theorem for $S_n^X = \frac{1}{\sqrt{n}} \sum_{i=1}^n X_i$, it is sufficient to show that:
\begin{align}
\label{convergence-equation}
\mathbb{E}f(S_n^X) - \mathbb{E}f(\sigma_n N) \rightarrow 0 \ \ \  \forall f \in \mathcal{F}
\end{align}
where $\mathcal{F}$ is a convergence-determining class of functions, $N \sim N(0,1)$, and $\sigma_n^2 = \text{Var}(S_n^X)$.  With this variant of the Lindeberg Method, the function class will correspond to the characteristic function; that is:
\begin{align}
\mathcal{F} &= \{ f :  f(x)= e^{itx}, t \in \mathbb{R} \}
\end{align} 
The proof technique of \citet{doukan-winteberger-invariance-principle} consists of blocking and bounding the following terms:
\begin{align} 
|\Delta_n| = |\mathbb{E}f(S_n^X) -  \mathbb{E}f(U_n)|  + |\mathbb{E}f(U_n) - \mathbb{E}f(Z_n)| + |\mathbb{E}f(Z_n) - \mathbb{E}f(\sigma_n N )|    
\end{align} 
where the first and third term are auxiliary terms that arise from blocking.  The fact that the second term goes to $0$ can be established using the following Dependent Lindeberg Lemma, which we have stated in a less general form suitable for our purposes.  

\begin{proposition}[Dependent Lindeberg Lemma for Characteristic Functions, after Lemma 3 of \citet{dependent-lindeberg-method}]
\label{dependent-lindeberg-lemma}
\begin{align}
\Delta_n \leq T(m) + 2 \ |t|^{2+\delta} A(m)
\end{align} 
where:
 \begin{align}
 \begin{split}
 T(m) &= \sum_{l=1}^m |\text{Cov}(e^{it(S_1+ \ldots S_{l-1})/\sqrt{n}}, e^{itS_l/\sqrt{n}}) |
 \\ A(m) &= n^{-1-\delta/2} \sum_{l=1}^m \mathbb{E}|S_l|^{2+\delta}
 \end{split}
 \end{align} 
\end{proposition} 

The first two terms may be bounded analogously to \citep{doukan-winteberger-invariance-principle}.  The strategy for bounding the third term will be discussed below.

\subsubsection{On the Variance of a Normalized Sum of a Weakly Dependent Sequence}
Before we state the non-stationary central limit theorem, we will provide a justification as to why a Lindeberg method with independent rather than dependent Gaussian random variables is appropriate. As we have mentioned previously, it turns out that after blocking, the variance of the partial sum of a weakly dependent process is fully explained by the variance of the large blocks.  This phenomenon applies to processes that are ``weakly dependent'' in a very general sense; all that is required is a form of summability on the autocovariance function.   

In the lemma below, we will continue to use notation for blocking introduced above in Section \ref{Notation}.  We will also need some additional notation.  Let $\sigma_n^2 = \text{Var}(S_n^X)$ denote the variance corresponding to the normalized sum. Furthermore, consider the following variance terms:
\begin{align}
\label{block-variance-terms}
\widehat{\sigma}_n^2 = \text{Var}\left( \frac{1}{\sqrt{n}} \sum_{l=1}^m \widetilde{S}_l \right), \ \ \ \widetilde{\sigma}_n^2 = \text{Var}\left( \frac{1}{\sqrt{n}} \sum_{l=1}^m \widetilde{S}_l^\prime \right)
\end{align}
Recall that $\widetilde{S}_l \stackrel{d}{=} S_l$ for all $l = 1, \ldots, m$ and that $\{\widetilde{S}_l\}_{l=1}^m$ are constructed to be mutually independent.  Analogous statements can be made about $\{\widetilde{S}_l^\prime \}_{l=1}^m$.  Further recall $q_n = o(n)$, $r_n = o(q_n)$ and $ q_n,r_n \rightarrow \infty$. We have the following asymptotic result:

\begin{lemma}
\label{variance-big-block-lemma}
 Let $\{X_i\}_{i \in \mathbb{N}}$ be a  sequence of random variables such that:
 \begin{align}
 \label{variance-summability}
 c^\prime = \sup_{i \in \mathbb{N}} \text{Var}(X_i) + 2 \ \sum_{k \in \mathbb{N} }\sup_{i \in \mathbb{N}} |\text{Cov}(X_{i}, X_{i+k})|  < \infty
 \end{align} 
 Then,
\begin{align}
 \lim_{n \rightarrow \infty} \sigma_n^2 - \widehat{\sigma}_n^2 = 0 
\end{align}

%  $\lim_{n \rightarrow \infty} \frac{1}{n} \sum_{i=1}^n \text{Var}(X_i)$ exists
% \end{enumerate}
% Then, $\sigma^2 = \lim_{n \rightarrow \infty} \sigma_n^2$ exists, and furthermore:
% \begin{align}
%  \lim_{n \rightarrow \infty} \widehat{\sigma}_n^2 = \sigma^2
% \end{align}
\end{lemma}

\begin{proof}
Consider the following decomposition of $\sigma_n^2$:
% To this end, let $\widehat{\sigma}_n^2 = \text{Var}(Z_n)$ and $\widetilde{\sigma}_n^2 = \text{Var}(Z_n^\prime)$. We will express $\text{Var}(S_n^X)$ as follows:
\begin{align}
\sigma_n^2 =  \widehat{\sigma}_n^2 +  \widetilde{\sigma}_n^2 + \frac{1}{n} \sum_{(i,j) \in \mathcal{C}_n} \text{Cov}( X_i, X_j) +  \frac{1}{n} \sum_{(i,j) \in \mathcal{C}_n^\prime} \text{Cov}(X_i, X_j)
\end{align} 
Above, $\mathcal{C}_n$ is the collection of indices $(i,j)$ such that $|i-j| \leq r_n$ and $i$ and $j$ belong to different blocks and $\mathcal{C}_n^\prime$ is the collection of indices $(i,j)$ such that $|i-j| > r_n$ and $i$ and $j$ belong to different blocks. We will show that all terms except $\widehat{\sigma}_n^2$ are negligible.  For $\widetilde{\sigma}_n^2$, notice that, since $\#(J_1,\ldots, J_{m_n}) \leq (m_n+1)r_n +q_n$, we have that:
\begin{align}
\widetilde{\sigma}_n^2 \leq \frac{ (m_n+1)r_n + q_n}{n} \cdot c^\prime \rightarrow 0
\end{align} 

For the third term notice that:
\begin{align}
\frac{1}{n} \sum_{(i,j) \in \mathcal{C}_n} | \text{Cov}( X_i, X_j)| \leq \frac{2(m_n+1) r_n}{n} c^\prime \rightarrow 0
\end{align}   
For the last term, notice that:
\begin{align}
\frac{1}{n} \sum_{|i-j| > r_n} |\text{Cov}(X_i, X_j)| \leq 2 \sum_{k=r_n+1}^\infty \sup_{i \in \mathbb{N}}|\text{Cov}(X_i, X_{i+k})| \rightarrow 0
\end{align} 

Now, we may subtract $\hat{\sigma}_n^2$ from both sides.  Since the limit of the RHS is $0$, the result follows.
\end{proof}

\begin{remark}
It is straightforward to see that the above lemma generalizes to triangular arrays if instead one assumes that:
\begin{align} 
c^\prime = \sup_{n,i} \text{Var}(X_{n,i}) + 2 \ \sum_{k \in \mathbb{N} } \sup_{n,i} |\text{Cov}(X_{n,i}, X_{n,i+k})|  < \infty
\end{align}
\end{remark}

We would like to note that, if one wants to ensure that $\sigma^2 = \lim_{n \rightarrow \infty} \sigma_n^2$ exists, an additional assumption is needed. It is both sufficient and necessary to assume that $ \lim_{n \rightarrow \infty} \frac{1}{n}\sum_{i=1}^n \text{Var}(X_i)$ exists. The issue here is that non-negative sequences that are bounded above may not have a Ces\`{a}ro limit; this issue does not arise for stationary processes since the Ces\`{a}ro mean of the variance is trivial in this case. 

% To illustrate this fact, below we will give an example of a sequence consisting of just $0$s and $1$s that nevertheless fails to have a Ces\`{a}ro limit.

% \begin{example}[Ces\`{a}ro Mean of a Bounded Non-Negative Sequence Need Not Converge] Consider the binary sequence $s$ formed by concatenating $2^m$ consecutive 0s and $2^m$ consecutive 1s for each $m \in \mathbb{N}$.  That is, $s = 0011 \ 00001111 \ \ldots$ We will consider two subsequences of the Ces\`{a}ro means $c_n = \sum_{i=1}^n\frac{s_n}{n}$ and show that they do not converge to the same limit.   

% For the first subsequence, consider $n_k = \sum_{j=2}^{k+1} 2^{j}$.  Then, by construction, we have that half the entries are $1$, so $c_{n_k} = 1/2$.  Now, consider the subsequence along $m_k =n_k + 2^{k+1}$; this corresponds to the last observation in a run of $0$s.  It may be verified that along this subsequence, we have that:
% \begin{align}
% \begin{split}
% c_{m_k} &= \frac{\sum_{j=1}^{k} 2^j}{ \sum_{j=2}^{k+1} 2^{j} + 2^{k+1}}
% \\ &= \frac{2^{k+1} -2}{ 2^{k+2} -4 + 2^{k+1}} \rightarrow 1/3 
% \end{split}
% \end{align} 
% In the last line, we used the fact that $\sum_{k=0}^{n} 2^k = 2^{n+1} -1$.  Since the subsequences do not converge to the same limit, it follows that a limit does not exist.    
% \end{example}

While assuming that the Ces\`{a}ro-mean of the variance exists is mild and much weaker than assuming even second-order stationarity, it is not needed if we only require a weakly-approaching Normal approximation. Nonetheless, we will state and prove a lemma that establishes the sufficiency and necessity of this condition below:    

\begin{lemma}
\label{variance-convergence-lemma}
 Let $\{X_i\}_{i \in \mathbb{N}}$ be a  sequence of random variables such that (\ref{variance-summability}) holds.  Then,  $\sigma^2 = \lim_{n \rightarrow \infty} \sigma_n^2$ exists if and only if:
\begin{align}
\lim_{n \rightarrow \infty } \frac{1}{n} \sum_{i=1}^n \text{Var}(X_i) \text{ exists}
\end{align}

Furthermore, if $\sigma^2$ exists, then:
\begin{align}
 \lim_{n \rightarrow \infty} \widehat{\sigma}_n^2 = \sigma^2
\end{align}  
\end{lemma} 
 \begin{proof}
Following for example, \citet{Rio-asymptotic-theory-weakly-dependent}, we have that variance of the sum, given by $V_n = \text{Var}(S_n)$, may be written as:
\begin{align}
V_n = v_1 + \ldots + v_n
\end{align}
where $v_k = V_k - V_{k-1}$ represents an increment of the variance process, with $V_0 = 0$.  Notice that $v_k$ is given by:
\begin{align}
\begin{split}
v_k &= \text{Cov}(S_k, S_k) -  \text{Cov}(S_{k-1}, S_{k-1}) 
\\ &= \text{Var}(X_k) + 2 \ \sum_{j=1}^{k-1} \text{Cov}(X_j, X_k) \equiv a_k + b_k
\end{split}
\end{align}
We will proceed by showing that $b_k \equiv \sum_{j=1}^{k-1} b_{k,j}$ converges.  Notice that:
\begin{align}
|b_{k,j}| \leq \sup_{i \in \mathbb{N}} 2 \ |\text{Cov}(X_i, X_{i+k-j})| 
\end{align}   
Therefore, by the Comparison test, $ \lim_{k \rightarrow \infty} \sum_{j=1}^{k-1} |b_{k,j}|$ converges, and since absolute convergence implies convergence, $b_k \rightarrow b$ for some finite $b$.

Now, notice that:
\begin{align}
\label{variance-a-b}
\text{Var}(S_n^X) =  \frac{V_n}{n} =\sum_{k=1}^n \frac{a_k}{n} + \sum_{k=1}^n \frac{b_k}{n} = \mathbf{I} + \mathbf{II}
\end{align}

Since $b_k \rightarrow b$, we also have Ces\`{a}ro convergence; therefore, $\mathbf{II} \rightarrow b$.  Now, by assumption, we also have that $\mathbf{I}$ converges as well. Since $\mathbf{I}$ and $\mathbf{II}$ converge, we have that $\text{Var}(S_n^X)$ converges, establishing sufficiency.  

For necessity, we will argue by contradiction. Suppose that $ \sigma^2$ exists, but $ \frac{1}{n} \sum_{i=1}^n \text{Var}(X_i)$ does not converge.  Then, subtracting $\mathbf{II}$ from both sides of (\ref{variance-a-b}), we have that $ \frac{1}{n} \sum_{i=1}^n \text{Var}(X_i)$ converges, a contradiction.    

The second claim follows from the previous lemma.   
\end{proof}
% For completeness, we will provide a proof of an elementary lemma regarding convergence in Ces\`{a}ro means that is used above and in the previous chapter. 

% \begin{proposition}[Ces\`{a}ro Means]
% \label{cesaro-lemma}
% Suppose  $\{a_k\}_{k \geq 1}$ is a sequence taking values in a normed vector space such that $a_k \rightarrow a$.  Then $\frac{1}{n}\sum_{k=1}^n a_k \rightarrow a$.
% \end{proposition}

% \begin{proof}
% Since $a_k \rightarrow a$, for any $\epsilon >0$, there exists $K$ such that for all $k > K$, $\norm{a_k -a} < \epsilon/3$. Furthermore, let $\norm{\sum_{k=1}^K a_k} =M$.  Since $M < \infty$, there exists $N_1$ such that, for all $n > N_1$, $M/n < \epsilon/3$. Finally, there exists $N_2$ such that for all $n > N_2$, $K/n \norm{a} < \epsilon/3$.  Now, by triangle inequality, for all $n > \max(N_1, N_2, K+1)$, we have that:
% \begin{align}
% \begin{split}
% \norm{\frac{1}{n}\sum_{k=1}^n a_k - a}  &\leq \norm{ \frac{1}{n} \sum_{k=K+1}^n a_k - \frac{n-K}{n} a } + \norm{\frac{n-K}{n}a - a} +  \norm{ \frac{1}{n}\sum_{k=1}^K a_k }
% \\ &\leq \frac{1}{n-K} \sum_{k=K+1}^{n} \norm{a_k -a}  + \frac{K}{n}\norm{a} + \frac{M}{n} 
% \\ & < \epsilon/3 + \epsilon/3 + \epsilon/3
% \end{split}
% \end{align} 
% \end{proof}

\subsubsection{Some Additional Blocking Lemmas for the Variance}
In this section, we provide some additional blocking lemmas that will be used to derive results for the bootstrap.  These lemmas will allow one to to approximate the distribution of the block sum without needing to delete observations between blocks.

We will introduce some additional notation. Let $b_n = o(n)$ denote the block length, and let $m_n = \lfloor n/b_n \rfloor$ denote the number of blocks.  Let $K_l = \{ (m_n-1)b_n +1, \ldots, m_n b_n  \}$ denote the indices corresponding to the block sum $B_l$, given by:
\begin{align}
B_l = \sum_{i \in K_l} X_i
\end{align}
Furthermore, let $\{\widetilde{B}_l\}_{l=1}^{m+1}$ be a collection of mutually independent random vectors such that: $B_l \stackrel{d}{=} \widetilde{B}_l$.  Furthermore, denote:
\begin{align}
\check{\sigma}_n^2=  \text{Var}\left( \frac{1}{\sqrt{n}} \sum_{l=1}^{m+1} \widetilde{B}_l \right)
\end{align}
We have the following result:
\begin{lemma}
\label{no-gap-lemma}
 Let $\{X_i\}_{i \in \mathbb{N}}$ be a sequence of random variables such that (\ref{variance-summability}) holds.  Then, we have that: 
\begin{align}
\lim_{n \rightarrow \infty}  \sigma_n^2 - \check{\sigma}_n^2 = 0 
\end{align} 
\begin{proof}
We have that $|\sigma_n^2 - \check{\sigma}_n^2 | \leq  |\sigma_n^2 - \widehat{\sigma}_n^2| +  | \widehat{\sigma}_n^2 - \check{\sigma}_n^2|$. By Lemma \ref{variance-big-block-lemma}, the first term converges to 0.  It suffices to show that the second term converges to $0$.
Let $r_n = o(q_n)$ and $b_n = q_n + r_n$.  Subdivide each $K_l$ into $I_l$ and $J_l$ for $l = 1, \ldots, m$, where $I_l = \{(m-1)b + 1 , \ldots , (m-1)b +q \}$ and $J_l = \{(m-1)b +q +1  , \ldots,  mp \}$, as before. We do not subdivide $K_{m+1}$.   We have that:
\begin{align}
\check{\sigma}_n^2 = \widehat{\sigma}_n^2 + \widetilde{\sigma}_n^2 + \frac{1}{n} \sum_{(i,j) \in \widetilde{\mathcal{C}}_n } \text{Cov}(X_i, X_j)  
\end{align}
where $\widetilde{\mathcal{C}}_n$ correspond to the cross-terms between observations in $I_l$ and $J_l$.
Using analogous reasoning to Lemma \ref{variance-big-block-lemma}, we have that the last two terms converge to 0, and the result follows.   
\end{proof}
\end{lemma}

We can further generalize this result to the multivariate setting. Let $\Sigma_n$, $\widehat{\Sigma}_{n}$,  $\widetilde{\Sigma}_{n}$, and $\check{\Sigma}_n$ denote the multivariate counterparts to $\sigma^2$, $\widehat{\sigma}^2$, $\widetilde{\sigma}^2$, and $\check{\sigma}^2$ respectively.  We have the following lemma:    
\begin{lemma}
\label{multivariate-big-block-lemma}
Let $\{X_i\}_{i \in \mathbb{N}}$ be a sequence of random vectors taking values in $\mathbb{R}^d$. Suppose the following condition holds:
\begin{align}
c^\prime = \sup_{i \in \mathbb{N}} \max_{1 \leq j \leq d} \mathrm{Var}(X_{ij}) + 2 \sum_{l \in \mathbb{N}} \sup_{i \in \mathbb{N}} \max_{j,k} \left |\text{Cov}\left(X_{ij},X_{(i+l)k} \right) \right| < \infty 
\end{align}
\end{lemma}
Then,
\begin{align}
\lim_{n \rightarrow \infty}  || \Sigma_n - \widehat{\Sigma}_n ||_\infty = 0 \ \ \text{ and } \ \ \lim_{n \rightarrow \infty}  || \Sigma_n - \check{\Sigma}_n ||_\infty = 0
\end{align}

\begin{proof}
By Cauchy-Schwartz inequality, it is the case that $\sup_{i \in \mathbb{N}} \max_{j,k} E|X_{ij} X_{ik}| \leq \sup_{i \in \mathbb{N}} \max_{j} EX_{ij}^2$.  It suffices to show convergence for each $(j,k) \in \{1, \ldots d \}^2$.  It can readily be seen that the following factorization applies to $\Sigma_{n,jk}$:
\begin{align}
\Sigma_{n,jk } =  \widehat{\Sigma}_{n,jk} +  \widetilde{\Sigma}_{n,jk}+ \frac{1}{n} \sum_{(i,l) \in \mathcal{C}_n} \text{Cov}( X_{ij}, X_{lk}) +  \frac{1}{n} \sum_{(i,l) \in \mathcal{C}_n^\prime} \text{Cov}(X_{ij}, X_{lk})
\end{align}
Now, using analogous reasoning to Lemma \ref{variance-big-block-lemma}, the first claim follows.  The second claim follows from analogous reasoning to Lemma \ref{no-gap-lemma}.  
\end{proof}
\subsubsection{Weakly Approaching Gaussian Random Variables}

To control one of the auxiliary terms in the Dependent Lindeberg lemma, we will also use a result about weakly approaching sequences of Gaussian random variables. The proposition below establishes necessary and sufficient conditions for two sequences of Gaussian random variables to be weakly approaching.

\begin{proposition}[Weakly Approaching Gaussian Random Variables] 
\label{weakly-approaching-gaussian}
Let $\{ X_n\}_{n \in \mathbb{Z}}$ and $\{ Y_n\}_{n \in \mathbb{Z}}$ be sequences of Gaussian random variables with respective mean processes $\{\mu_n^X\}_{n \in \mathbb{N}}$ $\{\mu_n^Y\}_{n \in \mathbb{N}}$ and respective variance processes $\{v_n^X\}_{n \in \mathbb{N}}$ $\{v_n^Y\}_{n \in \mathbb{N}}$. Suppose that $\bar{\mu} = \sup_{n \in \mathbb{N}} \ | \mu_Y| < \infty $ and $ \bar{v} = \sup_{n \in \mathbb{N}} \ v_Y < \infty$.  Then $\mathcal{L}(X_n) {\overset{wa} \iff } \mathcal{L}(Y_n)$ if and only if:
\begin{align}
\mu_n^X - \mu_n^Y \rightarrow 0 \ \ \ \text{and} \ \ \ v_n^X - v_n^Y \rightarrow 0
\end{align} 
\end{proposition}
\begin{proof}
First, we will show that the conditions on the mean and variance process of $\{Y_n\}_{n \in \mathbb{N}}$ ensure tightness.  To see this, for any $ 0 < \epsilon < 1$ consider the compact set $C_\epsilon = [-\bar{\mu} - \bar{\sigma} \ z_{(1-\epsilon)/2}, \ \bar{\mu} + \bar{\sigma} \ z_{(1-\epsilon)/2}]$, where $z_{\epsilon/2}$ is the $(1-\epsilon)/2$ quantile of the Standard Normal.  It follows that $P(Y_n \not\in C_\epsilon) \leq \epsilon$ for all $n$.   

Now by the weakly approaching continuity theorem (Proposition \ref{weakly-approaching-continuity-theorem}), we have the convergence of the characteristic functions for all $t \in \mathbb{R}$ is a necessary and sufficient condition for two sequences to be weakly approaching.  Notice that: 
\begin{align}
\label{gaussian-characteristic-difference}
\begin{split}
 \varphi_{X_n}(t) - \varphi_{Y_n}(t) &= \text{exp}\left(i\mu_n^X t - \frac{1}{2} v_n^X t^2 \right) - \text{exp}\left(i\mu_n^Y t - \frac{1}{2} v_n^Y t^2 \right)
 % \\ &= \text{exp}\left(i\mu_n^X t - \frac{1}{2} v_n^X t^2 \right) \left[ 1 - \text{exp}\left(i(\mu_n^Y -\mu_n^X)  t - \frac{1}{2} (v_n^Y - v_n^X) t^2 \right) \right]
  \\ &= \text{exp}\left(i\mu_n^Y t -\frac{1}{2} v_n^Y t^2 \right) \left[ \text{exp}\left( \ it(\mu_n^X -\mu_n^Y) \ \right) \text{exp}\left(- \frac{1}{2} (v_n^X - v_n^Y) t^2 \right) -1  \right]
 \\ &= A \times ( B \times C - 1)
 \end{split}
\end{align} 
We can see that $\mu_n^X - \mu_n^Y \rightarrow 0$ and $v_n^X - v_n^Y \rightarrow 0$ are sufficient since $\exp(\cdot)$ is continuous and $| \varphi_{Y_n}(t) | \leq 1$.  To show necessity, we will argue via the contrapositive.  We will show that, if either condition fails to hold, then the difference of characteristic functions fails to go to $0$ for all $t \in \mathbb{R}$.

First, note that, if $\mu_n^X$ or $v_n^X$  is unbounded, then $\{X_n\}_{n \in \mathbb{N}}$ is not tight, contradicting  Proposition \ref{tightness-preservation}.  Therefore, we may treat $(\mu_n^X - \mu_n^Y, v_n^X - v_n^Y )$ as a bounded sequence.  Now, since $(\mu_n^X - \mu_n^Y, v_n^X - v_n^Y ) \not \rightarrow (0,0)$ but is bounded, there exists a convergent subsequence such that $(\mu_n^X - \mu_n^Y, v_n^X - v_n^Y ) \rightarrow (b,c)$, where $b \neq 0$ or $c \neq 0$.

Notice that, since $\bar{v} < \infty$, $\liminf |A| >0$.  Since $\exp(\cdot)$ is continuous, along this subsequence, we have that $B$ and $C$ converge.  For $A \times ( B \times C - 1) \rightarrow 0$ along this subsequence, it is necessary that $B \times C \rightarrow 1$.  This further implies that $\text{Im}(B) \rightarrow 0$, which implies that $\text{Re}(B) \rightarrow \pm 1$.  Then, it must be the case that $C \rightarrow 1$ and $B \rightarrow 1$ since $C \geq 0$.  Observe that for $C \rightarrow 1$, $ct^2 = 0$ and for $B \rightarrow 1$,  $ bt = 0 \text{ mod } 2\pi$.  If $b \neq 0$, we may choose $t$ such that $bt \neq 0 \text{ mod } 2 \pi$ and hence $B \not \rightarrow 1$.  Similarly, for $c \neq 0$, we do not have convergence for $t \neq 0$.  Therefore, if $(\mu_n^X - \mu_n^Y, v_n^X - v_n^Y ) \not \rightarrow (0,0)$, then there exists a subsequence not converging to $0$ for all $t \in \mathbb{R}$, and hence the sequence cannot converge to $0$ for all $t \in \mathbb{R}$.   
     
  % First consider $\mu_n^X - \mu_n^Y \not \rightarrow 0$; the other case will be analogous.  Note that, if $\mu_n^X$ or $v_n^X$  is unbounded, then $\{X_n\}_{n \in \mathbb{N}}$ is not tight, contradicting  Proposition \ref{tightness-preservation}. Now, suppose $\mu_n^X - \mu_n^Y$ is bounded but not converging to $0$.  Then, there exists a convergent subsequence converging to some $ c \neq 0$. This implies that, for this subsequence, $B \rightarrow b \neq 1$ for $t$ such that $tc \neq 0 \text{ mod } 2 \pi$.  
% Consider a further subsequence such that $v_n^X \geq v_n^Y$. If the intersection between $\{n_k\}_{k \in \mathbb{N}}$ and this set is finite, we may instead factor out $\varphi_{X_n}$ instead of $\varphi_{Y_n}$ for $A$, which would result in a different $b$, but still $b \neq 1$ for all $t \in \mathbb{R}$.  In this case, we would consider $v_n^X < v_n^Y$.  Now, notice that $ 0 \leq C \leq 1$ along this subsequence and $\liminf |A| > 0$ since $\bar{v} < \infty$.  Therefore it is necessary that $b = 1$ for $A \times ( B \times C - 1) \rightarrow 0$, but we have shown that there exists $t \in \mathbb{R}$ such that $b \neq 1$ along this subsequence.  Therefore, the sequence cannot converge to $0$ since it contains a subsequence not converging to $0$. 

% For $v_n^X - v_n^Y \not \rightarrow 0$, we can again restrict our attention to the bounded case, and use an analogous subsequence argument since $|B| = 1$.  The result follows.   
\end{proof}
\begin{remark}
Using the conditional version of the continuity theorem (Proposition \ref{weakly-approaching-conditional-continuity-theorem}), it is straightforward to show that this result extends to the conditional case.  
\end{remark}

Now we will state the following multivariate generalization of the above lemma.  Although not needed for our central limit theorem, it will be used to derive bootstrap results. 

\begin{proposition}[Weakly Approaching Gaussian Random Variables, Multivariate Version] 
\label{gaussian-characteristic-difference-multivariate}
Let $\{ X_n\}_{n \in \mathbb{Z}}$ and $\{ Y_n\}_{n \in \mathbb{Z}}$ be sequences of Gaussian random vectors, taking values in $\mathbb{R}^d$ with respective mean processes $\{\mu_n^X\}_{n \in \mathbb{N}}$ $\{\mu_n^Y\}_{n \in \mathbb{N}}$ and variance processes $\{\Sigma_n^X\}_{n \in \mathbb{N}}$ $\{\Sigma_n^Y\}_{n \in \mathbb{N}}$. Suppose that $\bar{\mu} = \sup_{n \in \mathbb{N}} \ || \mu_n^Y||_\infty < \infty $ and $ \bar{\sigma} = \sup_{n \in \mathbb{N}} \ ||\Sigma_n^Y||_\infty < \infty$.  Then $\mathcal{L}(X_n) {\overset{wa} \iff } \mathcal{L}(Y_n)$ if and only if:
\begin{align}
||\mu_n^X - \mu_n^Y||_\infty \rightarrow 0 \ \ \ \text{and} \ \ \ ||\Sigma_n^X - \Sigma_n^Y||_\infty \rightarrow 0
\end{align} 
\end{proposition}
\begin{proof}
Again, notice that conditions imposed are necessary and sufficient for $\{Y_n\}_{n \in \mathbb{N}}$ to be tight.   Analogous to the previous proposition,  our goal is to bound, for all $t \in \mathbb{R}^d$:
\begin{align}
\begin{split}
 \varphi_{X_n}(t) - \varphi_{Y_n}(t) &= \text{exp}\left(i (\mu_n^X)^T t - \frac{1}{2}  t^T\Sigma_n^X t \right) - \text{exp}\left(i (\mu_n^Y)^Tt - \frac{1}{2} t^T\Sigma_n^X t \right)
 % \\ &= \text{exp}\left(i\mu_n^X t - \frac{1}{2} v_n^X t^2 \right) \left[ 1 - \text{exp}\left(i(\mu_n^Y -\mu_n^X)  t - \frac{1}{2} (v_n^Y - v_n^X) t^2 \right) \right]
  \\ &= \text{exp}\left(i (\mu_n^Y)^Tt - \frac{1}{2} t^T\Sigma_n^X t \right) \left[ \text{exp}\left( \ (\mu_n^X -\mu_n^Y)^T t \ \right) \text{exp}\left(- \frac{1}{2} t^T(\Sigma_n^X - \Sigma_n^Y)t \right) -1  \right]
 \\ &= A \times ( B \times C - 1)
 \end{split}
% \begin{split}
%  \varphi_{\alpha^TX_n}(t) - \varphi_{\alpha^TY_n}(t) &= \text{exp}\left(i\alpha^T\mu_n^X t - \frac{1}{2} \alpha \Sigma_n^X \alpha^T t^2 \right) - \text{exp}\left(i \alpha^T\mu_n^Y t - \frac{1}{2} \alpha v_n^Y\alpha^T t^2 \right)
%  % \\ &= \text{exp}\left(i\mu_n^X t - \frac{1}{2} v_n^X t^2 \right) \left[ 1 - \text{exp}\left(i(\mu_n^Y -\mu_n^X)  t - \frac{1}{2} (v_n^Y - v_n^X) t^2 \right) \right]
%   \\ &= \text{exp}\left(i\alpha^T\mu_n^Y t -\frac{1}{2} \alpha^T \Sigma_n^Y \alpha t^2 \right) \left[ \text{exp}\left( \ it\alpha^T(\mu_n^X -\mu_n^Y) \ \right) \text{exp}\left(- \frac{1}{2} \alpha^T(\Sigma_n^X - \Sigma_n^Y)\alpha t^2 \right) -1  \right]
%  \\ &= A \times ( B \times C - 1)
%  \end{split}
\end{align} 
It is again clear that $||\mu_n^X - \mu_n^Y||_\infty \rightarrow 0$ and $||\Sigma_n^X - \Sigma_n^Y||_\infty \rightarrow 0$ are sufficient.  For necessity, notice that, given that these two terms converge, $t^T(\Sigma_n^X - \Sigma_n^Y)t \rightarrow 0$ and $(\mu_n^X -\mu_n^Y)^T t \rightarrow 0$ for all $t \in \mathbb{R}^d$ if and only if the sup-norm of each term converges to $0$ (for the covariance term, note that a limit of the difference along a subsequence is symmetric and cannot be non-trivially skew-symmetric).  Therefore, we can make an analogous subsequence argument and the result follows.
\end{proof}

\subsubsection{Statement and Proof of Central Limit Theorem}
 Now, we will state and prove our adaptation of the Central Limit Theorem of \citet{doukan-winteberger-invariance-principle} to the non-stationary setting. 

\begin{theorem}[Non-stationary Central Limit Theorem for $\theta$-dependent sequences]
\label{non-stationary-clt}
 Let $\{X_{n,i}  \ 1 \leq i \leq  n \}$ be a triangular array of zero mean random variables.  Suppose that the following conditions hold:
 \begin{enumerate}
 \item[A1] $\sup_{n,i} E|X_{n,i}|^m \leq \mu$ for some $m > 2 + \zeta$, where $\zeta > 0$
 \item[A2] $\theta_r(X_1^n) = O(r^{-\theta})$  as $n \rightarrow \infty$ for some $\theta > 4 + \frac{2}{\zeta}$.
 \end{enumerate}
   Then,
\begin{align}
\frac{1}{\sqrt{n}} \sum_{i=1}^n X_{n,i} {\overset{wa} \iff }  N(0, \sigma_n^2)
\end{align}
Furthermore, if $\lim_{n \rightarrow \infty}\frac{1}{n} \sum_{i=1}^n \text{Var}(X_{n,i})$ exists, then $ \sigma^2 = \lim_{n \rightarrow \infty}\sigma_{n,i}^2$ exists, and we have that:
\begin{align}
\frac{1}{\sqrt{n}} \sum_{i=1}^n X_{n,i} \rightsquigarrow N(0, \sigma^2)
\end{align}
\end{theorem} 
\begin{remark}
Central limit theorems are typically stated in alternative form, where the goal is to show that $S_n/\sigma_n \rightsquigarrow N(0,1)$. While this framework has its advantages, we would like to note that, with weakly dependent processes, it is possible that $\sigma_n \rightarrow 0$. The weakly approaching version is able to handle these degenerate cases without needing to impose some condition that implies $\sigma_n^2 >0$, or in the multivariate case, that $\Sigma_n$ is invertible.
\end{remark}

\begin{proof}
As discussed in the introduction of Appendix \ref{dependent-lindeberg-method}, we will consider a variant of the Lindeberg method with blocking where the convergence-determining class $\mathcal{F}$ corresponds to the characteristic function.  We will further assume rates for blocking. Let $q_n = \lfloor n^a \rfloor$ and $r_n= \lfloor n^b \rfloor$ where $a >b$. 

% Let $N \sim N(0,1)$ and for each $n$, define the Gaussian random variable:
% \begin{align}
% Z_n = \frac{1}{\sqrt{n}} \sum_{l=1}^m \check{Z}_m
% \end{align}
% Observe that $\text{Var}(Z_n) = \widehat{\sigma}_n^2$, which was defined in (\ref{block-variance-terms}).  Additionally, we will used the following terms to denote the normalized sums of the blocked sequences:
% \begin{align}
% U_n = \frac{1}{\sqrt{n}} \sum_{l=1}^m S_l \ , \ \ \  V_n = \frac{1}{\sqrt{n}} \sum_{l=1}^m S_l^\prime 
% \end{align}

%  Before applying the Lindeberg method, consider the following decomposition: 
% \begin{align}
% \mathbb{E}[f(S_n^X) - f(U_n)] + \mathbb{E}\left[f(U_n) - f(Z_n) \right] +  \mathbb{E}\left[f(Z_n) - f(\sigma_n N) \right] 
% \end{align} 
% The first and third terms are auxiliary terms, while the second term will be bounded using the Lindeberg method.

  We will start by bounding the first auxiliary term. Notice that a first-order Taylor expansion yields:
\begin{align}
\begin{split}
\mathbb{E}\left[f(S_n^X) - f(U_n) \right] &\leq \frac{||f^{\prime \prime}||_\infty}{2} \ \mathbb{E}| S_n^X - U_n|^2
\end{split}
\end{align}
Even under possible non-stationarity, the bound for the first term will be analogous to the stationary case, but we will provide it for completeness.  

Notice that $\mathbb{E}| S_n^X - U_n|^2 = \text{Var}(V_n)$. Since $\#(J_1,\ldots, J_m) \leq (m+1)r +q$, we have that:
\begin{align}
\text{Var}(V_n) &\leq \frac{ 2(m+1)r +q}{n} \cdot c^\prime \rightarrow 0
\end{align}

Now we will bound the other auxiliary term. Since $\sigma_n^2 \leq c^\prime$ and  $\mathbb{E}[N] = 0$, it follows that $\{ \sigma_n N \}_{n \in \mathbb{N}}$ is a tight sequence of Gaussian random variables.  By Proposition \ref{weakly-approaching-gaussian}, it follows that this auxiliary term converges to $0$ if $\sigma_n^2 - \widehat{\sigma}_n^2 \rightarrow 0$. Since $c^\prime < \infty$, by Lemma \ref{variance-big-block-lemma}, we indeed have $|\sigma_n^2 - \widehat{\sigma}_n^2| \rightarrow 0$. Therefore, this auxiliary term also converges to $0$.     

Finally, we will bound the main term with the Dependent Lindeberg Method.  By Proposition \ref{dependent-lindeberg-lemma}, it is sufficient to show that:
 \begin{align}
 T(m) = \sum_{j=1}^m | \text{Cov}(e^{it(S_1 + \ldots  + S_{j-1})/\sqrt{n}}, e^{it(S_j)/\sqrt{n}}) | \rightarrow 0 \  \text{ and } \  A(m) \rightarrow 0
 \end{align}
Analogous to \citet{weak-dependence}, notice that $|| e^{it(S_j)/ \sqrt{n}}||_L \leq |t|/{\sqrt{n}}$ and that $||e^{itz}||_\infty = 1$.  Therefore, it follows that:   
\begin{align}
 T(m) &\leq m q \frac{|t|}{\sqrt{n}} \theta_X(r)
 \leq |t| \ n^{1/2- b \theta} 
 \end{align}

 To bound $|A(m)|$, notice that for $n$ large enough, the $2+\delta$-moment bound derived in Proposition \ref{2-plus-delta-moment-bound} implies that:   
 \begin{align}
 \begin{split}
|A(m)| &= n^{-1-\delta/2 } \sum_{l=1}^m \mathbb{E}|S_l|^{2+\delta}
\\ &\leq  n^{-1-\delta/2 } \  m_n \ \left(C \sqrt{q_n}\right)^{2+\delta}
 \\ & \preceq m_n \left(\frac{q_n}{n} \right)^{1+\delta/2}
 \\ & \preceq n^{(a-1)\delta/2}
\end{split}
 \end{align}
We see that both terms converge to $0$ if $b \cdot \theta > 1/2$ and for any $0 < a < 1$.  We can always choose such $b$ since $\theta > 4$.  The weakly approaching claim follows.  Weak convergence under the additional assumption of $\frac{1}{n} \sum_{i=1}^n \text{Var}(X_i)$ existing follows directly from Lemma \ref{variance-convergence-lemma}.  
\end{proof}

Below we state a multivariate version of the non-stationary Central Limit Theorem under $\theta$-dependence.  
\begin{theorem}[Non-stationary Central Limit Theorem for $\theta$-dependent sequences, Multivariate Version] 
\label{multivariate-clt}
 Let $\{X_{n,i} \ 1 \leq i \leq n \}$ be a centered $\mathbb{R}^d$-valued triangular array satisfying the moment condition:
\begin{align}
\sup_{n,i} \max_{1 \leq j \leq d} \mathbb{E}|X_{n,ij}|^{2 + \zeta} \leq \mu
\end{align}   
for some $\zeta >0$. Suppose that $\theta_r(X_1^n) = O(r^{-\theta})$ as $n \rightarrow \infty$ where $\theta > 4 + \frac{2}{\zeta}$. Then,   
\begin{align}
\frac{1}{\sqrt{n}} \sum_{i=1}^n X_{n,i} {\overset{wa} \iff }  N(0, \Sigma_n) 
\end{align}   
where $\Sigma_n = \text{Var}(\frac{1}{\sqrt{n}} \sum_{i=1}^n X_{n,i})$. Furthermore, if $\lim_{n \rightarrow \infty} \frac{1}{n} \sum_{i=1}^n \text{Var}(X_{n,i})$ exists, then $ \Sigma = \lim_{n \rightarrow \infty} \Sigma_n$ exists, and we have that:
\begin{align}
\frac{1}{\sqrt{n}} \sum_{i=1}^n X_{n,i} \rightsquigarrow N(0, \Sigma)
\end{align} 
\end{theorem}
% Then, $\lim_{n \rightarrow \infty} \text{Var}(S_n^X) = \Sigma$ exists, with $\Sigma_{kl} = \lim_{n \rightarrow \infty} \frac{1}{n} \sum_{i=1}^n \sum_{j=1}^n \text{Cov}(X_{ik}, X_{jl})$ and: 
% \begin{align}
% \frac{1}{\sqrt{n}} \sum_{i=1}^n X_i \rightsquigarrow N(0, \Sigma) 
% \end{align}   
% \end{corollary}

% \begin{remark}
% Since limits of positive semidefinite matrices are positive semidefinite, $\Sigma$ is well-defined.  
% \end{remark}
\begin{proof}
By  the weakly approaching  Cram\`{e}r-Wold device (Proposition \ref{weakly-approaching-cramer-wold}), it is sufficient to  establish tightness of the sequence $\Sigma_n^{1/2} N$ and then show $\lambda^T S_n^X {\overset{wa} \iff } \lambda^T \Sigma_n^{\frac{1}{2}} N$ for all $\lambda \in \mathbb{R}^d$. Notice that $\Sigma_{n,ii} \leq c^\prime$ for all $i \in \{1, \ldots, d \}$; therefore the sequence is tight.  
 Now, we may express the $\lambda^T S_n^X$ as:
\begin{align}
\begin{split}
\frac{1}{\sqrt{n}} \sum_{j=1}^d \lambda_j \sum_{i=1}^n X_{ij} &= \frac{1}{\sqrt{n}} \sum_{i=1}^n \sum_{j =1}^d \lambda_j X_{ij} 
\\ & \equiv \frac{1}{\sqrt{n}} \sum_{i=1}^n h(X_i)
\end{split} 
\end{align}
where $||h||_L = \sum_{j=1}^d \lambda_j$.  Let $U_i = h(X_i)$.  Then, we have that $\theta_{r}(X_1^n) = O(r^{-\theta})$ and certainly $ \sup_{i \in \mathbb{N}} E|\lambda^T X_i|^{2+\zeta} \leq  d^{2+\zeta}\sup_{i \in \mathbb{N}} \sum_{j=1}^d |\lambda_j|^{2 + \zeta} E|X_{ij}|^{2+\zeta} < \infty$ and we may apply the univariate central limit theorem to $U_i$. Now, notice that:
\begin{align}
\begin{split}
\text{Var}(\lambda^T S_n^X) &= \frac{1}{n} \sum_{i =1 }^n \sum_{j= 1}^n \text{Cov}\left(\sum_{k=1}^d \lambda_k X_{ik}, \sum_{l=1}^d\lambda_l X_{jl} \right)
\\ &= \sum_{k,l} \lambda_k \lambda_l \left(\frac{1}{n} \sum_{i =1 }^n \sum_{j= 1}^n  \text{Cov}( X_{ik}, X_{jl})  \right) 
\end{split} 
\end{align}
% It may be verified that the covariance inequality from Proposition \ref{theta-dependent-covariance-lemma} may be adapted to multivariate setting as follows: 
% \begin{align}
% \left|\text{Cov}\left(X_{ik}, X_{(i+r)l}\right)\right| \leq 8 \ \mu^{\frac{1}{m-1}} \theta_X(r)^{\frac{m-2}{m-1}}
% \end{align}
% where $\sup_{i \in \mathbb{N}} \max_{1 \leq j \leq d} E |X_{ij}|^{2 + \zeta} \leq \mu$. Also, for any $i \in \mathbb{N}$, notice that:
% \begin{align}
% E|X_{ik} X_{il}| \leq || X_{ik}||_{2+\delta} \  || X_{il}||_{2+\delta} < \infty
% \end{align}
% which follows from the condition on the $2+\delta$ moments of $X_{ij}$. Now, we can see that: 
% \begin{align}
% c_{kl}^\prime = \sup_{i \in \mathbb{N}} E(|X_{ik} X_{il} |) +  2 \sum_{j \in \mathbb{N}} \sup_{i \in \mathbb{N}} |\text{Cov}( X_{ik}, X_{(i+j)l}) | < \infty
% \end{align}
% Now since $\lambda^T \Sigma_n \lambda \leq \sum_{k,l} |\lambda_k \lambda_l | c_{kl}^\prime < \infty$, it follows that $\{\lambda^T \Sigma_n^{1/2} N\}_{n \in \mathbb{N}}$ is a tight sequence of Normal random variables, and we may appeal to the univariate central limit theorem and conclude they are weakly approaching. 
Now, by definition, the variance of the normalized sum is equal to $\lambda^T \Sigma_n \lambda$; the first part of the claim follows.  Analogous to the univariate case, the second claim follows from Lemma \ref{variance-convergence-lemma}, where entry-wise Ces\`{a}ro convergence of the elements of the $\Sigma_n$ ensures convergence.  
\end{proof}

\bibliographystyle{apalike}
\bibliography{ms.bbl}

\end{document}